\documentclass{amsart}

\usepackage{amsmath}
\usepackage[foot]{amsaddr}

\usepackage{cancel}

\usepackage{amssymb}
\RequirePackage[numbers]{natbib}
\usepackage[dvips]{epsfig}
\usepackage{graphicx}
\usepackage{latexsym}
\usepackage{amsthm}
\usepackage{amssymb}
\usepackage{ulem}
\usepackage{cancel}
\RequirePackage[colorlinks=true,citecolor=blue,urlcolor=blue]{hyperref}

\usepackage{amsmath,amsthm,amsfonts,amssymb,amscd,amsbsy,dsfont,hyperref}

\usepackage{palatino}

\usepackage[margin=3cm]{geometry}

\usepackage{eucal}
\usepackage{float}
\usepackage{xcolor}

\newcommand{\E}{\mathbb E}
\newcommand{\R}{\mathbb R}
\newcommand{\PP}{\mathbb P}
\newcommand{\eps}{\varepsilon}
\renewcommand{\O}{\mathcal{O}}
\newcommand{\F}{\mathcal{F}}
\newcommand{\C}{\mathcal{C}}
\newcommand{\var}{\mathrm{Var}}
\newcommand{\mG}{\mathcal{G}}
\renewcommand{\H}{\mathbb{H}}
\newcommand{\supp}{\text{supp}}

\newcommand{\Z}{\mathbb{Z}}

\DeclareMathOperator*{\argmin}{arg\,min}

\newtheorem{thm}{Theorem}

\newtheorem{condition}[thm]{Condition}
\newtheorem{lem}[thm]{Lemma}

\newtheorem{cor}[thm]{Corollary}
\newtheorem{defi}[thm]{Definition}

\newtheorem{prop}[thm]{Proposition}
\newtheorem{rk}[thm]{Remark}

\newtheorem*{assumption*}{Assumption}

\title[High frequency multidimensional diffusions]{Nonparametric Bayesian estimation in a multidimensional diffusion model with high frequency data}

\author{Marc Hoffmann}
\email{hoffmann@ceremade.dauphine.fr}

\author{Kolyan Ray}
\email{kolyan.ray@imperial.ac.uk}

\begin{document}

\maketitle

\vspace{-0.5cm}

\begin{center}
CEREMADE, Universit\'e Paris Dauphine-PSL and Imperial College London

\end{center}

\begin{abstract}
We consider nonparametric Bayesian inference in a multidimensional diffusion model with reflecting boundary conditions based on discrete high-frequency observations. We prove a general posterior contraction rate theorem in $L^2$-loss, which is applied to Gaussian priors. The resulting posteriors, as well as their posterior means, are shown to converge to the ground truth at the minimax optimal rate over H\"older smoothness classes in any dimension. Of independent interest and as part of our proofs, we show that certain frequentist penalized least squares estimators are also minimax optimal.\\

\noindent \textit{MSC2020 subject classifications}: 62G20, 62F15, 60J60.

\noindent \textit{Keywords}: Bayesian nonparametrics, multidimensional diffusions, high-frequency data, Gaussian processes, penalized least squares estimator.
\end{abstract}

\tableofcontents

\section{Introduction}

Diffusion models are widely used in applications, including in the physical and biological sciences, economics and finance, as well as for particle filters and emulators amongst many others purposes. Denote by $u(t,x)$ the density of a quantity that is diffusing in a closed system at time $t>0$ and location $x\in \O$ in bounded convex domain $\O\subseteq \R^d$ with smooth boundary $\partial \O$. Fick's laws of diffusion state that the resulting dynamics are governed by the evolution equation
\begin{equation}\label{eq:heat_equation}
\frac{\partial u}{\partial t} = \nabla \cdot (f\nabla u),
\end{equation}
with Neumann boundary conditions and where $f>0$ is the diffusivity of the possibly inhomogeneous medium $\O$. This describes the macroscopic behaviour of many particles in $\O$ undergoing diffusion governed by Brownian dynamics.

On a microscopic level, the individual behaviour of a single such particle can be modelled as the multidimensional diffusion process $(X_t)_{t\geq 0}$ on $\O$ that is reflected upon hitting the boundary $\partial \mathcal O$ and arises as the solution to the SDE:
\begin{equation} \label{eq: tanaka}
\begin{split}
X_t & = X_0 + \int_0^t \nabla f(X_s)ds+ \int_0^t \sqrt{2f(X_s)}dB_s+ \ell_t, 
\\
& \ell_t  = \int_0^t n(X_s)d|\ell|_s, \qquad |\ell|_t = \int_0^t {\mathbf 1}_{\{X_s \in \partial \mathcal O\}} d|\ell|_s,
\end{split}
\end{equation}
where $(B_t)_{t \geq 0}$ is a standard $d$-dimensional Brownian motion, $(\ell_t)_{t \geq 0}$ a bounded variation process with $\ell_0=0$ that accounts for the boundary reflection and $n(x)$ denotes the unit inward normal to the boundary $\partial \mathcal O$ at $x$. We assume
$$f:\O \to [f_{\min},\infty)$$
is a positive function with $f_{\min} > 0$ and take $X_0$ to be uniformly distributed on $\mathcal O$, which is the invariant measure of the SDE \eqref{eq: tanaka}, and hence $X$ is started at stationarity. If $\partial \O$ and $\nabla f$ are continuously differentiable, existence and uniqueness of the solution to \eqref{eq: tanaka} follows from \cite{lions1984stochastic}. These conditions, together with the convexity of the domain $\mathcal O$, can be relaxed, but will be sufficient for the level of generality intended here.

Consider observations regularly collected at discrete time points, resulting in data of the form 
\begin{equation} \label{eq: data}
X^N = (X_0, X_{D}, \ldots, X_{ND})
\end{equation}
with sampling interval $D>0$. The statistical problem at hand is then to make inference on the diffusivity $f$ based on discrete observations of the location of a single particle corresponding to the macroscopic diffusion model \eqref{eq:heat_equation}. We study nonparametric Bayesian estimation of $f$ in model \eqref{eq: tanaka} using Gaussian process priors in the high-frequency and long term sampling regime, where $D = D_N \to 0$ such that the time-horizon $ND_N \to\infty$. In particular, we establish minimax optimal frequentist contraction rates for the posterior for $f$ about the `ground truth' $f_0$ generating the data, thereby establishing theoretical guarantees for this method.

A statistical feature of model \eqref{eq: tanaka} is that its generator $\mathcal L_f \phi = \nabla \cdot (f \nabla \phi) $ is a divergence form operator, which implies that the invariant measure $\mu_f$ is simply the uniform measure on the domain $\O$ for all $f$:
\begin{equation}\label{eq:invariant}
\mu_f(A) = \lim_{T\to\infty} \frac{1}{T}\int_0^T 1_A(X_t) dt \propto \text{vol}(A),
\end{equation}
for any measurable $A \subset \O$, so that the average asymptotic time spent by the process $X$ in different regions does not contain information about $f$. By Markovianity, all information about $f$ is thus contained in the transitions $X_{(i-1)D} \mapsto X_{iD}$, which motivates using a likelihood based approach, such as the Bayesian one we pursue here. In particular, this contrasts with several other well-studied diffusion models where one can identify the relevant model parameters from the invariant measure $\mu_f$, see e.g. \cite{DR07,S18,GR20,AWS22}. 

To better understand our approach, it is helpful to consider the one-dimensional diffusion model,
\begin{equation}\label{eq:diff_general}
dX_t' = b(X_t') dt + \sigma(X_t') dW_t,
\end{equation}
where $b:\R \to \R$ is the drift and $\sigma:\R \to \R \otimes \R^d$ is the diffusion coefficient. In the high frequency case where $b$ and $\sigma$ are both $s$-smooth, their minimax estimation rates scale like $(ND)^{-s/(2s+1)}$ and $N^{-s/(2s+1)}$, respectively, and one can thus estimate the diffusivity $\sigma$ at a faster rate than the drift $b$ \cite{H99}. In model \eqref{eq: tanaka}, where these parameters are coupled, the diffusivity term is similarly more informative about $f$ than the drift, which must be exploited to obtain minimax optimal rates.
While the Bayesian methodology does not explicitly make this distinction, simply placing a prior on $f$ as usual, our proofs themselves heavily rely on this idea. One can therefore view rate-optimal estimation of $f$ in model \eqref{eq: tanaka} as qualitatively analogous to estimation of $\sigma$ in model \eqref{eq:diff_general}. Our work is thus relevant to the literature on estimating the diffusion coefficient \cite{H99, CGCR07, schmisser2012non} under high-frequency sampling. Our results indicate that Bayesian methods can correctly pick up this feature of the data via the likelihood.\\

We prove a general contraction rate theorem for $f$ under approximation-theoretic conditions on the prior following the classic testing approach of Bayesian nonparametrics \cite{GV17}, which requires:
\begin{itemize}
\item[(i)] the integrated log-likelihood process is not too small with high probability;
\item[(ii)] the existence of suitable tests with exponentially decaying type-II errors.
\end{itemize}
Since the present high-frequency setting involves increasingly correlated observations, establishing (i) is much more involved than in the i.i.d. setting, where the log-likelihood tensorizes \cite{GGV00}. In particular, we must deal with the full dependent log-likelihood ratio process as a whole. Following the intuition from the continuous observation setting \cite{VMVVVZ06,NR20,GR20}, we employ martingale techniques to study an approximation of the log-likelihood which allows to deal with the dependence structure of the Markov chain. Making this approximation both precise and uniform is a key challenge, and our proof relies on refined small time expansions of the transition density or heat kernel via Riemannian geometry \cite{azencott1984densite,berline2003heat,B20}.

Turning to (ii), we construct suitable tests by establishing concentration inequalities for frequentist estimators \cite{GN11}. Given the above discussion regarding model \eqref{eq:diff_general}, our estimators must draw information from the diffusivity term to be minimax optimal. For this task, we follow the ideas of Comte et al. \cite{CGCR07} in the scalar case of using penalized least squares estimators. 
As a by-product, we obtain frequentist estimators that we prove to be optimal in a minimax sense. To the best of our knowledge, these are the first minimax results for estimating the diffusion coefficient in a multidimensional setting.
We finally apply our general contraction rate theorem to concrete Gaussian priors, such as the Mat\'ern process.
Gaussian priors are widely used in practice \cite{RW06} and have been applied in multiple ways to diffusion models, see for instance \cite{PSVZ13,RBO13,BRO18}. While we do not focus on computational issues here, note that posterior sampling based on discrete data is possible using Euler–Maruyama schemes \cite{KP92} or other methods \cite{BPRF06,PPRS12,BFS16,vdMS17,vdMS17b}.

Regarding related works, theoretical properties of Bayesian nonparametric methods in the scalar ($d=1$) diffusion model \eqref{eq:diff_general} have been well-studied, see for instance \cite{VMVVVZ06,PSVZ13,VMVZ13,VWVZ16,NS17,A18}. In the multivariate setting, contraction rates were recently obtained under continuous observations \cite{NR20,GR20} and only consistency for discrete observations \cite{GS14,N22b}. While the continuous observation model provides some high-level intuition regarding the use of martingales techniques in our proofs, we note it is fundamentally unsuited as a model for statistical estimation of diffusion coefficients, see Section 3.3 in \cite{GR20} for further discussion. For discrete observations, \cite{N22b} extends the one-dimensional results of \cite{GHR04,NS17} to establishes posterior consistency in the same multidimensional model \eqref{eq: tanaka} in the low-frequency setting using PDE and spectral techniques, showing that Bayesian methods can in principle adapt to the sampling regime, see \cite{A18} for further discussion. In the present high-frequency setup, a spectral approach will not give sharp rates and we instead use refined tools from stochastic analysis (in turn unsuited to the low-frequency setting) to obtain minimax optimal contraction rates.

Inference for the diffusivity $f$ has also been studied in other contexts, notably in the inverse problems literature where one often studies a time-independent version of the PDE \eqref{eq:heat_equation} corresponding to a `steady state' measurement of diffusion, see Section 1.1.2 of \cite{N22}. One can then consider an idealized statistical model where one observes the solution to this PDE under additive noise that does not typically depend on $f$ \cite{S10,N22}. As well as being a simplified observation model, this will typically yield ill-posed convergence rates in terms of the number of observations $N$. In contrast, we show here that in the high-frequency Markovian setting, one recovers the `usual' (and faster) nonparametric convergence rate $N^{-s/(2s+d)}$ for an $s$-smooth $f$. As discussed above, this is possible because our methods exploit information about $f$ contained in the noise structure of \eqref{eq: tanaka} that is absent in these simplified models. While not surprising from a high-frequency diffusion perspective, this does highlight how studying simplified noise models in inverse problems can lead to qualitatively different statistical behaviour, and that embedding unknown parameters into the noise model need not make the problem more difficult, and can even do the opposite.

\section{Main results}

\subsection{Preliminary setup for a reflected diffusion model in $\R^d$}\label{sec:setup}

We now rigorously set up the statistical framework for our results in model \eqref{eq: tanaka}.
To account for the boundary behaviour, we assume that $f$ is known to be 1 in an open neighbourhood of $\partial \O$. More precisely, for $\mathcal K\subset \O $ a known compact set with $\mathsf{dist}(\mathcal K, \partial \O) >0$, define the parameter space
\begin{equation}\label{eq:F0}
\F_0 = \F_0(\mathcal K,d,f_{\min}) =  \left\{ f \in \C^{\alpha} (\O): \inf_{x\in \O} f(x) \geq 2f_{\min} ~ \text{and} ~ f(x) = 1 \text{ for all }x \in\mathcal O\setminus \mathcal K \right\},
\end{equation}
where $0<f_{\min}<1/2$  and the minimal smoothness equals
\begin{equation}\label{eq:alpha}
\alpha = \alpha_d = \max \left( 4 , 2 \left\lfloor d/4+1/2 \right\rfloor \right).
\end{equation}
The minimal value $\alpha_d$ scales like $d/2$ for large dimension $d$ and is needed to employ suitable bounds on the transition densities of the diffusion \cite{C03}, see Section \ref{sec:F_class} for further discussion about the parameter space $\F_0$. In particular, since $\alpha \geq 4$, this implies existence and uniqueness of the solution to \eqref{eq: tanaka}.

We consider the statistical experiment generated by the discrete observations
$$X^N = (X_0, X_{D}, \ldots, X_{ND})$$
from \eqref{eq: tanaka} with sampling rate $D_N = D>0$. We work in the high-frequency and long-term sampling regime as stated in the following assumption.

\begin{assumption*}[Sampling regime]
Suppose that $D = D_N \to 0$ such that $ND \to \infty$, but $ND^2 \to 0$ as $N\to\infty$.
\end{assumption*}
\noindent This is the minimal assumption for all our results and will be assumed throughout the paper without further mention.

Let $\PP_f$ (simply $\PP$ when no confusion may arise) denote the unique law of the Markov process $(X_t)_{t \geq 0}$ arising from \eqref{eq: tanaka}. Consider the corresponding infinitesimal generator
\begin{equation} \label{eq: def generator}
\mathcal L_f(\phi) = \nabla \cdot (f \nabla \phi) = \nabla f \cdot \nabla \phi + f \Delta \phi = \sum_{j = 1}^d \frac{\partial}{\partial x_j}\Big(f \frac{\partial}{\partial x_j}\phi\Big),
\end{equation} 
densely defined on continuous functions. For such a divergence form operator, the unique invariant probability measure can be explicitly computed as the uniform measure on $\mathcal{O}$ as in \eqref{eq:invariant},
see p. 46 of \cite{BGL14}, so that the process $(X_t)_{t \geq 0}$ in \eqref{eq: tanaka} is stationary. Without loss of generality, we will assume that $\mathrm{vol}(\mathcal O)=1$. Let $(P_{f,t})_{t \geq 0}$ denote the family of transition operators associated to $(X_t)_{t \geq 0}$, each of which admits a transition density $p_{f,t}(x,y)$ on $\mathcal O \times \mathcal O$, so that for every real-valued and bounded function $\phi:\O \to \R$,
\begin{equation*}\label{eq:transition}
P_{f,t}\phi(x) = \E_{f}\big[\phi(X_t)\,|\,X_0=x\big] = \int_{\mathcal O}\phi(y)p_{f,t}(x,y)dy.
\end{equation*}
In particular $p_{f,t}$ arises as the fundamental solution to the heat equation \eqref{eq:heat_equation} with Neumann boundary conditions. 
It can be shown that there exists $\lambda_f >0$ such that $\|P_{f,t} \phi \|_{2} \leq \mathrm{e}^{-\lambda_f t}\|\phi \|_{2}$
for every $\phi: \mathcal O \rightarrow \R$ such that $\int_{\mathcal O}\phi d\mu_f = \int_{\mathcal{O}} \phi dx = 0$, where $\lambda_f$ satisfies $\lambda_f \geq f_{\min} /p_\O >0$ with $p_\O>0$ the Poincar\'e constant for the domain $\O$, see for example Section 3.1 of \cite{N22b}.  This implies that the transition operator $P_{f,D}$ of the discrete time Markov chain $X_0,X_D,X_{2D},\dots$ has first non-trivial eigenvalue
\begin{equation}\label{eq:spectral_gap}
1-e^{-D\lambda_f} \geq rD >0,
\end{equation}
for some $r=r(f_{\min},p_\O)>0$, namely the Markov chain has a spectral gap decreasing linearly in the step size $D\to 0$. 

From the above, we obtain the likelihood
$$e^{\ell_N(f)} = e^{\ell(f;X_0,\ldots, X_{ND})} =  \prod_{i = 1}^N  p_{f,D}(X_{(i-1)D}, X_{iD})$$
based on observations $X^N = (X_0,X_D,\dots,X_{ND})$. We consider a Bayesian approach by placing on $f$ a possibly $N$-dependent prior $\Pi= \Pi_N$, supported on some set $\widetilde{\mathcal{F}} \subseteq \{f\in \C^2(\O): \inf_{x\in \O} f(x) \geq f_{\min}\}$, leading to the posterior distribution
\begin{align*}\label{eq:posterior}
\Pi(A|X_0,X_D,\dots,X_{ND}) &= \frac{\int_{A} \prod_{i = 1}^N p_{f,D}(X_{(i-1)D}, X_{iD}) d\Pi(f) }{\int_{\widetilde \F} \prod_{i = 1}^N p_{f,D}(X_{(i-1)D}, X_{iD}) d\Pi(f)} , \qquad  A\subseteq \widetilde{\F}~ \text{  measurable}.
\end{align*}
In the following, we will study frequentist contraction rates for the posterior for $f$ about the `ground truth' $f_0$ assumed to generate the data $X^N$ in \eqref{eq: tanaka}.

\textbf{Additional notation and function spaces.} Let $|\cdot|$ denote the usual Euclidean norm on $\R^d$. For a multi-index $j =(j_1,\dots,j_d) \in \mathbb{N}^d$, set $|j| = j_1+\dots+j_d$ and consider the resulting partial differential operator $\partial^j  = \frac{\partial^{|j|}}{\partial_1^{j_1}\ldots \partial_d^{j_d}}$.
For integer $k \geq 0$, let $\C^k(\O)$ denote the space of $k$-times differentiable functions on $\O$ with uniformly continuous derivatives, while $\C(\O) = \C^0(\O)$ denotes the space of continuous function equipped with the supremum norm $\|\cdot\|_\infty$. For non-integer $\beta>0$, set
$$\C^\beta (\O) = \left\{ f \in C^{\lfloor \beta\rfloor}(\O): \sup_{x,y\in \O: x\neq y} \frac{|\partial^j f(x) - \partial^j f(y)|}{|x-y|^{\beta - \lfloor\beta\rfloor}}<\infty \quad \text{for all } |j|=\lfloor \beta\rfloor \right\},$$
where $\lfloor \beta \rfloor$ denotes the largest integer less than or equal to $\beta$. For $\beta\geq 0$, $\C^\beta(\O)$ is equipped with the usual norm
$$\|f\|_{\C^\beta} = \|f\|_{\mathcal C^\beta(\mathcal O)} = \sum_{|j|=0}^{\lfloor \beta \rfloor} \sup_{x \in \mathcal O}|\partial^j f(x)| + \sum_{|j| = \lfloor \beta \rfloor} \sup_{x,y\in \O:x\neq y} \frac{|\partial^j f(x) - \partial^j f(y)|}{|x-y|^{\beta - \lfloor\beta\rfloor}},$$
where the second sum is removed for integer $\beta$. In particular, $\|f\|_\infty =  \|f\|_{\mathcal C^0}$, and the norms $\|\cdot\|_{\mathcal C^\beta}$ are non-decreasing in $\beta$. 

Write $\langle \cdot , \cdot \rangle_2$ for the usual inner product for $L^2 = L^2(\O)$.
For integer $s\geq 0$, define the Sobolev space on $\O$ by
$$H^s(\O) = \left\{ f \in L^2(\O): \partial^j f \text{ exists weakly and } \partial^j f \in L^2(\O) \text{ for all } |j|\leq s \right\} $$
equipped with the inner product
$$\langle f,g\rangle_{H^s(\O)} = \sum_{|j|\leq s} \langle \partial^j f, \partial^j g\rangle_{2}.$$
For non-integer $s\geq 0$, one defines $H^s(\O)$ by interpolation \cite{LM73,T95}. When no confusion may arise, we will often drop the explicit reference to the domain $\O$ in the notation.

We will repeatedly use positive quantities that do  depend on some parameters of the model, and that we informally call constants, and that may vary from line to line.
These never depend on $N$, but do usually depend on the dimension $d$ of the ambient space. Other relevant dependences may be emphasised by a subscript, such as $C_f$, $C_{\|f\|_{\mathcal C^k}}$ or $C_{f,f_0}$. The notation $A_N \lesssim B_N$ means $A_N\leq CB_N$ for every $N \geq 1$, where $C$ is a ``constant" according to our informal terminology. For real numbers $a,b$, let $a \wedge b$ and $a\vee b$ denote the minimum and maximum of $a$ and $b$, respectively. We also sometimes write $\PP(A,B) = \PP(A \cap B)$ to shorten notation.

\subsection{Gaussian process priors}\label{sec:GP}

Gaussian priors are widely used in diffusion models and we investigate here their theoretical frequentist convergence rates. We assign to $f$ a prior based on a Gaussian process, which must be modified to account for the constrained parameter space $\F_0$ in \eqref{eq:F0}, in particular that $f\geq 2f_{\min}$ and $f=1$ near the boundary. We therefore employ an exponential link function $\Phi:\R \to (f_{\min},\infty)$, similar to those used in density estimation \cite{GV17} or inverse problems \cite{S10}:
\begin{equation}\label{eq:link}
\Phi(x) = f_{\min} + (1-f_{\min}) e^x,
\end{equation}
which is strictly increasing on $\R$ with $\Phi(0) = 1$.

We now detail our full prior construction, starting from an underlying Gaussian process $V$. For definitions and background material on Gaussian processes and their associated RKHS, see Chapter 11 of \cite{GV17} or \cite{RW06}. 

\begin{condition}\label{cond:RKHS}
Let $\Pi_V = \Pi_{V,N}$ be a mean-zero Gaussian Borel probability measure on the Banach space $\C(\O)$ that is supported on a separable measurable linear subspace of $\C^4 (\O)$, and assume that its reproducing kernel Hilbert space (RKHS) $(\H_V,\|\cdot\|_{\H_V})$ embeds continuously into the Sobolev space $H^s(\O)$ for some $s\geq 4$, i.e. $\H_V \hookrightarrow H^s(\O)$. 
\end{condition}

Recall that any $f_0 \in \F_0$ satisfies $\{x\in \O: f_0(x) \neq 1\} \subseteq \mathcal K$ by \eqref{eq:F0}, where $\mathcal K\subset \O $ is a known compact set with $\mathsf{dist}(\mathcal K, \partial \O) >0$. There thus exists $\delta>0$ and an open set $\O_0$ having smooth $\mathcal C^\infty$-boundary $\partial \O_0$ such that $\mathcal K \subsetneq \O_0 \subsetneq \O$ and
\begin{equation}\label{eq:delta}
\mathsf{dist}(\mathcal K, \partial \O_0) \geq \delta \qquad  \text{and} \qquad  \mathsf{dist}(\O_0, \partial \O) \geq \delta,
\end{equation}
see, e.g., the proof of Proposition 8.2.1 in \cite{D08}. A simple example showing the relationship between such sets is plotted in Figure \ref{fig:domain}.

\begin{figure}
\includegraphics[width=0.6\textwidth]{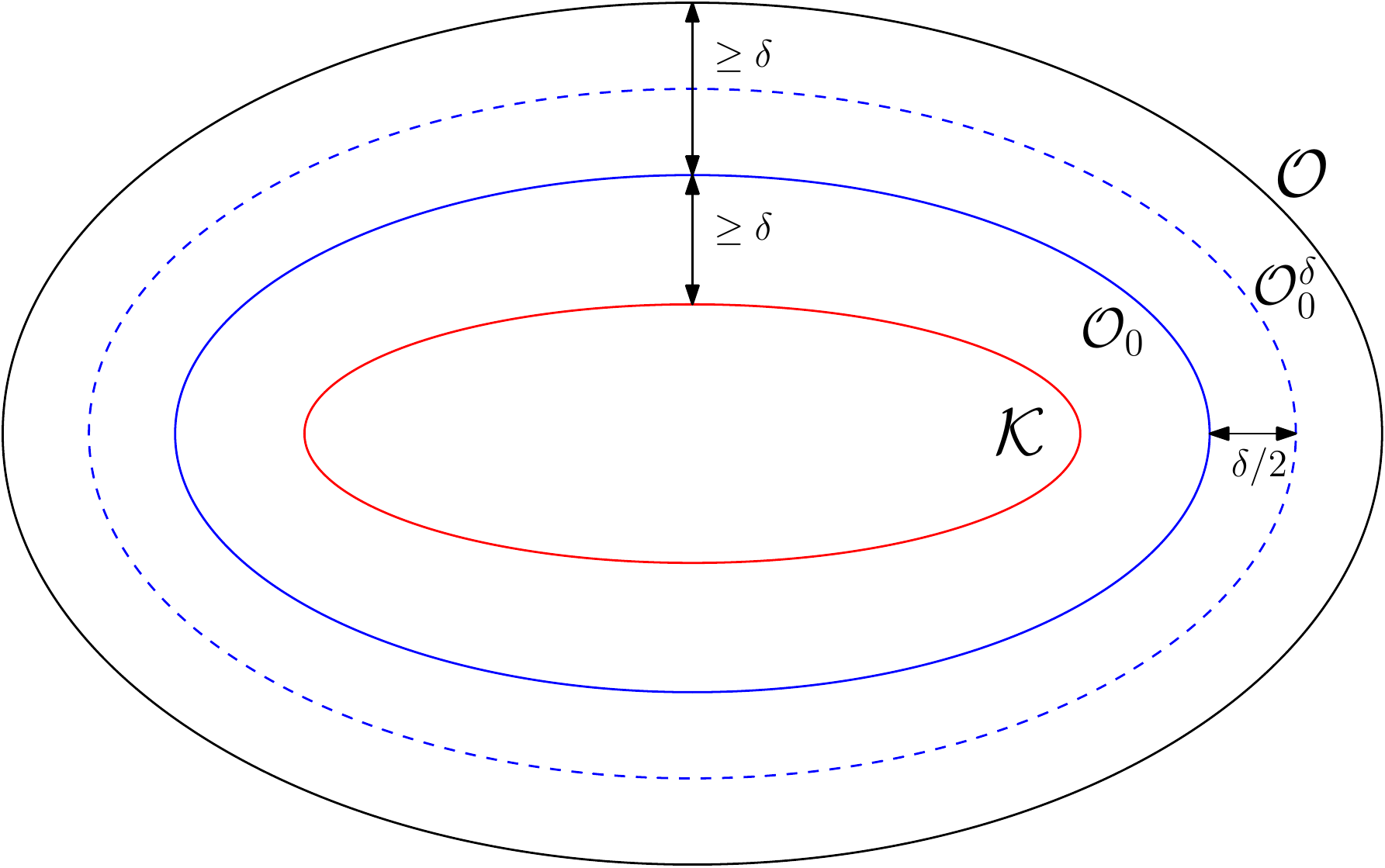}
\caption{A simple example of a suitable domain $\O$ (black), the region $\mathcal{K}$ (red) where the ground truth $f_0 \in \mathcal{F}_0$ is not equal to 1, and a set $\mathcal{O}_0$ (blue) that $\delta$-separates these sets as in \eqref{eq:delta}. The $\delta/2$-enlargement $\O_0^\delta$ (dashed blue) of $\O_0$ given in \eqref{eq: def offset} is also plotted.}\label{fig:domain}
\end{figure}

Let $\chi\in \C^\infty(\O)$ be a smooth bounded cutoff function such that $\chi \equiv 1$ on $\mathcal K$ and $\chi \equiv 0$ outside $\O_0$. We take as prior distribution $\Pi = \Pi_N$ for $f$ the law of the random function
\begin{equation}\label{eq:prior}
f(x) = \Phi ( \chi(x) W(x)), \qquad \qquad W(x) = \frac{V(x)}{N^{d/(4s+2d)}}.
\end{equation}
It follows that $W$ is again a mean-zero Gaussian process with the same support and RKHS as $V$, but with rescaled covariance function. The cutoff function $\chi$ ensures that $f$ transitions smoothly to take value 1 near the domain boundary. Since the sample paths of $V$ are in $\C^4(\O)$ almost surely under the prior by Condition \ref{cond:RKHS}, then $\Pi(\F)=1$ for
\begin{equation}\label{eq:F}
\mathcal F = \left\{f \in \mathcal \C^4(\mathcal O): \inf_{x\in \O} f(x) \geq f_{\min} ~ \text{and} ~ f(x) = 1 \text{ for all }x \in\mathcal O\setminus \mathcal O_0 \right\}.
\end{equation}
Note that $\F_0 \subsetneq \F$ since $\mathcal K \subsetneq \O_0$, so that the prior lives on a slightly larger parameter space than $\F_0$, see Section \ref{sec:F_class} for further discussion. Moreover, since $\Phi^{-1}(y) = \log \frac{y-f_{\min}}{1-f_{\min}}$ has uniformly bounded derivatives of all orders on $(2f_{\min},\infty)$, we have that for any $f_0 \in \F_0 \cap \C^s(\O)$, the function $w_0 = \Phi^{-1}(f_0)$ is also in $\C^s(\O)$ with $\supp(w_0) \subseteq \mathcal K$.

The following is the main result of our paper, establishing contraction rates for these rescaled Gaussian priors.

\begin{thm}\label{thm:GP}
Consider the sampling interval $D = N^{-a}$ for some $a\in(1/2,1)$ and suppose $s>s_{d,a}^*$, where
\begin{equation}\label{eq:s_star}
s_{d,a}^* =  \max \left( 4+\frac{d}{2},  \frac{2-ad}{2a-1}, \frac{d(1+a)}{2(1-a)}  \right).
\end{equation}
Let $\Pi = \Pi_N$ denote the prior \eqref{eq:prior} with mean-zero Gaussian process $V \sim \Pi_V$ having RKHS $\H_V$ and satisfying Condition \ref{cond:RKHS} for this $s$. Let $f_0 \in \F_0 \cap \C^s(\O)$, set $w_0 = \Phi^{-1}(f_0)$ and suppose that there exists a sequence of functions $v_{0,N}\in \H_V$ in the RKHS of $V$ such that
\begin{itemize}
\item[(i)] $\|v_{0,N}\|_{\H_V} = O(1)$,
\item[(ii)] $\|\chi v_{0,N}\|_{\C^{\alpha}} = O(1)$ for $\alpha = \max \left( 4 , 2 \left\lfloor d/4+1/2 \right\rfloor \right)$,
\item[(iii)] $\|w_0 - \chi v_{0,N}\|_{\C^k}  = O(N^{-\frac{s-k}{2s+d}})$ for $k=0,1,2,3$,
\end{itemize}
as $N\to\infty$. Then for $M>0$ large enough, as $N \to \infty$,
$$\E_{f_0} \Pi(f: \|f-f_0\|_2 \geq M N^{-\frac{s}{2s+d}} | X_0,X_D,\dots,X_{ND}) \to 0.$$
\end{thm}

Since $N^{-\frac{s}{2s+d}}$ is the minimax estimation rate in the high-frequency sampling regime (see Theorem \ref{thm:lower_bd} below for the corresponding lower bound), this result says that the posterior contracts about the truth at the minimax-optimal rate in any dimension $d$. Given the invariant measure is uninformative here, being the uniform distribution on $\mathcal{O}$ for all $f\in \F_0$, this confirms that the Bayes method can indeed perform optimal inference by picking up sufficient information from the transitions $X_{(i-1)D} \mapsto X_{iD}$ via the likelihood.

The minimal smoothness condition $s_{d,a}^*$ in \eqref{eq:s_star} can be rewritten as
\begin{equation*}
s_{d,a}^* =  \begin{cases}
\max \left( 4+\tfrac{d}{2}, \tfrac{2-ad}{2a-1}, \tfrac{d(1+a)}{2(1-a)}  \right) \qquad & \text{ if $d=1,2$ and $a\in(1/2,1)$, or $d=3$ and $a\in(1/2,2/3)$},\\
\max \left( 4+\tfrac{d}{2},  \tfrac{d(1+a)}{2(1-a)}  \right) & \text{ if $d=3$ and $a\in[2/3,1)$},\\
\tfrac{d(1+a)}{2(1-a)}  & \text{ if $d\geq 4$ and $a\in(1/2,1)$}.
\end{cases}
\end{equation*}
This becomes more stringent as $a \to 1$, namely the frequency $D = N^{-a} \to N^{-1}$ increases, since the time horizon $ND = N^{1-a}$ then grows more slowly. The term $\tfrac{d(1+a)}{2(1-a)}$ should be thought of as the main condition, with the extra terms for $d = 1,2,3$ coming from minimal smoothness assumptions needed to use various expansions and bounds.

Theorem \ref{thm:GP} requires that the true $w_0 = \Phi^{-1}(f_0)$ can be very well approximated by elements $v_{0,N}$ of the RKHS $\mathbb{H}_V$ of $V$. Under Condition \ref{cond:RKHS}, an $s$-smooth truth almost lies in $\mathbb{H}_V$ as needed, but this in turn implies that the Gaussian process $V$ will typically have $(s-d/2)$-smooth sample paths (e.g. Proposition I.4 of \cite{GV17} for the Mat\'ern process), thereby undersmoothing the truth. When the prior mismatches the true smoothness, suitably rescaling the prior as in \eqref{eq:prior} has been shown to still yield optimal rates in several benchmark statistical models \cite{VVVZ07,KVVVZ11}. This rescaling has also recently been used in the Bayesian inverse problems literature, where it is typically used to control stability estimates \cite{GN20,N22}.

We now consider two examples of Gaussian process priors satisfying the assumptions of Theorem \ref{thm:GP}. Let $V = \{V(x):x\in \O\}$ denote a Mat\'ern process on $\O$ with regularity parameter $s-d/2>0$, that is $V$ is a mean-zero stationary Gaussian process with covariance function
$$K(x,y) = K(x-y) = \int_{\R^d} e^{-i(x-y).\xi} (1+|\xi|^2)^{-s} d\xi, \qquad x,y\in \O.$$
The covariance function can alternatively be represented in terms of special functions, see e.g. p.84 of \cite{RW06}. While the Mat\'ern process models (almost) $(s-d/2)$-smooth functions in a H\"older sense, its rescaled version $W = V/N^{d/(4s+2d)}$ in \eqref{eq:prior} concentrates on $s$-smooth functions.

\begin{cor}\label{cor:Matern}
Let $D= N^{-a}$ for some $a\in(1/2,1)$, and let $\Pi = \Pi_N$ denote the prior \eqref{eq:prior} with $V$ a Mat\'ern process of regularity $s-d/2$ with $s>s_{d,a}^*$ for $s_{d,a}^*$ as in \eqref{eq:s_star}. If $f_0 \in \F_0 \cap \C^s(\O)$, then for $M>0$ large enough, 
$$\E_{f_0} \Pi(f: \|f-f_0\|_2 \geq M N^{-\frac{s}{2s+d}} | X_0,X_D,\dots,X_{ND}) \to 0 \qquad  \text{as }N\to\infty.$$
\end{cor}

As a second example, we consider a truncated Gaussian wavelet series prior. Let $\{\psi_{lr}: l\geq J_0,r \in \Z^d\}$ denote an orthonormal basis of $L^2(\R^d)$ composed of sufficiently regular, compactly supported Daubechies wavelets, where $J_0 \in \mathbb{N}$ is the base resolution level of the scaling functions, which we also denote by $\{\psi_{J_0,r}\}$ to simplify notation (see Chapter 4 of \cite{GN16} for details). Let $R_l$ denote the set of indices $r$ for which the support of $\psi_{lr}$ intersects $\O_0$. For $J_0$ large enough (see Section \ref{sec:general_contraction}), consider the Gaussian series expansion
\begin{equation}\label{eq:gauss_wavelet}
V(x) = \sum_{l=J_0}^J \sum_{r\in R_l} 2^{-ls} g_{lr} \psi_{lr}(x), \qquad g_{lr} \sim^{iid} N(0,1),
\end{equation}
where $2^J \simeq N^\frac{1}{2s+d}$. Note we could extend the second index set from $r \in R_l$ to $r \in \mathbb{Z}^d$ to cover all of $\R^d$ since wavelet functions $\psi_{lr}$ supported outside $\O_0$ ultimately play no role in the prior \eqref{eq:prior} due to the cuffoff function $\chi$.

\begin{cor}\label{cor:gauss_wav}
Let $D= N^{-a}$ for some $a\in(1/2,1)$, and let $\Pi = \Pi_N$ denote the prior \eqref{eq:prior} with $V$ a Gaussian wavelet series as in \eqref{eq:gauss_wavelet} with $s>s_{d,a}^*$ for $s_{d,a}^*$ as in \eqref{eq:s_star}. If $f_0 \in \F_0 \cap \C^s(\O)$, then for $M>0$ large enough, 
$$\E_{f_0} \Pi(f: \|f-f_0\|_2 \geq M N^{-\frac{s}{2s+d}} | X_0,X_D,\dots,X_{ND}) \to 0 \qquad  \text{as }N\to\infty.$$
\end{cor}

Using a uniform integrability argument (e.g. Theorem 2.3.2 of \cite{N22}), this implies the same convergence rate for the posterior mean.

\begin{cor}
Let $\bar{f}_N = E^\Pi[f|X_0,X_D,\dots,X_{ND}]$ denote the posterior mean based on the Mat\'ern process or Gaussian wavelet series prior. Under the conditions of Corollary \ref{cor:Matern} (Mat\'ern) or Corollary \ref{cor:gauss_wav} (Gaussian wavelet series), there exists a constant $M>0$ such that as $N\to\infty$,
$$\PP_{f_0} (\|\bar{f}_N-f_0\|_2 \geq M N^{-\frac{s}{2s+d}} ) \to 0.$$
\end{cor}

We thus have two concrete examples of Gaussian priors for which the posterior (means) converge to the true $f_0 \in \F_0 \cap \C^s$ at rate $N^{-\frac{s}{2s+d}}$. The following lower bound shows that this is indeed the minimax rate of convergence.

\begin{thm}[Minimax lower bound]\label{thm:lower_bd} Let $D= N^{-a}$ for some $a\in(1/2,1)$. For any $s > \alpha_d$ (the minimal smoothness \eqref{eq:alpha}) and $M>0$, the following frequentist lower bound holds:
$$\liminf_{N\to\infty} \inf_{\widehat f_N}\sup_{f \in \mathcal F_0, \|f\|_{\mathcal C^s} \leq M} N^{\frac{2s}{2s+d}} \E_f \big[\|\widehat f_N-f\|_{2}^2\big] >0,$$
where the infimum is taken among all estimators based on data \eqref{eq: data}.
\end{thm}
In particular, we extend the univariate minimax lower bound of \cite{H99}. The proof is given in Appendix \ref{sec: proof minimax lower bound}.

\section{A general contraction theorem}\label{sec:general_contraction}

We now state and discuss the abstract contraction rate theorem used to derive the convergence results for Gaussian process priors in Section \ref{sec:GP} above. This is based on the general testing approach of Bayesian nonparametrics \cite{GV17}, which requires (i) that the prior puts sufficient mass on a neighbourhood of the truth and (ii) the existence of suitable tests with exponentially decaying type-II errors.

For $0 < \varepsilon_N \leq \varepsilon_{1,N}\leq \eps_{2,N} \leq \eps_{3,N} \to 0$ positive sequences with $N\eps_N^2 \to \infty$, $r>0$ and $\alpha = \alpha_d$ as in \eqref{eq:alpha}, define the following neighbourhood of $f_0$:
\begin{align}\label{eq:Cn}
\C_N & = \C_N(f_0,\varepsilon_N, \varepsilon_{1,N},\eps_{2,N},\eps_{3,N}, r) \nonumber \\
& =  \Big\{f \in \mathcal F: \|f\|_{\C^{\alpha}} \leq r, \|f-f_0\|_{\infty} \leq \varepsilon_N,  \|f-f_0\|_{\mathcal C^k} \leq \varepsilon_{k,N} ~ \text{for } k=1,2,3 \Big\},
\end{align}
which can be related to the information theoretic distance of the model induced by the log-likelihood ratio process.
Further define 
\begin{equation}\label{eq:EnVn}
\begin{split}
E_N & := \eps_N^2 \Big( 1  +  \tfrac{\eps_{2,N}}{\eps_N^2} D + \tfrac{\eps_{3,N}}{\eps_N^2} D^{3/2} \Big) \\
V_N &:= N\eps_N^2 + N \eps_N^4 D^{-1}  + N^2 \eps_{2,N}^2D^2 + N^2 \eps_{3,N}^2 D^3 +N^2D^4.
\end{split}
\end{equation}
The quantities $NE_N$ and $V_N$ control the expectation and variance, respectively, of the integrated log-likelihood ratio process restricted to the small-ball $\C_N$, which is used to lower bound the normalized denominator of the Bayes formula. This is the content of the next result, which is the most technically involved part of our proof.

\begin{thm}[Evidence lower bound] \label{thm:variance}
Let $f_0 \in \F_0$ and $\nu$ be a probability measure supported on $\C_N$ as in \eqref{eq:Cn} with $r>0$ and sequences $0 < \varepsilon_N \leq \varepsilon_{1,N}\leq \eps_{2,N} \leq \eps_{3,N} \to 0$ satisfying $N\eps_N^2 \to \infty$ as $N\to\infty$. Then the integrated log-likelihood ratio process
$$\Lambda_{D}^N = \int_{\mathcal C_N} \log \prod_{i = 1}^N \frac{p_{f,D}(X_{(i-1)D}, X_{iD})}{p_{f_0,D}(X_{(i-1)D}, X_{iD})}\nu(df)$$
satisfies
\begin{align*}
\sup_{f\in \C_N} \E_{f_0} \left[ \log \frac{p_{f_0,D}(X_{0}, X_{D})}{p_{f,D}(X_{0}, X_{D})} \right] & \lesssim  E_N, \qquad \qquad \mathrm{Var}_{f_0} (\Lambda_{D}^N) \lesssim  V_N,
\end{align*}
where $E_N$, $V_N$ are defined in \eqref{eq:EnVn} and the constants depend only on $f_0,r,d,\O,\delta,f_{\min}$. This implies that for every $c>0$,
\begin{align}\label{eq:elbo}
\PP_{f_0} \left( \int_{\C_N} \prod_{i = 1}^N \frac{p_{f,D}(X_{(i-1)D}, X_{iD})}{p_{f_0,D}(X_{(i-1)D}, X_{iD})}\nu(df) \leq e^{-cN\eps_N^2-C_0N E_N} \right) \leq \frac{V_N}{c^2 N^2 \eps_N^4},
\end{align}
where $C_0>0$ is a fixed constant depending only on $f_0,r,d,\O,\delta,f_{\min}$.
\end{thm}

Theorem \ref{thm:variance} shows that the Bayesian evidence is at least $e^{-CN\eps_N^2}$ with $\PP_{f_0}$-probability tending to one if 
\begin{equation}\label{eq:main_conditions}
E_N \lesssim \eps_N^2 \qquad \text{and} \qquad V_N /(N^2 \eps_N^4) \to 0.
\end{equation}
An overview of the proof of Theorem \ref{thm:variance} is found in Section \ref{sec:small_ball}.

\begin{rk}\label{rk:ex}
For $a\in(1/2,1)$, consider the `usual' nonparametric choices of sequences for a $\C^s$-smooth truth:
$$D = N^{-a}, \quad \eps_N = N^{-\frac{s}{2s+d}}, \quad  \eps_{k,N} = N^{-\frac{s-k}{2s+d}}, \quad k=1,2,3.$$
Then \eqref{eq:main_conditions} is implied by
$$s > \begin{cases}
\frac{2-ad}{2a-1} \quad & \text{ if } d=1,2,~a\in(1/2,1), \text{ or } d=3,~a\in(1/2,2/3],\\
0 & \text{ if } d=3,~a\in[2/3,1), \text{ or } d\geq4, ~a\in(1/2,1).
\end{cases}$$
\end{rk}

For comparison, consider the i.i.d model $Y_1,\dots,Y_N \sim^{iid} P$ with $P$ having density $p$, so that $(Y_1,\dots,Y_N) \sim P^N = \otimes_{i=1}^N P$. The corresponding neighbourhood can then be expressed in terms of the Kullback-Leibler divergence and its $2^{nd}$-variation \cite{GGV00},
\begin{align*}
B_2(p_0,E_N,V_N) &= \left\{ p: P_0^N \log (p_0^N/p^N) \leq NE_N, ~ \var_{P_0^N} (\log(p_0^N/p^N) \leq V_N \right\} \\
&= \left\{ p: P_0 \log (p_0/p) \leq E_N, ~ \var_{P_0} (\log(p_0/p)) \leq V_N/N \right\},
\end{align*}
where the last equality follows from the densities tensorizing in the i.i.d. model. In this model, one can take $E_n \simeq \eps_N^2$ and $V_N \simeq N\eps_N^2$ (\cite{GV17}, Lemma 8.10 with $k=2$), which gives some intuition behind the roles of $E_N$ and $V_N$. In contrast, the present diffusion setting induces a Markovian dependence structure in the data and hence the likelihood. In the low frequency regime where $D>0$ is fixed, the spectral gap of the Markov chain is bounded away from zero by \eqref{eq:spectral_gap} and hence one can use spectral techniques \cite{Paulin2015,NS17,N22b} to show that the model dependence is not too dissimilar from i.i.d.. However, in the present high frequency setting, the spectral gap shrinks rapidly as $D\to 0$ inducing a much stronger dependence, thereby making such techniques unsuitable. We must therefore deal with the full dependent log-likelihood ratio process $\Lambda_D^N$ as a whole, rendering this computation much more involved.

Our strategy is to use second-order small time expansions of the transition densities \cite{azencott1984densite,berline2003heat,B20}, which heuristically corresponds to replacing $p_{f,D}$ in $\Lambda_D^N$ with the corresponding Euler scheme without drift, that is a $N_d(x,2Df(x) I_d)$ density. This reflects the property of model \eqref{eq: tanaka} that the drift is of smaller order than the diffusivity, and hence does not affect the leading order terms in $\Lambda_D^N$ as $D \to 0$. Such expansions can also be viewed as a quantitative form of local asymptotic normality (LAN) with uniform remainders over $\C_N$. We then further approximate the resulting process by a martingale difference, which allows us to deal with the dependence of the process. Note that we require control of the derivatives of $f-f_0$ up to third order in $\C_N$ in \eqref{eq:Cn} to ensure the higher order terms of these expansions are negligible, uniformly over $\C_N$.

Turning to the existence of tests with exponentially decreasing type-II errors, we employ plug-in tests based on estimators satisfying suitable concentration inequalities as was done for i.i.d. models in \cite{GN11}, see also \cite{R13,NS17,A18,MRS20}. As estimators, we follow the ideas of Comte et al. \cite{CGCR07} who use penalized least squares estimators in the scalar case. We therefore extend the results of \cite{CGCR07} to the multidimensional setting with domain boundary, exploiting the structure of the $L^2$-loss function to obtain the required concentration inequalities. 

Our proofs require certain wavelet-based approximations, which must be adapted to deal with the boundary as we now make precise. Let
\begin{equation} \label{eq: def offset}
\mathcal O_0^{\delta} = \{x \in \mathcal O,\,\mathsf{dist}(x,\mathcal O_0) \leq \delta/2\}
\end{equation}
denote the $\delta/2$-enlargement of $\mathcal O_0$, and note that $\mathcal O_0^{\delta} \subsetneq \mathcal O$ since $\mathsf{dist}(\mathcal O_0, \partial \mathcal O) \geq \delta$ by assumption \eqref{eq:delta}, see Figure \ref{fig:domain} for an example. As in Section \ref{sec:GP}, let $\{\psi_{lr}: l\geq J_0,r \in \Z^d\}$ denote an orthonormal basis of $L^2(\R^d)$ composed of sufficiently regular, compactly supported Daubechies wavelets, where $J_0 \in \mathbb{N}$ is the base resolution level of the scaling functions, which we also denote by $\{\psi_{J_0,r}\}$ to simplify notation. Let $R_l$ denote the set of indices $r$ for which the support of $\psi_{lr}$ intersects $\O_0$, in which case $|R_l| = O(2^{ld})$ since $\O_0$ is a bounded domain in $\R^d$. Since $\mathsf{dist}(\O_0^\delta, \partial \O) \geq \delta/2>0$ and $\text{diam}(\supp(\psi_{lr})) = O(2^{-l})$, we may take $J_0$ large enough such that no wavelet function $\psi_{lr}$, $l \geq J_0$, has support intersecting both $\O_0$ and $(\O_0^\delta)^c$. Any function $g\in L^2(\O)$ with $\supp(g) \subseteq \O_0$ can then be uniquely represented as
\begin{equation}\label{eq:wavelet_expansion}
g = \sum_{l=J_0}^\infty \sum_{r \in R_l} \langle g,\psi_{lr} \rangle_{2} \psi_{lr}.
\end{equation}
Even though the wavelets $\{\psi_{lr}: l \geq J_0, r\in R_l\}$ do not form an orthonormal basis of $L^2(\O)$ due to their behaviour on $\O \setminus \O_0$, any $g$ as above can be extended to a function on $\R^d$ (or any set containing $\O_0$) by setting it to zero outside $\O_0$. In particular, the Sobolev and H\"older norms on $\O$ and $\R^d$ coincide for such functions and we may therefore use all the usual wavelet characterizations and embeddings for Sobolev and H\"older norms. This fact will be used without mention in the proofs.

For $J \geq J_0$, set
\begin{equation}\label{eq:VJ}
V_J = V_J(\O_0) = \left\{ f = \sum_{l = J_0}^J \sum_{r \in R_l} f_{lr} \psi_{lr}  \right\} \subset \C(\O)
\end{equation}
to be the linear space of all functions in the wavelet projection space of resolution level $J$, restricted to those wavelets with support intersecting $\O_0$ but not $(\O_0^\delta)^c$. Note that
\begin{equation} \label{eq: technically useful}
g(x) = g(x) {\bf 1}_{x \in \mathcal O_0^{\delta}} \qquad \qquad \text{for every}\;\;g \in V_J
\end{equation}
since $\supp (g) \subseteq \O_0^\delta$ for $g\in V_J$. We must slightly modify the usual notion of a wavelet projection to account for the behaviour near the boundary in $\F$ in \eqref{eq:F}. To that end, note that $\supp(f-1)\subseteq \O_0$ for any $f\in \F$ and thus $f-1$ has a wavelet expansion as in \eqref{eq:wavelet_expansion}. Setting $P_J:L^2(\O) \to V_J$ to be the $L^2$-projection operator onto $V_J$, we define the projection operator $\overline{P}_J : \F \to \F$ by
\begin{align}\label{eq:PJ}
\overline{P}_J[f](x) := 1+P_J [f-1](x) = 1 +  \sum_{l = J_0}^J \sum_{r \in R_l} \langle f-1,\psi_{lr} \rangle_{2} \psi_{lr} (x) ,\qquad  x\in\O.
\end{align}
Since $\supp(f-1) \subseteq \O_0$, $P_J[f-1](x) =\sum_{l = J_0}^J \sum_{r \in R_l} \langle f-1,\psi_{lr} \rangle_{2} \psi_{lr} (x)$ coincides with the usual wavelet projection of $f-1$ on all of $\R^d$. Intuitively, $\overline{P}_J[f]$ should be thought of as the usual wavelet projection of resolution level $J$ adapted to account for the boundary conditions of $\F$.

We are now ready to define our projection estimator following Comte {\it et al.} \cite{CGCR07}. Let
\begin{equation}\label{eq:estimator}
\widehat g_{N} \in \argmin_{g \in V_J}\sum_{i = 1}^N\big( (Y_{i,D} - 1){\bf 1}_{\mathcal A_{i,D}}-g(X_{(i-1)D})\big)^2,
\end{equation}
with $\mathcal A_{i,D} = \{X_{(i-1)D} \in \mathcal O_0^\delta\}$ and 
$$Y_{i,D} = \frac{1}{2dD}|X_{iD}-X_{(i-1)D}|^2.$$
Note that by \eqref{eq: technically useful}, the sum in \eqref{eq:estimator} actually spans over $i$ such that $X_{(i-1)D}$ lies in $\mathcal O_0^\delta$. We then consider the estimator
\begin{equation}\label{eq:estimator_f}
\widehat f_N(x)  = 1 + \widehat{g}_N(x), \qquad \qquad x \in \O.
\end{equation}
The idea behind \eqref{eq:estimator} is that for small $D>0$, {by It\^o's formula, we have the signal plus noise representation} 
$$Y_{i,D}-1 = f(X_{(i-1)D}) -1+ \varepsilon_{i,D}+ r_{i,D},$$
where $\varepsilon_{i,D}$ is a martingale error term with variance of order $1$ and $r_{i,D}$ is a small remainder term combining stochastic expansions and boundary effects, see \eqref{eq: expand square} in the proofs for a precise definition. Thus $\widehat{g}_N$ is an estimate of $f-1$, which has support contained in $\O_0$ for all $f\in \F$. Indeed, one needs only estimate $f$ on $\O_0$, since $f \equiv 1$ is already known on $\O \backslash \O_0$ for all $f\in \F$. This permits to separate estimation on the interior of $\O$, where the function $f$ is unknown, with the behaviour near the boundary $\partial \O$, where the reflecting boundary conditions alter the diffusion dynamics. We next establish a concentration inequality for the estimator $\widehat{f}_N$.

\begin{thm}[Exponential inequality] \label{thm:exp_inequality}
Let $\widehat{f}_N = \widehat{f}_N(X_0,X_1,\dots,X_{ND})$ be the estimator \eqref{eq:estimator_f}. Let $\eps_N,\xi_N \to 0$, $2^J = 2^{J_N} \to\infty$ and $R_J$ be sequences such that $N\eps_N^2 \to \infty$ as $N\to \infty$. Define the sets
\begin{align*}
\F_N' = \big\{ f\in \F: \|f\|_{\mathcal C^1} \leq r, & ~ \|f-\overline{P}_Jf\|_2 \leq C\xi_N, ~ \|f-\overline{P}_Jf\|_\infty \leq R_{J} \big\},
\end{align*}
where $C,r>0$ and $\overline{P}_Jf$ denotes the projection \eqref{eq:PJ}. Assume further that $2^{Jd} = o(\sqrt{ND})$,
\begin{equation*}
R_{J}^2  \eps_N^2 \lesssim D \xi_N^2, \qquad \qquad 2^{3Jd/2}N^{-1} + 2^{Jd/2} N^{-1/2} + 2^{Jd/2} \eps_N^2 + \eps_N \lesssim \xi_N,
\end{equation*}
as $N \to \infty$. Then there exist events $\mathcal B_N$ satisfying
$$\sup_{f \in \mathcal F}\PP_{f}(\mathcal B_N^c) \rightarrow 0$$
as $N \rightarrow \infty$, such that for every $L>0$,
$$\sup_{f\in \F_N'} \PP_f\big(\|\widehat f_N -\overline{P}_Jf \|_{2} \geq K\xi_N,\mathcal B_N\big) \leq C' e^{-LN\eps_N^2},$$
where $K$ depends on $L$.
\end{thm}

The proof can be found in Section \ref{sec:proof_exp_ineq}. Theorem \ref{thm:exp_inequality} says that if the underlying function $f$  can be well-enough approximated by its projection $\overline{P}_J f$ as quantified by the conditions in $\mathcal{F}_N'$, then the estimator $\widehat{f}_N$ will concentrate around this projection with all but exponentially small $\PP_f$-probability.\\

As far as we are aware, this extension of \cite{CGCR07} in Theorem \ref{thm:exp_inequality} is the first frequentist estimator of the diffusion function $f$ in the multivariate setting and thus may be of independent interest. 
Combined with the lower bound of Theorem \ref{thm:lower_bd}, we obtain 
\begin{thm} \label{thm: minimaxity final}
Let $D= N^{-a}$ for some $a\in(1/2,1)$, $s > \alpha_d$, and let $M>0$. Define $\widehat f_N^\star = \min(\widehat f_N, M)_+$, where $\widehat f_N$ is as in \eqref{eq:estimator_f} constructed with Daubechies wavelets with at least $\lfloor s\rfloor-1$ vanishing moments, 
and $2^J \simeq N^{1/(2s+d)}$. Then
$$\sup_{f \in \mathcal F_0, \|f\|_{\mathcal C^s} \leq M} \E_f\big[\|\widehat f_N-f\|_2^2\big] \lesssim N^{-2s/(2s+d)}.$$
Combined with Theorem \ref{thm:lower_bd}, we obtain that the (normalised) rate $N^{-s/(2s+d)}$ is asymptotically minimax for estimating $f$ over $\{f \in \mathcal{F}_0 :\|f\|_{\mathcal C^s} \leq M\}$.
\end{thm}

\begin{rk}[Minimax rates in related diffusion models]\label{rk:minimax}
The minimax rates for estimating a drift with smoothness $s_b$ and a diffusion coefficient with smoothness $s$ when both terms are decoupled are $(ND)^{-s_b/(2s_b+1)}$ and $N^{-s/(2s+1)}$ in dimension $d=1$, see \cite{H99}. Theorem \ref{thm:lower_bd} extends the lower bound for the diffusion coefficient to an arbitrary dimension and, together with Theorem \ref{thm: minimaxity final}, shows that the rate $N^{-s/(2s+d)}$ is indeed minimax optimal. While we prove our results in a coupled model for which the drift has the form $\nabla f(x)$, a glance at the proofs of the upper bound and the lower bound, where we use the control of the KL divergence from Theorem \ref{thm:variance} in an extended model, shows that the result is actually valid in a more general setting with decoupled drift and diffusion coefficient.
\end{rk}

The proof of Theorem \ref{thm: minimaxity final} is given in Appendix \ref{sec: proof minimaxity}. We emphasize that while this estimator plays a key role in our \textit{proof} of minimax optimal posterior contraction rates using the Gaussian prior \eqref{eq:prior}, it is not actually involved in the prior construction or Bayesian method itself. Using Theorems \ref{thm:variance} and \ref{thm:exp_inequality}, we obtain the following general posterior contraction theorem for our setting, whose proof is found in Section \ref{sec:general_contract_proof}.

\begin{thm}[General contraction theorem]\label{thm:post_contract}
Let $\Pi=\Pi_N$ be a sequence of prior distributions supported on $\F$ in \eqref{eq:F}, let $r,K_0>0$ be fixed constants and $0 < \varepsilon_N \leq \varepsilon_{1,N}\leq \eps_{2,N} \leq \eps_{3,N} \to 0$, $\xi_N \to 0$, $2^J = 2^{J_n} \to\infty$ and $R_{J}$ be sequences satisfying $N\eps_N^2\to\infty$, $2^{Jd} = o(\sqrt{ND})$ and 
\begin{equation}\label{eq:exp_ineq_conditions}
R_{J}^2  \eps_N^2 \lesssim D \xi_N^2, \qquad \qquad 2^{3Jd/2}N^{-1} + 2^{Jd/2} N^{-1/2} + 2^{Jd/2} \eps_N^2 + \eps_N \lesssim \xi_N,
\end{equation}
as $N\to \infty$. Further suppose that $E_N$ and $V_N$ in \eqref{eq:EnVn} satisfy
 $$E_N \leq K_0 \eps_N^2, \qquad \text{and} \qquad  V_N /(N^2 \eps_N^4) \to 0.$$ 
Let 
$$\C_N = \Big\{f \in \mathcal F: \|f\|_{\C^{\alpha}} \leq r, \|f-f_0\|_{\infty} \leq \varepsilon_N,  \|f-f_0\|_{\mathcal C^k} \leq \varepsilon_{k,N} ~ \text{for } k=1,2,3 \Big\}$$
be the set defined in \eqref{eq:Cn},
\begin{align*}
\F_N \subseteq \big\{ f\in \F: \|f\|_{\mathcal C^1} \leq r, & ~ \|f-\overline{P}_Jf\|_2 \lesssim \xi_N, ~ \|f-\overline{P}_Jf\|_\infty \leq R_{J} \big\},
\end{align*}
where $\overline{P}_J$ denotes the projection \eqref{eq:PJ}, and $C_0>0$ be the fixed constant in Theorem \ref{thm:variance}, which depends only on $f_0,r,d,\O,\delta,f_{\min}$.
Assume the true $f_0\in \F_0$ satisfies $\|f_0-\overline{P}_Jf_0\|_2 \lesssim \xi_N$ and $\|f_0 - \overline{P}_J f_0\|_\infty \lesssim R_{J}$. Suppose that for some $C,L>0$,
\begin{enumerate}
\item[(i)] $\Pi(\F_N^c) \leq Le^{-(C+C_0K_0+2)N\eps_N^2}$,
\item[(ii)] $\Pi(\C_N) \geq e^{-CN\eps_N^2}$.
\end{enumerate}
Then for $M>0$ large enough, as $N \to\infty$,
$$\E_{f_0} \Pi(f: \|f-f_0\|_2 \geq M \xi_N|X_0,X_D,\dots,X_{ND}) \to 0.$$
\end{thm}

To summarize, in order to apply Theorem \ref{thm:post_contract} in examples, we must verify that the prior places most of its mass on functions $f\in \F_N$ that are well-approximable by their projections $\overline{P}_Jf$, and also that the prior places at least an exponentially small amount of mass in a neighbourhood of the truth in the sense of $\C_N$, i.e. conditions (i) and (ii) of Theorem \ref{thm:post_contract}, respectively.

\section{Discussion and generalizations}

\subsection{The parameter space $\mathcal{F}_0$}\label{sec:F_class}

The parameter space $\F_0$ in \eqref{eq:F0} captures relevant features of the underlying physical model, while allowing us to focus on the inferential task of estimating the diffusivity $f_0$ in the \textit{interior} of the domain $\O$ without being overly encumbered by boundary technicalities. We consider reflecting boundary conditions, rather than say a simpler periodic model \cite{NR20,GR20}, to better represent the microscopic process corresponding to diffusion as in \eqref{eq:heat_equation}. However, estimation near boundaries often behaves in qualitatively different ways \cite{Ray16,Reiss20}, for instance exhibiting different convergence rates. Since this is not the focus of the present work, we assume $f_0\in\F_0$ is known near the boundary $\partial \O$ to avoid such statistical issues (note we must still deal with probabilistic aspects of the boundary in our proofs). We set $f=1$ on $O\backslash \mathcal K$ for simplicity and clarity of presentation, but our theory could accommodate more general constraints of the form $f = g$ on $\O\backslash \mathcal{K}$ for some \textit{known} smooth $g \geq 2f_{\min}$ by considering a modified prior $f(x) = \Phi(\chi(x)W(x) + \Phi^{-1}(g(x)))$. Our proofs can straightforwardly be adapted by considering wavelet projections of the form $\overline{P}_J[f](x) := g+P_J [f-g](x)$ instead of \eqref{eq:PJ}, since this still separates out the boundary because $\supp (f-g) \subseteq \O_0^\delta$ for any $f$ in the prior support.

The lower bound $f_0\geq 2f_{\min}$ is arbitrary, but provides uniformity for our results in terms of the constants in the convergence rates. While our prior link function depends on this unknown quantity, this is purely for our proof arguments and in any case one can take $f_{\min}$ as close to zero as desired. In practice, we would expect the standard exponential link function $\Phi(x) = e^x$ (i.e. with $f_{\min}=0$) often used in inverse problems \cite{S10,N22} to work equally well. The class $\F_0$ also assumes minimal smoothness $f_0\in\C^{\alpha_d}$, where $\alpha_d$ in \eqref{eq:alpha} grows roughly like $\max(4,d/2)$. The $d/2$-dependence is needed to establish suitable bounds on the transition densities as in \eqref{eq: estimate coulhon}, which follow from results in \cite{D89,C03}. The minimal smoothness $\alpha_d \geq 4$ dominates for small dimensions and is an artifact of our proof, which requires expansions of the transition densities to third order with suitable control of the remainder terms.

\subsection{Extensions to other models}

While the model \eqref{eq: tanaka} we consider here is motivated by physical considerations, our results conceptually extend to SDEs with more general drifts of the form
$$dX_t = b(\nabla f(X_t),X_t) dt + \sqrt{2f(X_t)} dW_t + d\ell_t,$$
where $b$ is suitably regular, see \eqref{eq: tanaka ext} in the proofs below for the exact formulation. Indeed, many of our probabilistic results are actually proved for this more general framework, which includes the present model with $b(x,y) =x$, while also allowing an independent drift $G(X_t)$ by taking $b(x,y) = G(y)$. Our results and techniques should thus be extendable to estimation of the diffusion coefficient $f$ in these more general models where either the drift is viewed as a nuisance parameter or is less informative than the diffusion term regarding $f$.

To deal with the boundary, we first show that in small time $D\to 0$, a particle started away from the boundary will hit it with very small probability (Step 1 in Section \ref{sec:variance_bound_proof}). Our proof then approximates the present model by one \textit{without boundary} on the whole of $\R^d$, see \eqref{eq: diff whole space} below, so that our probabilistic expansions actually treat this unbounded diffusion case. Our techniques should thus extend to diffusion models where one can show that with high probability, the particle is constrained to a compact set, for instance unbounded domains in $\R^d$ with sufficiently strong confining drifts. In this way, our proof essentially separates out the boundary problem (corresponding to either reflecting boundary conditions or confining drifts) from statistical inference in the interior of the effective domain.

One can also consider the problem of simultaneous estimation of the drift and diffusion coefficient when these are in principle uncoupled as in \eqref{eq:diff_general}. However, in the one-dimensional case $d=1$, it is known that these estimation problems interact with each other and it can still be helpful to couple the priors, for instance by first modelling the diffusion coefficient $f$ and then modelling the drift \textit{conditional} on $f$, see Remark 7 in \cite{NS17}. This is an interesting direction, but is beyond the scope of the present work.

\section{Proof of Theorem \ref{thm:variance}: expectation and variance of the integrated log-likelihood restricted to $\C_N$ } \label{sec:small_ball}

The proof of the variance and expectation of the log-likelihood is split into several approximation steps employing different techniques. 
Following the suggestion of a referee, we write the proof of Theorem 8 and its technical auxiliary results in a slightly more general setting: we replace \eqref{eq: tanaka} by the extended model
\begin{equation} \label{eq: tanaka ext}
\begin{split}
X_t & = X_0 + \int_0^t b(\nabla f(X_s),X_s)ds+ \int_0^t \sqrt{2f(X_s)}dB_s+ \ell_t, 
\\
& \ell_t  = \int_0^t n(X_s)d|\ell|_s, \qquad |\ell|_t = \int_0^t {\mathbf 1}_{\{X_s \in \partial \mathcal O\}} d|\ell|_s,
\end{split}
\end{equation}
The drift function: $b = (b^1,\ldots, b^d): \R^{2d}  \rightarrow \R^d$ belongs to the class 
\begin{equation} \label{eq: def drift class}
\mathcal B = \left\{b: \R^{2d}  \rightarrow \R^d,\; \sum_{i = 1}^d\|b^i\|_{\mathcal C^4} \leq \mathfrak b \right\}
\end{equation}
for some $\mathfrak b >0$.
In particular, we recover model \eqref{eq: tanaka} for the choice $b(x,y) = x$ and putting $b(x,y) = G(y)$ for some arbitrary sufficiently smooth $G:\R^d \rightarrow \R^d$ enables us to have a generic drift vector field that can be considered as a nuisance.

We provide here an overview of the proof, deferring the detailed technical arguments to later in the paper to aid readability.
We first prove the more difficult variance bound for $\Lambda_{D}^N$ in Theorem \ref{thm:variance}.

\subsection{Proof of the variance bound of Theorem \ref{thm:variance} }\label{sec:variance_bound_proof}

For an integer $m\geq 1$, real-valued random variables $(Z_{f,k}:f \in \mathcal F, k\geq 1)$ and a probability measure $\nu(df)$ with support $\mathcal A \subset \mathcal F$, we will repeatedly use the inequalities
\begin{align} 
\mathrm{Var}_{f_0}\Big(\int_{\mathcal F}\sum_{k=1}^m Z_{f,k} \nu(df)\Big)  
\leq \int_{\mathcal F} \mathrm{Var}_{f_0}\Big(\sum_{k=1}^m Z_{f,k}\big)\nu(df)   \leq m\sum_{k=1}^m\sup_{f \in \mathcal A}\E_{f_0}[Z_{k,f}^2].  \label{eq: rough}
\end{align}

\noindent {\it Step 1: Restricting observations away from the boundary}. We first decompose the log-likelihood according to whether $X_{(i-1)D} \in \mathcal O_0$ or not. On the event $\mathcal A_{i,D }= \{X_{(i-1)D}\in \mathcal O_0^\delta$\}, where $\mathcal O_0^\delta$ is the $\delta/2$-enlargement of $\mathcal O_0$ defined in \eqref{eq: def offset} above, the process does not hit the boundary during $[(i-1)D, iD]$ with overwhelming probability,  and hence we may use analytic approximation techniques that ignore the boundary reflection. Since $f=f_0=1$ near the boundary for $f,f_0 \in \mathcal{F}$, on $(\mathcal A_{i,D})^c$ the likelihood ratio is already close to one and thus makes a negligible contribution. We quantify this last observation in the following result.

\begin{lem}\label{lem:KL_near_boundary}
Let $\F'  = \{ f\in \F: \|f\|_{\mathcal C^\alpha} \le r\}$ for $\alpha=\max( 4 , 2 \left\lfloor d/4+1/2 \right\rfloor)$. Then, in Model \eqref{eq: tanaka ext}, for $k=1,2$, any $f,f_0 \in \F'$ and any $\gamma>0$,
\begin{align*}
&\int_{x \in \O \setminus \O_0^{\delta}, y \in \O} \Big(\log \frac{p_{f_0,D}(x,y)}{p_{f,D}(x,y)}\Big)^k p_{f_0}(x,y) dx \, dy \\
& \leq  C_1 D^{-d/2} (1+\gamma^{2k} (\log N)^k) N^{-C_2 \gamma^2} + C_3 N^{C_4 \gamma^2} \exp(-c D^{-1}),
\end{align*}
where the constants $C_1-C_4,c>0$ are uniform over $(\mathcal O, d, f_{min}, r,\delta,\mathfrak b)$.
\end{lem}
The proof is based on explicit  bounds for the heat kernel in small time, for instance Theorem 3.2.9 in \cite{D89} for the upper bound and Theorem 3.1 in \cite{C03} for the lower bound 
 It is delayed until Section \ref{sec: proof of prop away from boundary 1}.  
Abbreviating $X_i^D = (X_{(i-1)D},X_{iD})$ and applying Lemma \ref{lem:KL_near_boundary} with $k=2$ together with \eqref{eq: rough} yields
\begin{align*}
&\mathrm{Var}_{f_0}\Big(\int_{\mathcal C_N} \sum_{i = 1}^N \log \frac{p_{f,D}}{p_{f_0,D}}(X_i^D) {\mathbf 1}_{(\mathcal A_{i,D})^c}\nu(df)\Big) \\
&\leq N^2 \int_{x\in \mathcal O\setminus \mathcal O_0^{\delta/2}, y \in \mathcal O} \Big( \log \frac{p_{f,D}(x,y)}{p_{f_0,D}(x,y)}\Big)^2p_{f_0,D}(x,y)dxdy \\
& \lesssim D^{-d/2} (\log N)^2 N^{-C_2 \mu} + N^{C_4 \mu} \exp(-c  D^{-1})
\end{align*}
for large enough $\mu>0$ and constants uniform over $\C_N$. In view of the statement of Theorem \ref{thm:variance}, it remains to consider the variance of $\int_{\mathcal C_N} \sum_{i = 1}^N \log \frac{p_{f,D}}{p_{f_0,D}}(X_i^D) {\mathbf 1}_{\mathcal A_{i,D}}\nu(df)$, i.e. we may restrict to the events $\mathcal{A}_{i,D}$.\\

\noindent {\it Step 2}: \textit{Approximation by an unbounded diffusion model on $\R^d$}. We can smoothly extend any $f \in \mathcal F$ to all of $\R^d$ by setting $f=1$ outside $\mathcal O$. Consider a diffusion process $(\widetilde X_t)_{t \geq 0}$ taking values on all of $\R^d$ arising as the solution to the stochastic differential equation 
\begin{equation} \label{eq: diff whole space}
\widetilde X_{t} = x+\int_{0}^{t}b(\nabla f(\widetilde X_s), X_s)ds+\int_{0}^{t}\big(2f(\widetilde X_s)\big)^{1/2}dB_s,
\end{equation}
defined on the same probability space with the same driving Brownian motion $(B_t)_{t \geq 0}$ as $(X_t)_{t \geq 0}$. Thus in this section, $\PP_f$ denotes a probability measure defined on a rich enough probability space  to accommodate a Brownian motion $(B_t)_{t \geq 0}$ (and therefore the strong solutions $(X_t)_{t \geq 0}$ of \eqref{eq: tanaka ext} and $(\widetilde X)_{t \geq 0}$ of  \eqref{eq: diff whole space} driven by the parameters $f, b$).

The process $(\widetilde X_t)_{t \geq 0}$ in turn generates a family of transition densities $\widetilde p_{f,D}: \mathcal \R^d \times \R^d \rightarrow [0,\infty)$. We may pick a continuous version of $\widetilde p_{f,D}$ so that
in particular, for any $x \in \mathcal O$,
\begin{equation}  \label{eq: diff whole}
\widetilde p_{f,D}(x,y)dy = \PP_{f}(\widetilde X_D \in dy\,|\,X_0=x).
\end{equation}
Consider the decomposition
\begin{align}
\int_{\mathcal C_N} \sum_{i = 1}^N \log \frac{p_{f,D}}{p_{f_0,D}}(X_i^D) {\mathbf 1}_{\mathcal A_{i,D}}\nu(df) &  = \int_{\mathcal C_N} \sum_{i = 1}^N \log \frac{\widetilde p_{f,D}}{\widetilde p_{f_0,D}}(X_i^D) {\mathbf 1}_{\mathcal A_{i,D}}\nu(df) \nonumber \\
& \quad + \int_{\mathcal C_N}  \sum_{i = 1}^N\Big(\log \frac{p_{f,D}}{\widetilde p_{f,D}}+\log \frac{\widetilde p_{f_0,D}}{p_{f_0,D}}\Big)(X_i^D) {\mathbf 1}_{\mathcal A_{i,D}}\nu(df). \label{eq:p_to_tildep}
\end{align}
We have an analogous result to Lemma \ref{lem:KL_near_boundary} for controlling the expectation and variance of the approximation of $\log p_{f,D}$ by $\log \widetilde p_{f,D}$ starting from $x \in \mathcal O_0^\delta$.
\begin{lem}\label{lem:p_to_tildep}
Let $\F'  = \{ f\in \F: \|f\|_{\mathcal C^\alpha} \le r \}$ for $\alpha=\max( 4 , 2 \left\lfloor d/4+1/2 \right\rfloor)$. Then for small enough $D$, $k=1,2$, in Model \eqref{eq: tanaka ext}, any $g = f,f_0 \in \F'$ and any $\gamma>0$,
\begin{align*}
&\int_{x \in \O_0^{\delta}, y \in \O} \Big(\log \frac{p_{g,D}(x,y)}{\widetilde p_{g,D}(x,y)}\Big)^k p_{f_0}(x,y) dx \, dy \\
& \leq  C_1 D^{-d/2} (1+\gamma^{2k} (\log N)^k) N^{-C_2 \gamma^2} + C_3 N^{C_4 \gamma^2} \exp(-c D^{-1}),
\end{align*}
where the constants $C_1-C_4,c>0$ are uniform over $(\mathcal O, d, f_{min}, r,\delta,\mathfrak b)$. 
The same estimate holds replacing $(\log \frac{p_{g,D}(x,y)}{\widetilde p_{g,D}(x,y)})^k$ by $(\log \frac{\widetilde p_{g,D}(x,y)}{p_{g,D}(x,y)})^k$.
\end{lem}
The proof uses similar tools as for Lemma \ref{lem:KL_near_boundary} and is delayed until Section \ref{sec: proof near boundary 2}. Using \eqref{eq: rough} and Lemma \ref{lem:p_to_tildep} with $k=2$ and $\gamma>0$ large enough, the $\PP_{f_0}$-variance of the second term in the RHS of \eqref{eq:p_to_tildep} is bounded above by
\begin{align*}
& N^2 \sup_{f \in \C_N, b \in \mathcal B} \E_{f_0}\Big[\Big(\log \frac{p_{f,D}}{\widetilde p_{f,D}}(X_0^D)+\log \frac{\widetilde p_{f_0,D}}{p_{f_0,D}}(X_0^D)\Big)^2 {\mathbf 1}_{\mathcal A_{i,D}}\Big] \\
& \lesssim D^{-d/2} (\log N)^2 N^{-C_2 \mu} + N^{C_4 \mu} \exp(-c  D^{-1}),
\end{align*}
for large enough $\mu>0$ and constants uniform over $\C_N$. 
We have thus established that for any large enough but fixed $\mu>0$,
\begin{align*}
\mathrm{Var}_{f_0}\left(\Lambda_{D}^N \right) & \lesssim \mathrm{Var}_{f_0} \Big(\int_{\mathcal C_N} \sum_{i = 1}^N \log \frac{\widetilde p_{f,D}}{\widetilde p_{f_0,D}}(X_i^D) {\mathbf 1}_{\mathcal A_{i,D}}\nu(df)\Big) \\
& \qquad +  D^{-d/2} (\log N)^2 N^{-C_2 \mu} + N^{C_4 \mu} \exp(-c  D^{-1}).
\end{align*}

\noindent {\it Step 3: Approximating $\widetilde p_{f,D}$ by a Gaussian transition density in small-time}. We approximate $\widetilde p_{f,D}$ by the transition density of the corresponding Euler scheme $N_d(x,2Df(x) I_d)$ without drift, namely
\begin{equation}\label{eq:proxy}
q_{f,D}(x,y) = \frac{1}{(4\pi D f(x))^{d/2}} \exp\Big(-\frac{|y-x|^2}{4Df(x)}\Big).
\end{equation}
Note that while $q_{f,D}$ is defined on $\R^d\times \R^d$, it will be evaluated at $X_i^D$ which lies in $\mathcal O \times \mathcal O$, $\PP_{f_0}$-almost surely. Define
$$\Lambda_{q,D}^N = \int_{\mathcal C_N} \sum_{i = 1}^N\log \frac{q_{f,D}}{q_{f_0,D}}(X_i^D){\mathbf 1}_{\mathcal A_{i,D}}\nu(df),$$
again writing $X_i^D = (X_{(i-1)D},X_{iD})$. The key estimate to prove Theorem \ref{thm:variance} is the following result, which quantifies this approximation.
\begin{prop}\label{prop:tildep_to_q}
For $f_0 \in \F_0$ and any probability measure $\nu$ supported on $\C_N$ as in \eqref{eq:Cn}, in Model \eqref{eq: tanaka ext}, it holds that
\begin{align*}
& \mathrm{Var}_{f_0} \Big(\int_{\mathcal C_N} \sum_{i = 1}^N \log \frac{\widetilde p_{f,D}}{\widetilde p_{f_0,D}}(X_i^D) {\mathbf 1}_{\mathcal A_{i,D}}\nu(df)\Big) \\
 & \lesssim \var_{f_0} \big( \Lambda_{q,D}^N \big) + N \eps_{1,N}^2  (D+ND^2) + N^2  \varepsilon_{2,N}^2 D^2 + N^2 \varepsilon_{3,N}^2 D^3+ N^2 D^4 \\
& \qquad \qquad \qquad \qquad + N^2 \eps_{1,N}^2 D\exp(-cD^{-1}) ,
\end{align*}
where the constants are uniform over $(\mathcal O, d, f_{min}, r,\delta,\mathfrak b)$.
\end{prop}
The proof of Proposition \ref{prop:tildep_to_q}  relies on second-order small time expansions of the heat-kernel and is deferred to Section \ref{sec:tildep_to_p}.\\

\noindent {\it Step 4: Variance of the proxy log-likelihood}. It remains to control the final variance term in the RHS of the last estimate. It involves the proxy density $q_{f,D}$ but is evaluated at the points of the original diffusion $X$ given by \eqref{eq: tanaka} with reflection at the boundary.
\begin{prop} \label{thm: variance proxy}
For $f_0 \in \F$ and any probability measure $\nu$ supported on $\C_N$ as in \eqref{eq:Cn}, in Model \eqref{eq: tanaka ext}, it holds that
$$\mathrm{Var}_{f_0} (\Lambda_{q,D}^N) \lesssim  N\varepsilon_N^2\big(1+D+ND^2+D^{-1}\varepsilon_N^2 +N e^{-C_{f_0}'D^{-1}} \big),$$
where the constants are uniform over $(\mathcal O, d, f_{min}, r,\delta,\mathfrak b)$.
\end{prop}

The proof of Proposition \ref{thm: variance proxy} is given in  Section \ref{sec:proxy_variance} and relies on martingale arguments, which are key to dealing with the dependence structure of the Markov chain. Combining the bounds from Steps 1-4 and keeping track of the leading order terms establishes the desired variance bound for $\Lambda_{D}^N$ in Theorem \ref{thm:variance}.

\subsection{Proof of the expectation bound of Theorem \ref{thm:variance}}

The proof follows along similar, though easier, lines as the variance bound and is written in the extended Model \eqref{eq: tanaka ext} likewise. The expectation bounds will be proved uniformly over $f\in \C_N$, which will sometimes be implied without explicit reference.\\

\noindent {\it Step 1: Restricting observations away from the boundary}. Applying Lemma \ref{lem:KL_near_boundary} above with $k=1$ gives
\begin{align*}
\E_{f_0}\Big[ \log \frac{p_{f_0,D}}{p_{f,D}}(X_0^D) \Big]  & \leq \E_{f_0}\Big[ \log \frac{p_{f_0,D}}{p_{f,D}}(X_0^D) {\mathbf 1}_{\mathcal A_{1,D}}\Big]  \\
& \quad + C_1 D^{-d/2} (\log N) N^{-C_2 \mu} + C_3 N^{C_4 \mu} \exp(-c  D^{-1})
\end{align*}
for large enough $\mu>0$, and where the constants $C_1-C_4,c>0$ are uniform over $\C_N$  {\color{red} and $\mathcal B$}.\\

\noindent {\it Step 2: Local approximation by an unbounded diffusion model on $\R^d$}. Let $\widetilde p_{f,D}: \mathcal \R^d \times \R^d \rightarrow [0,\infty)$ denote the transition density defined in \eqref{eq: diff whole} above, corresponding to the unbounded diffusion model \eqref{eq: tanaka ext} over the whole space $\R^d$.
Consider $\E_{f_0}$-expectation of the integrands in (minus) the decomposition \eqref{eq:p_to_tildep}. Using Lemma \ref{lem:p_to_tildep} above with $k=1$ and large enough $\gamma>0$, the $\E_{f_0}$-expectation of (minus) the integrand in the second term in \eqref{eq:p_to_tildep} is bounded above by
\begin{align*}
& \sup_{f \in \C_N, b \in \mathcal B} \E_{f_0}\Big[\log \frac{\widetilde p_{f,D}}{p_{f,D}}(X_0^D) {\mathbf 1}_{\mathcal A_{1,D}}\Big] 
 \lesssim D^{-d/2} (\log N) N^{-C_2 \mu} + N^{C_4 \mu} \exp(-c  D^{-1}),
\end{align*}
for large enough $\mu>0$. Arguing in the same way, an identical bound also holds for the $\E_{f_0}$-expectation of the integrand of minus the third term in \eqref{eq:p_to_tildep}. Therefore,
\begin{align*}
\E_{f_0} \Big[ \log \frac{p_{f_0,D}}{p_{f,D}}(X_0^D) {\mathbf 1}_{\mathcal A_{1,D}} \Big] &\lesssim  
\E_{f_0} \Big[ \log \frac{\widetilde p_{f_0,D}}{\widetilde p_{f,D}}(X_0^D) {\mathbf 1}_{{\mathcal A}_{1,D}} \Big] \\
&\qquad + D^{-d/2} (\log N) N^{-C_2 \mu} +  N^{C_4 \mu} \exp(-c  D^{-1}).
\end{align*}

\noindent {\it Step 3: Approximating $\widetilde p_{f,D}$ by a Gaussian transition density in small-time}. Recalling the definition \eqref{eq:proxy} of the proxy density $q_{f,D}$ from above, we now approximate the expected log-likelihood. The proof of the following proposition is deferred to Section \ref{sec:tildep_to_p}.

\begin{prop}\label{prop:tildep_to_q_expectation}
For $f_0 \in \F_0$ and any probability measure $\nu$ supported on $\C_N$ as in \eqref{eq:Cn}, in Model \eqref{eq: tanaka ext}, it holds that
\begin{align*}
& \sup_{f \in \C_N, b \in \mathcal B} \E_{f_0} \Big[ \log \frac{\widetilde p_{f_0,D}}{\widetilde p_{f,D}}(X_0^D) {\mathbf 1}_{\mathcal A_{1,D}} \Big] \\
& \lesssim \sup_{f\in \C_N, b \in \mathcal B} \E_{f_0} \Big[ \log \frac{q_{f_0,D}}{q_{f,D}}(X_0^D) {\mathbf 1}_{\mathcal A_{1,D}} \Big]  \\
&\quad +  \varepsilon_N^2 \Big(\tfrac{\eps_{1,N} }{\eps_N^2} D + \tfrac{\eps_{2,N}}{\eps_N^2}D + \tfrac{\eps_{3,N}}{\eps_N^2}D^{3/2} + \tfrac{\eps_{1,N}}{\eps_N^2} D^{1/2} \exp(-cD^{-1}) \Big),
\end{align*}
where the constants are uniform over $(\mathcal O, d, f_{min}, r,\delta,\mathfrak b)$.
\end{prop}

\noindent {\it Step 4: Expectation of the proxy log-likelihood}. It finally remains to control the remaining expectation  in Proposition \ref{prop:tildep_to_q_expectation} involving the proxy density $q_{f,D}$. The proof of the following estimate can be found in Section \ref{sec:proxy_variance}.

\begin{prop} \label{thm: expectation proxy}
For $f_0 \in \F$ and any probability measure $\nu$ supported on $\C_N$ as in \eqref{eq:Cn}, in Model \eqref{eq: tanaka ext}, it holds that
\begin{align*}
\sup_{f\in \C_N, b \in \mathcal B} \E_{f_0} \Big[ \log \frac{q_{f_0,D}}{q_{f,D}}(X_0^D) {\mathbf 1}_{\mathcal A_{1,D}} \Big] \lesssim  \varepsilon_N^2 + \eps_N D  + \eps_N D^{1/2} \exp(-cD^{-1}),
\end{align*}
where the constants are uniform over $(\mathcal O, d, f_{min}, r,\delta, \mathfrak b)$.
\end{prop}

Combining the bounds from Steps 1-4 and keeping track of the leading order terms establishes the expectation bound for $\Lambda_{D}^N$.

It remains to show the evidence lower bound \eqref{eq:elbo}. By Jensen's inequality, the desired probability in \eqref{eq:elbo} is upper bounded by
\begin{align*}
\PP_{f_0} \left( \int_{\mathcal C_N}  \sum_{i = 1}^N \log \frac{p_{f,D}(X_{i}^D)}{p_{f_0,D}(X_{i}^D)}\nu(df) \leq -cN\eps_N^2  - N E_N \right).
\end{align*}
Using the expectation bound just derived, the $\E_{f_0}$-expectation of the left-hand side is lower bounded by $-NE_N$. Using Chebychev's inequality and the variance bound just derived, the last probability is then bounded by
\begin{align*}
& \PP_{f_0} \Bigg( \int_{\mathcal C_N} \sum_{i = 1}^N \left( \log \frac{p_{f,D}(X_i^D)}{p_{f_0,D}(X_i^D)} - \E_{f_0} \left[ \log \frac{p_{f,D}(X_i^D)}{p_{f_0,D}(X_i^D)} \right] \right) \nu(df) \leq -cN\eps_N^2 \Bigg) \\
& \leq \frac{1}{c^2 N^2 \eps_N^4} \var_{f_0} \left(  \int_{\mathcal C_N} \sum_{i = 1}^N  \log \frac{p_{f,D}(X_i^D)}{p_{f_0,D}(X_i^D)} \nu(df) \right) \leq \frac{V_N}{c^2 N^2 \eps_N^4},
\end{align*}
which completes the proof of Theorem \ref{thm:variance}.

\section{Heat kernel estimates in small time}

In this section of independent interest, we use various estimates for the transition densities of diffusions to study the expectation and variance of the integrated log-likelihood process $\Lambda_{p,D}^N$ in Theorem \ref{thm:variance}. This is the key part of the paper where we combine stochastic calculus and perturbations of small time expansions for the transition density of a multidimensional diffusion process with reflections.\\ 

For terms near the boundary (the events $(\mathcal A_{i,D})^c)$, it suffices to use upper and lower bounds of the same order. 
Combining the upper bound from Theorem 3.2.9 of \cite{D89} with the lower bound from Theorem 3.1 of \cite{C03}, one has that for $x,y\in \O$, the transition densities of $(X_t)$ generated by \eqref{eq: tanaka ext} satisfy:
\begin{equation} \label{eq: estimate coulhon}
c_{f}^-D^{-d/2}\exp\Big(-C_{f}^-\frac{|y-x|^2}{D}\Big) \leq p_{f,D}(x,y) \leq c_{f}^+D^{-d/2}\exp\Big(-C_{f}^+\frac{|y-x|^2}{D}\Big),
\end{equation}
where these estimates are \textit{uniform} over $f \in \mathcal C_N \subset \{f: f\geq f_{min}, \|f\|_{\mathcal C^\alpha} \leq r\}$ and $b \in \mathcal B$ for any positive even integer $\alpha > d/2-1$, see also the proof of Proposition 4 of \cite{N22b} for the case of the conductivity equation with $b(x,y) = x$. In particular, the smallest such integer is $ 2 \left\lfloor d/4+1/2 \right\rfloor$, which is exactly the minimal smoothness assumption used to define $\alpha = \alpha_d$ in \eqref{eq:alpha}. For the main terms in the interior of $\O$ (the events $\mathcal A_{i,D}$), we will require more precise estimates, which are developed in Section \ref{sec:Riemann_proof} below.

\subsection{Proof of Lemma \ref{lem:KL_near_boundary}} \label{sec: proof of prop away from boundary 1}
Using \eqref{eq: estimate coulhon}, we have
$$\log \frac{c_{f_0}^-}{c_{f}^+}-(C_{f_0}^--C_{f}^+)\frac{|y-x|^2}{D} \leq \log \frac{p_{f_0,D}(x,y)}{p_{f,D}(x,y)} \leq \log \frac{c_{f_0}^+}{c_{f}^-}-(C_{f_0}^+-C_{f}^-)\frac{|y-x|^2}{D},$$
which implies for $k=1,2$,
\begin{equation} \label{eq: basic coulhon}
\left(\log \frac{p_{f_0,D}(x,y)}{p_{f,D}(x,y)}\right)^k p_{f_0}(x,y) \leq \big(c_{f_0,f}+C_{f_0,f}\frac{|y-x|^{2k}}{D^{k}}\big)D^{-d/2}\exp\big(-C_{f_0}^+\frac{|y-x|^2}{D}\big)
\end{equation}
for all $x,y\in \O$. Since all constants above are uniform over $\F'$ by \eqref{eq: estimate coulhon}, so too are all constants below, which will not be explicitly mentioned. For $\gamma>0$, the region of the integral such that $|y-x| \geq \gamma \sqrt{D\log N}$ thus satisfies
\begin{align} 
& \int_{x\in \mathcal O\setminus \mathcal O_0^{\delta}, |y-x| \geq \gamma \sqrt{D \log N}} \left( \log \frac{p_{f_0,D}(x,y)}{p_{f,D}(x,y)}\right)^k p_{f_0,D}(x,y)dx\, dy  \nonumber \\
& \qquad \leq C_{f,f_0} D^{-d/2} (1+\gamma^{2k} (\log N)^k) N^{-C_{f_0}^+\gamma^2}. \label{eq: debut off diag}
\end{align}
Using that $\log(1+z) \leq |z|$ for $z \geq -1$ and \eqref{eq: estimate coulhon}, for $k=1$, 
\begin{align*}
\log \frac{p_{f_0,D}(x,y)}{p_{f,D}(x,y)} p_{f_0,D}(x,y) & \leq |p_{f,D}(x,y)-p_{f_0,D}(x,y) | \frac{p_{f_0,D}(x,y)}{p_{f,D}(x,y)} \\
& \leq \frac{c_{f_0}^+}{c_f^-} |p_{f,D}(x,y)-p_{f_0,D}(x,y)|  \exp\left( (C_f^- -C_{f_0}^+) \frac{|y-x|^2}{D}\right).
\end{align*}
We obtain
\begin{equation}\label{eq:near_diagonal_int k=1}
\begin{split}
& \int_{x\in \mathcal O\setminus \mathcal O_0^{\delta}, |y-x| \leq \gamma \sqrt{D \log N}} \log \frac{p_{f_0,D}(x,y)}{p_{f,D}(x,y)} p_{f_0,D}(x,y)dx \, dy \\
& \qquad \leq C_{f,f_0} N^{C_{f,f_0}'\gamma^2} \int_{x \in  \mathcal O\setminus \mathcal O_0^{\delta}}\|p_{f,D}(x,\cdot)-p_{f_0,D}(x,\cdot)\|_{TV}dx,
\end{split}
\end{equation}
where $\|\cdot\|_{TV}$ denotes total variation norm and $p_{f,D}(x,\mathcal A) = \int_{\mathcal A}p_{f,D}(x,y)dy$ for any Borel set $\mathcal A \subset \mathcal O$ with a slight abuse of notation. In order to obtain a similar estimate as \eqref{eq:near_diagonal_int k=1} for $k=2$, we introduce the set
$$\mathcal E = \Big\{(x,y) \in \mathcal O \times \mathcal O:\Big|\frac{p_{f_0, D}(x,y)}{p_{f,D}(x,y)}-1\Big| \leq \tfrac{1}{2}\Big\}.$$
Using $(\log(1+z))^2 \leq 4z^2$ for $|z| \leq \tfrac{1}{2}$, for every $(x,y) \in \mathcal E$, we have
\begin{align*}
& \left( \log \frac{p_{f_0,D}(x,y)}{p_{f,D}(x,y)}\right)^2 \\
& \leq 4\frac{(p_{f,D}(x,y)-p_{f_0,D}(x,y)\big)^2}{\min\big(p_{f,D}(x,y), p_{f_0,D}(x,y))^2} \\
& \leq 4\big(p_{f,D}(x,y)-p_{f_0,D}(x,y)\big)^2\min(c_{f}^-,c_{f_0}^-)^{-2} D^{d} \exp\left( 2\max(C_{f}^+,C_{f_0}^+)\frac{|y-x|^2}{D}\right) \\
& \leq C'_{f,f_0} \big|p_{f,D}(x,y)-p_{f_0,D}(x,y)\big| D^{d/2} \exp\left(C_{f,f_0}\frac{|y-x|^2}{D}\right)
\end{align*}
using the rough bound
$$\big(p_{f,D}(x,y)-p_{f_0,D}(x,y)\big)^2 \leq \big|p_{f,D}(x,y)-p_{f_0,D}(x,y)\big|\big(p_{f,D}(x,y)+p_{f_0,D}(x,y)\big)$$
with \eqref{eq: estimate coulhon}.
We now obtain
\begin{equation}\label{eq:near_diagonal_int}
\begin{split}
& \int_{x\in \mathcal O\setminus \mathcal O_0^{\delta}, |y-x| \leq \gamma \sqrt{D \log N}, (x,y) \in \mathcal E} \Big( \log \frac{p_{f_0,D}(x,y)}{p_{f,D}(x,y)}\Big)^2 p_{f_0,D}(x,y)dx \, dy \\
& \qquad \leq C_{f,f_0}  D^{d/2} N^{C_{f,f_0}'\gamma^2} \int_{x \in  \mathcal O\setminus \mathcal O_0^{\delta}}\|p_{f,D}(x,\cdot)-p_{f_0,D}(x,\cdot)\|_{TV}dx.
\end{split}
\end{equation}
For $(x,y) \in \mathcal E^c$, thanks to \eqref{eq: basic coulhon}, we have
\begin{align*}
&\int_{x\in \mathcal O\setminus \mathcal O_0^{\delta}, |y-x| \leq \gamma \sqrt{D \log N}, (x,y) \in \mathcal E^c} \Big( \log \frac{p_{f_0,D}(x,y)}{p_{f,D}(x,y)}\Big)^2 p_{f_0,D}(x,y)dx \, dy \\
\leq &\big(c_{f_0,f}+C_{f_0,f}\gamma^{2k} (\log N)^k\big)D^{-d/2}\int_{x\in \mathcal O\setminus \mathcal O_0^{\delta}, |y-x| \leq \gamma \sqrt{D \log N}, (x,y) \in \mathcal E^c}dx \, dy.
\end{align*}
Moreover, the lower bound of \eqref{eq: estimate coulhon} enables to write
\begin{align*}
&\mathcal E^c \cap \big\{|y-x| \leq \gamma \sqrt{D \log N} \big\} \\
& = \big\{|p_{f_0,D}(x,y)-p_{f,D}(x,y)| > \tfrac{1}{2}p_{f,D}(x,y)\big\} \cap \big\{|y-x| \leq \gamma \sqrt{D \log N} \big\} \\
& \subset \Big\{|p_{f_0,D}(x,y)-p_{f,D}(x,y)| > \tfrac{1}{2}c_{f}^-D^{-d/2}\exp\big(-C_{f}^-\frac{|y-x|^2}{D}\big)\Big\} \cap \big\{|y-x| \leq \gamma \sqrt{D \log N} \big\} \\
& \subset \big\{|p_{f_0,D}(x,y)-p_{f,D}(x,y)| > \tfrac{1}{2}c_{f}^-D^{-d/2}(\log N)^{-C_{f}}\big\},
\end{align*}
and this entails
\begin{align}
&\big(c_{f_0,f}+C_{f_0,f}\gamma^{2k} (\log N)^k\big)D^{-d/2}\int_{x\in \mathcal O\setminus \mathcal O_0^{\delta}, |y-x| \leq \gamma \sqrt{D \log N}, (x,y) \in \mathcal E^c}dx \, dy \nonumber\\
\leq &  \big(c_{f_0,f}+C_{f_0,f}\gamma^{2k} (\log N)^k\big)2(c_{f}^-)^{-1}(\log N)^{C_{f}}\int_{x\in \mathcal O\setminus \mathcal O_0^{\delta}}|p_{f_0,D}(x,y)-p_{f,D}(x,y)|dx \, dy \nonumber\\
\leq & \,C_{f,f_0}(1+\gamma^2) (\log N)^{C_{f,f_0}} \int_{x \in  \mathcal O\setminus \mathcal O_0^{\delta}}\|p_{f,D}(x,\cdot)-p_{f_0,D}(x,\cdot)\|_{TV}dx. \label{eq: final log}
\end{align}
Putting together \eqref{eq:near_diagonal_int k=1}, \eqref{eq:near_diagonal_int} and \eqref{eq: final log}, we obtain, for $k=1,2$
\begin{align*}
& \int_{x\in \mathcal O\setminus \mathcal O_0^{\delta}, |y-x| \leq \gamma \sqrt{D \log N}} \Big(\log \frac{p_{f_0,D}(x,y)}{p_{f,D}(x,y)} \Big)^k p_{f_0,D}(x,y)dx \, dy \\
& \qquad \leq C_{f,f_0} N^{C_{f,f_0}'\gamma^2} \int_{x \in  \mathcal O\setminus \mathcal O_0^{\delta}}\|p_{f,D}(x,\cdot)-p_{f_0,D}(x,\cdot)\|_{TV}dx
\end{align*}
Thanks to Lemma \ref{lem: tot variation small} right below, the last display is then bounded by
$C_{f,f_0,\delta} N^{C_{f,f_0}'\gamma^2} \exp(-c_\delta D^{-1} )$, which completes the proof.

\subsection{The behaviour of the transition density near the boundary}

The following lemma shows that the transition densities for different functions in $\F$ are similar for points near the boundary.

\begin{lem} \label{lem: tot variation small} For $x \in \mathcal O \setminus \mathcal O_0^{\delta}$ and $f,f_0 \in \F$, it holds that
$$\|p_{f,D}(x,\cdot)-p_{f_0,D}(x,\cdot)\|_{TV} \leq \frac{16}{\sqrt{\pi}\delta} e^{-\delta^2/(16D)}.$$
\end{lem}

\begin{proof}
Write $\PP_x^f$ for the law of the solution $X$ to \eqref{eq: tanaka ext} with parameter $f$, conditional on $X_0=x$. Let $\tau = \inf\{t \geq 0, X_t \in \partial \mathcal O_0\}$ denote the hitting time of the boundary of $\mathcal O_0$. For a Borel set $\mathcal A \subset \mathcal O$, by the strong Markov property,
\begin{align*}
p_{f,D}(x,\mathcal A) & = \E_x^f[{\mathbf 1}_{\{X_D \in \mathcal A \}} {\mathbf 1}_{\{\tau> D\}}\big]+ \E_x^f[p_{f, D-\tau}(X_{\tau},\mathcal A) {\mathbf 1}_{\{\tau \leq D\}}\big],
\end{align*}
where, as before, we write $p_{f,D}(x,\mathcal A) = \int_{\mathcal A}p_{f,D}(x,y)dy$ for any Borel set $\mathcal A \subset \mathcal O$ with a slight abuse of notation. Since $x \in \mathcal O \setminus \mathcal O_0^{\delta}$, we have
$$ \E_x^f[{\mathbf 1}_{\{X_D \in \mathcal A \}} {\mathbf 1}_{\{\tau> D\}}\big]= \E_x^{f_0}[{\mathbf 1}_{\{X_D \in \mathcal A \}} {\mathbf 1}_{\{\tau> D\}}\big]$$ 
since the functions $f$ and $f_0$ coincide on $\mathcal O \setminus \mathcal O_0$. It follows that 
\begin{equation} \label{eq: tot var 1}
\|p_{f,D}(x,\cdot)-p_{f_0,D}(x,\cdot)\|_{TV} \leq 2\PP_x(\tau \leq D).
\end{equation}
Here, for $x \in \mathcal O \setminus \mathcal O_0^{\delta}$, the hitting time $\tau$ of the boundary of $\mathcal O_0$ by the process $(x+\sqrt{2}B_t+\ell_t)_{t \geq 0}$ has law that does not depend on $f$ nor $f_0$. Introduce now the two stopping times $0 \leq S \leq T$:
\begin{align*}
S & = \inf\{t\geq 0,\;X_t \in \mathcal \partial \mathcal O_0^{\delta}\}, \qquad \qquad T   =  \inf\{t\geq S,\;|X_t - X_S| \geq \delta\},
\end{align*}
the first hitting time of the boundary $\partial \mathcal O_0^{\delta}$ and the first exit time of the ball of center $X_S$ and radius $\delta/2$, respectively. Necessarily,
$$\big\{\tau \leq D\big\} \subset \left\{\sup_{0 \leq t \leq D}|X_{(t+S)\wedge T}-X_{S}| = \delta/2\right\}.$$
Moreover, the process $(X_{(t+S)\wedge T}-X_{S})_{t\geq 0}$ has the same law as $(\sqrt{2}B_{t\wedge T})_{t \geq 0}$. It follows that
\begin{equation*} \label{eq: tot var 2}
\begin{split}
\PP_x(\tau \leq D) \leq \PP_x\left(\sup_{0 \leq t \leq D}|B_{t\wedge T}| \geq \delta/2^{3/2}\right) & \leq  \PP_x\left(\sup_{0 \leq t \leq D}|B_{t}| \geq \delta/2^{3/2}\right) \\
& \leq 4\PP_x \left( B_D \geq \delta /2^{3/2} \right) \leq  \frac{4}{\sqrt{2\pi}} \frac{2^{3/2}}{\delta} \exp \left( -\frac{\delta^2}{16D} \right),
\end{split}
\end{equation*}
where the third inequality follow from the reflection principle for Brownian motion and the last inequality from the standard Gaussian tail bound. Combining \eqref{eq: tot var 1} with the last display proves the lemma. 
\end{proof}

\subsection{Proof of Lemma \ref{lem:p_to_tildep}} \label{sec: proof near boundary 2}
Let $g = f$ or $f_0$ in the notation below. For $\widetilde p_{g,D}(x,y)$, an estimate of the kind \eqref{eq: estimate coulhon} is classical (see e.g. \cite{aronson1967bounds, friedman2010stochastic} for diffusion processes over the whole space $\R^d$). Using the analogous estimate to \eqref{eq: estimate coulhon}, one can then argue exactly as in the proof of Lemma \ref{lem:KL_near_boundary}, again splitting the integral according to whether $|y-x| \leq \gamma \sqrt{D \log N}$ or not. In particular, using the analogues of \eqref{eq: debut off diag},\eqref{eq:near_diagonal_int} and \eqref{eq: final log}, one gets for $k=1,2$ and any $\gamma>0$,
\begin{align*} 
& \int_{x\in \mathcal O_0^{\delta}, y\in \O} \left( \log \frac{p_{g,D}(x,y)}{\widetilde p_{g,D}(x,y)}\right)^k p_{f_0,D}(x,y)dx\, dy  \\
& \quad \leq C_{g,f_0} D^{-d/2} (1+\gamma^{2k} (\log N)^k) N^{-C_{f_0}^+\gamma^2} + C_{g,f_0} N^{C_{g,f_0}'\gamma^2}  \int_{x\in \mathcal O_0^{\delta}} \|p_{g,D}(x,\cdot)-\widetilde p_{g,D}(x,\cdot)\|_{TV}dx,
\end{align*}
where the constants are uniform over $\F'$, $\|\cdot\|_{TV}$ denotes total variation distance on $\R^d$, and 
we implicitly extend $p_{g,D}(x,\cdot)$ into a probability measure on $\R^d$ by setting $p_{g,D}(x,\mathcal A) = p_{g,D}(x,\mathcal A \cap \mathcal O)$ for any Borel set $\mathcal A$ in $\R^d$. We claim that
\begin{equation} \label{eq: assoc trans}
\|p_{g,D}(x,\cdot)-\widetilde p_{g,D}(x,\cdot)\|_{TV} \leq 4d\exp\left(-\frac{\delta^2}{20d\|g\|_{\infty}D}\right)
\end{equation}
for $x \in \mathcal O_0^{\delta}$ and any $g\in \F$ with $D < D_0(\delta,\|g\|_{\mathcal C^1})$ small enough.
By \eqref{eq: assoc trans} the second last display is then bounded by
\begin{align*} 
C_{g,f_0} D^{-d/2} (1+\gamma^{2k} (\log N)^k) N^{-C_{f_0}^+\gamma^2} + C_{g,f_0} N^{C_{g,f_0}'\gamma^2}  \exp\left(-c_{|g|_\infty,\delta} D^{-1} \right),
\end{align*}
which completes the proof. It thus remains to prove \eqref{eq: assoc trans}.
%
Let $x \in \mathcal O_0^{\delta}$ and $\widetilde \tau =\inf\{t\geq 0, \widetilde X_t \in \partial \mathcal O\}$  be the hitting time of the boundary $\partial \mathcal O$. For any Borel set $\mathcal A$ in $\R^d$,  by the strong Markov property, 
\begin{align*}
p_{g,D}(x,\mathcal A) & = \E_g[{\mathbf 1}_{\{X_D \in \mathcal A \}} {\mathbf 1}_{\{\widetilde \tau> D\}}\,|\,X_0=x\big]+ \E_g[p_{g, D-\widetilde \tau}(X_{\widetilde \tau},\mathcal A \cap \mathcal O) {\mathbf 1}_{\{\widetilde \tau \leq D\}}\,|\,X_0=x\big] \\
& = \widetilde p_{g,D}(x,\mathcal A) -\E_g[{\mathbf 1}_{\{\widetilde X_D \in \mathcal A\}} {\mathbf 1}_{\{\widetilde \tau \leq D\}}\,|X_0=x\big] +\E_g[p_{g, D-\widetilde \tau}(X_{\widetilde \tau},\mathcal A \cap \mathcal O) {\mathbf 1}_{\{\widetilde \tau \leq D\}}\,|\,X_0=x\big],
\end{align*}
since $X_t$ and $\widetilde X_t$ both started at $x \in \mathcal O_0^{\delta}$ at $t=0$ coincide until they hit the boundary $\partial \mathcal O$. It follows that 
$$\| p_{g,D}(x,\cdot)-\widetilde p_{g,D}(x,\cdot) \|_{TV} \leq 2 \PP_g(\widetilde \tau \leq D\,|\,X_0=x).$$
We conclude thanks to Lemma \ref{lem: hitting time}.

\subsection{Approximating transition densities for small-time}\label{sec:Riemann_proof}

This section is devoted to the approximation of small-time transition densities for diffusions in $\R^d$.  We use some Riemannian geometry to derive second-order small-time expansions of the heat kernel, following results that date back to Azencott \cite{azencott1984densite}. Recall that 
$$\widetilde p_{f,D},\; q_{f,D}: \R^d \times \R^d \to[0,\infty)$$
denote, respectively, (a smooth version of) the transition density of $(\widetilde X_t)_{t \geq 0}$ defined in \eqref{eq: diff whole space}, and the proxy transition density of an Euler scheme without drift defined in \eqref{eq:proxy}. The next result gives an expansion of the log-likelihood ratio of the transition densities, uniformly over the domain $\O$. 
\begin{lem} \label{thm: first big approx}
There exist smooth functions: $\gamma_{f_0,f, b}: \mathcal O \rightarrow \R^d$ with
$$\mathrm{supp}(\gamma_{f_0,f,b}) \subset \mathcal O_0, \qquad \|\gamma_{f_0,f, b}\|_{\infty} \lesssim \|f-f_0\|_{\C^1},$$
where the constants in the inequality depend only on an upper bound for $\|f\|_{\C^1}$, $\|f_0\|_{\C^1}$ and also on $\mathfrak b$, and such that the following expansion holds: 
\begin{align*}
 \log \frac{\widetilde p_{f_0,D}(x,y)}{\widetilde p_{f,D}(x,y)}{\mathbf 1}_{\{x \in \mathcal O_0^{\delta}\}}  & = \log \frac{q_{f_0,D}(x,y)}{q_{f,D}(x,y)} {\mathbf 1}_{\{x \in \mathcal O_0^{\delta}\}} \\
&\quad +\gamma_{f_0,f}(x)\cdot (y-x)+\frac{1}{8D}|y-x|^2\,\nabla(f^{-1}-f_0^{-1})(x)\cdot (y-x) {\mathbf 1}_{\{x \in \mathcal O_0^{\delta}\}}\\
&\quad + \big(\|f-f_0\|_{\mathcal C^2} D+\|f-f_0\|_{\mathcal C^3}|y-x|D+|y-x|^2D + \|f-f_0\|_{\mathcal C^2}|y-x|^2\\
& \quad \qquad  +\|f-f_0\|_{\mathcal C^3}|y-x|^3+|y-x|^4\big)r_{f_0,f, b}(x,y),
\end{align*}
for a remainder term $r_{f_0,f, b}:\R^d \times \R^d \rightarrow \R$ satisfying 
\begin{equation} \label{eq: stability remainder}
\sup_{\|f\|_{\C^4}+ \|f_0\|_{\C^4}\leq r ,b \in \mathcal B} \|r_{f_0,f,b}\|_{\infty}  < \infty
\end{equation}
for every $r >0$.
\end{lem}

While only local to $(x,y) \in \O \times \O$, the stability property \eqref{eq: stability remainder} is sufficient for our purpose: the approximation is only needed for $(x,y) = (X_{(i-1)D}, X_{iD}) \in \mathcal O \times \mathcal O$, a property that holds $\PP_{f_0}$-almost surely. The proof of Lemma \ref{thm: first big approx} is deferred to Appendix \ref{sec:Riemann_proof}.

We rely on a key second-order estimate of the heat kernel in small time, associated to the metric tensor $g_{ij}(x) = f(x)^{-1}\delta_{ij}$ induced by the diffusion matrix $f(x)^{-1}\mathrm{Id}$ on $\R^d$, viewed as a Riemannian manifold. For $x,y \in \R^d$, let
$$\ell_f(x,y) = \inf\left\{\int_0^1 \frac{|\dot \gamma_t|}{f(\gamma_t)^{1/2}}dt,\;\gamma_0=x,\gamma_1=y\right\}$$ 
denote the Riemannian geodesic distance between $x$ and $y$, where the infimum is taken over all smooth paths $\gamma:[0,1]\rightarrow \R^d$ connecting $x$ and $y$ at times $0$ and $1$, and $\dot \gamma_t$ is the time derivative of $\gamma_t$.

\begin{lem} \label{lem: heat kernel expansion}
The following small time expansion holds:
\begin{align*}
\widetilde p_{f,D}(x,y) & = \frac{1}{(4\pi D)^{d/2}}\exp\Big(-\frac{1}{4D}\ell_f(x,y)^2\Big)
\big(\alpha_f(x,y)+\beta_f(x,y)D+D^2\Gamma_{f,D}(x,y)\big),
\end{align*}
where 
\begin{equation}  \label{eq: robust}
 \|\alpha_f\|_{\infty} +\|\beta_f\|_{\infty}+\sup_{D>0}\|\Gamma_{f,D}\|_{\infty} \lesssim 1,
\end{equation}
and the $\|\cdot\|_\infty$-norm is taken for $(x,y) \in \mathcal O$ and the estimates are uniform over $f \in \mathcal C_N$.
Moreover, there exist smooth (at least $\mathcal C^4$) real-valued functions  
$$\alpha : \R \rightarrow \R,\;\; \alpha^{i}: \R^{d\times 1 \times d} \rightarrow \R,\alpha^{ij}: \R^{d \times 1 \times d \times d^2} \rightarrow \R,\;\;\alpha^{ijk}: \R^{d \times 1 \times d \times d^2 \times d^3} \rightarrow \R$$
for $1 \leq i,j,k \leq d$ and independent of $f$, such that 
\begin{align}
\alpha_f(x,y) & = \alpha(f(x))+\sum_{i =1}^d\alpha^i\big(x,f(x), (\partial_{i'} f(x))_{i'}\big)(x^i-y^i) \nonumber\\
& +\sum_{i,j=1}^d\alpha^{ij}\big(x,f(x), (\partial_{i'} f(x))_{i'}, (\partial_{i'j'}^2f(x))_{i'j'}\big)(x^i-y^i)(x^j-y^j) \nonumber \\
& +\sum_{i,j,k=1}^d\alpha^{ijk}\big(x,f(x), (\partial_{i'} f(x))_{i'}, (\partial_{i'j'}^2f(x))_{i'j'},(\partial_{i'j'k'}^3f(x))_{i'j'k'}\big)(x^i-y^i)(x^j-y^j)(x^k-y^k) \nonumber\\
& +|x-y|^4r_f(x,y), \label{eq: De Witt}
\end{align}
where $x = (x^1,\ldots, x^d), y=(y^1,\ldots, y^d)$ and the remainder term satisfies 
$\|r_f\|_{\infty} \lesssim 1$, 
uniformly in $f \in \mathcal C_N$ with the supremum taken over $(x,y) \in \mathcal O \times \mathcal O$. Moreover, 
\begin{equation} \label{eq: rep alpha}
\alpha(f(x)) = f(x)^{-d/2}.
\end{equation}
An analogous expansion holds for $\beta_f$ (except for \eqref{eq: rep alpha}) with smooth (at least $\mathcal C^4$) functions
$$\beta: \R^{d \times 1 \times d \times d^2} \rightarrow \R,\;\;\beta^{i}: \R^{d \times 1 \times d \times d^2 \times d^3} \rightarrow \R,$$
for $1 \leq i \leq d$, independent of $f$ and such that
\begin{align}
\beta_f(x,y) & = \beta(x,f(x),  (\partial_{i'} f(x))_{i'},  (\partial_{i'j'} f(x))_{i'j'}) \nonumber\\
& \quad +\sum_{i =1}^d\beta^i\big(x,f(x), (\partial_{i'} f(x))_{i'},  (\partial_{i'j'} f(x))_{i'j'},  (\partial_{i'j'k'} f(x))_{i'j'k'}\big)(x^i-y^i) +|x-y|^2 \widetilde r_f(x,y). \label{eq: De Witt beta}
\end{align}
\end{lem}

The existence of a small time expansion of the heat kernel in the first part of the lemma is classical, see {\it e.g.} \cite{berline2003heat}. The robust estimates \eqref{eq: robust} follows from the main result of Azencott \cite{azencott1984densite}. The second part, namely the form of the expansion  \eqref{eq: De Witt} for the functions $\alpha$ and $\beta$ and how they involve the derivatives of $f$ at $x$ uses the recent result of Bilal \cite{B20} that obtains explicit representations by mixing small time and space expansions, a key idea to control the density near the diagonal in small time. We sketch Bilal's approach and results in the proof below and refer the reader to Section 2 of Bilal's paper together with his appendix for more details.

\begin{proof}[Sketch of proof of Lemma \ref{lem: heat kernel expansion}]

Consider the Fokker-Planck equation
\begin{equation} \label{eq: FP for Bilal}
\partial_t \rho_{t} + \mathrm{div}(b_{\nabla f} \rho_{t}) = \tfrac{1}{2}\sum_{i,j=1}^d\partial^2_{ij}((\sigma_f\sigma_f^\top)_{ij}  \rho_{t})
\end{equation}
with initial value $\rho_0(dx)$ as a probability distribution.
In our case, $b_{\nabla f}(x)= b(\nabla f(x), x)$  and $(\sigma_f(x) \sigma_f(x)^\top)_{ij} = 2f(x)\delta_{ij}$, yielding
$$\partial_t \rho_t  = \Delta (f\rho_t)- \mathrm{div}(b(\nabla f, \cdot)\rho_t) =: \mathcal D_f \rho_t,$$
Whenever existence and uniqueness hold, the solution is given by
$\rho_t(y) = \int_{\mathcal O} \widetilde p_{f,t}(x,y) \rho_0(dx),$
where $\widetilde p_{f,t}(x,y)$ is the Markov transition density associated to the process $\widetilde X$ defined in \eqref{eq: diff whole space} on the whole space $\R^d$ and that we are looking to expand.  Rewriting $\widetilde p_{f,t}(x,y)$ as $f(x)^{-d/2}K^f_t(x,y)$, this satisfies the Fokker-Planck equation if
$$\partial_t K^f_t(x,\cdot) = \mathcal D_fK_t^f(x,\cdot)$$
with $K_{t}^f(x,y)dy \rightarrow f(x)^{d/2}\delta_x(dy)$ weakly as $t \rightarrow 0$.
Our ansatz for the heat kernel takes the form 
\begin{equation} \label{eq: ansatz}
K_t^f(x,y) \sim K_t^{f,0}(x,y)\sum_{r = 0}^\infty F_r(x,y)t^r,
\end{equation}
where 
$$K^{f,0}_t(x,y) = \frac{1}{(4\pi t)^{d/2}}\exp\left(-\frac{\ell_f(x,y)^2}{4t}\right)$$
and the notation $u_t \sim \sum_{r \geq 0}v_rt^r$ means $a_t=\sum_{r=0}^kb_rt^r + O(t^{k+1})$. 
The existence of such an expansion is provided for instance by Theorem 1.2. in Azencott \cite{azencott1984densite}. It remains to find a formula for the coefficients  $F_r(x,y)$ for $r=0,1$ 
that is compatible with the expansion \eqref{eq: De Witt} via the representation $\alpha_f(x,y) = f(x)^{-d/2}F_0(x,y)$ and $\beta_f(x,y) = f(x)^{-d/2}F_1(x,y)$, and we can then conclude with Azencott's result.\\ 

In order to do so, we follow Bilal's approach and sketch the Section 2 of his paper \cite{B20}. In that part of the proof and it that part only, we will use Einstein notation for differential operators.\footnote{{\it i.e.} when an index variable appears twice in a single term and is not otherwise defined it implies summation of that term over all the values of the index.} By plugging the ansatz \eqref{eq: ansatz} into  
\eqref{eq: FP for Bilal} and letting $t \rightarrow 0$, one obtains a formula for the functions $F_r$, which reads, for $r=0$ and $r=1$ (see in particular equations $(2.15)$ and $(2.16)$ in Section 2 of \cite{B20}):
\begin{equation} \label{eq: def F_0}
2g^{ij}\partial_i \partial_jF_0 = \big(-g^{ij}\partial_i\partial_j \ell_f^2+2d-b^i\partial_i\ell^2_f\big) F_0,\;\;F_0(x,x) =1,
\end{equation}
and
\begin{equation} \label{eq: def F_1}
(4+g^{ij}\partial_i \partial_j \ell^2_f-2d+b^i\partial_i \ell^2+2g^{ij}\partial_i \ell_f^2 \partial_j\big)F_1 = 0,
\end{equation}
with $\partial_i = \frac{\partial}{\partial y^i}$ and  where $g_{ij} = f^{-1}\delta_{ij}$ with inverse $g^{ij} = f\delta_{ij}$ is the metric tensor induced by the diffusion coefficient $\sqrt{2f(x)}\delta_{ij}$ and $b^i(x) = b^i(\nabla f(x), x)$ in coordinates. 
 We next expand $\ell_f^2$  around $x$ in the $y$ variable to obtain (see Equation $(A.21)$ in \cite{B20}), with $\epsilon = x-y$:
 \begin{align}
 \ell^2_f(x,y) & = g_{ij}\epsilon^i\epsilon^j +\tfrac{1}{2}\partial_kg_{ij}\epsilon^k\epsilon^i \epsilon^j +\big(\tfrac{1}{6} \partial_i \partial_j g_{lk}-\tfrac{1}{12} g_{nm}\Gamma^n_{ij} \Gamma^m_{kl}\big) \epsilon^i\epsilon^j\epsilon^k\epsilon^l + O(|\epsilon|^5), \label{eq: expansion metrique}
 \end{align}
 where $\Gamma_{ij}^k = \tfrac{1}{2}g^{kl}(\partial_i g_{il+\partial_j}g_{il}-\partial_lg_{ij})$ denotes the Christoffel symbol associated with the metric tensor $g_{ij}$,  
 and this provides with expansions for $\partial_i\ell_f^2$ and  $\partial_i\partial_j \ell_f^2$ around $x$ in the $y$ variable (with now $\partial_i = \tfrac{\partial}{\partial x^i}$):
 \begin{equation} \label{eq: def gradient metrique}
 \partial_i\ell_f^2 = 2g_{ij}\epsilon^i+\tfrac{3}{2}\partial_{(k}g_{ij)}\epsilon^j \epsilon^k+(\tfrac{2}{3}\partial_{(i}\partial_jg_{lk)}-\tfrac{1}{3}g_{nm}\Gamma^n_{ij} \Gamma^m_{kl)})\epsilon^j\epsilon^k\epsilon^l + O(|\epsilon|^4)
 \end{equation}
 and
 \begin{equation} \label{eq: def laplacien}
 \partial_i\partial_j \ell_f^2 = 2g_{ij}+3\partial_{(k}g_{ij)}\epsilon^k+(2\partial_{(i}\partial_j g_{lk)}-g_{nm}\Gamma^n_{(ij} \Gamma^m_{kl)})\epsilon^k\epsilon^l + O(|\epsilon|^3),
\end{equation}
 see in particular Equations $(A.22)$ and $(A.23)$ in \cite{B20} and where $a_{(i_1 \ldots i_n)}$ denotes symmetrisation in the indices\footnote{For instance $a_{(ij)} = \tfrac{1}{2}(a_{ij}+a_{ji})$.}. Next we expand $F_r(x,y)$, $r=0,1$, around $x$ in the $y$ variable, with $\epsilon = y-x$:
\begin{equation} \label{eq: expansion de Witt}
F_r(x,y) = t_r(x)+u_r^i(x)\epsilon^i+v_r^{ij}(x) \epsilon^i \epsilon^j + w_r^{ijk}(x) \epsilon^i\epsilon^j\epsilon^k+O(|\epsilon|^4).
\end{equation}
for some functions $t_r, u_r, v_r$ and $w_r$ and with $t_0(x)=1$ in particular. Plugging \eqref{eq: def gradient metrique} and \eqref{eq: def laplacien} in \eqref{eq: def F_0} and \eqref{eq: def F_1} and expanding the solution via the representation \eqref{eq: expansion de Witt} in powers of $\epsilon^i$, $\epsilon^i\epsilon^j$, $\epsilon^{i}\epsilon^j\epsilon^k$ is sufficient to obtain $t_r(x), u_r^i(x), v_r^{ij}$ and $w_r^{ijk}$, For instance, (2.33) in \cite{B20} explicitly gives
$$t_0(x) = 1,\;\;u_0^i(x) = -\tfrac{1}{2}a_i(x),\;\;v_0^{ij}(x) = \tfrac{1}{8}a_i(x)a_j(x)+\tfrac{1}{12}\mathcal R_{ij}-\tfrac{1}{8}\big(\partial_ia_j(x)+\partial_ja_i(x)\big)$$
where
$$a_i(x) = b^i(x)-\partial_jg^{ij}(x)-g^{ij}(x)\partial_jV(x),\;\;\exp(-2V(x)) = \mathrm{det}(g^{ij}(x))$$
and $\mathcal R_{jk} = \partial_i\Gamma_{jk}^i-\partial_j \Gamma_{ki}^i+\Gamma_{ip}^i\Gamma_{jk}^p-\Gamma_{jp}^i \Gamma_{jp}^i\Gamma_{ik}^p$ is the Ricci curvature tensor associated with the metric tensor $g_{ij}$. In particular,
$$u_0^i(x) = -\tfrac{1}{2}a_i(x) = b^i(\nabla f(x), x)-2\partial_i f(x)- d \,\mathrm{div}f(x).$$
Also (2.36) in \cite{B20} explicitly yields
$$t_0(x) = \tfrac{1}{6}g^{ij}(x)\mathcal R_{ij}(x) =  \beta(f(x),  (\partial_{i'} f(x))_{i'},  (\partial_{i'j'} f(x))_{i'j'}),$$
{\it i.e.} the scalar Ricci curvature. We can move forward with explicit computations that become increasingly more difficult with the order of approximation and that involve nontrivial geometric quantities. But an inspection of \eqref{eq: def F_0} and \eqref{eq: def F_1} shows that they involve explicit and rational functions of $b(\nabla f(x),x)$ and its derivatives, $f(x),\partial_i f(x)$, $\partial^2_{ij}f(x)$ and $\partial^3_{ijk}f(x)$ with increasing differentiation order for $F_0$ while $F_1$ involves derivatives of $f$ of order 2 and higher. The result follows.

\end{proof}

\noindent \textit{Proof of Lemma \ref{thm: first big approx}}.
Using Lemma \ref{lem: heat kernel expansion},  and a first-order Taylor's expansion, write
\begin{align}
\log \frac{\widetilde p_{f_0,D}(x,y) }{\widetilde p_{f,D}(x,y) }&  = -\frac{1}{4D} (\ell_{f_0}(x,y)^2-\ell_{f}(x,y)^2)+\log \frac{\alpha_{f_0}(x,y)}{\alpha_f(x,y)} \nonumber\\
&\qquad +\Big(\frac{\beta_{f_0}(x,y)}{\alpha_{f_0}(x,y)}-\frac{\beta_{f}(x,y)}{\alpha_{f}(x,y)}\Big)D+D^2\Gamma'_{f_0,f, D}(x,y),  \label{eq: rough taylor}
\end{align}
where $\sup_{f \in \mathcal C_N}\sup_{D>0}\|\Gamma'_{f_0,f, D}\|_\infty \lesssim 1$, thanks to \eqref{eq: robust}, \eqref{eq: rep alpha} and $f(y) \geq f_{\min} >0$ on $\O$. The linear term in $D$ can be rewritten as
\begin{align*}
\frac{\beta_{f_0}(x,y)}{\alpha_{f_0}(x,y)}-\frac{\beta_{f}(x,y)}{\alpha_{f}(x,y)} & =  \frac{\beta_{f_0}(x,y) - \beta_{f}(x,y)}{\alpha_{f_0}(x,y)} +  (\alpha_f(x,y) - \alpha_{f_0}(x,y))\frac{\beta_{f}(x,y)}{\alpha_{f_0}(x,y)\alpha_f(x,y)} .
\end{align*}
Using \eqref{eq: De Witt},
\begin{align*}
\alpha_f(x,y)-\alpha_{f_0}(x,y) & = \alpha(f(x))-\alpha(f_0(x)) \\
&\qquad +\sum_{i =1}^d\Big(\alpha^i\big(f(x), \partial_i f(x)\big)-\alpha^i\big(f_0(x), \partial_i f_0(x)\big)\Big)(x^i-y^i)\\
& \qquad +|x-y|^2r'_{f_0,f}(x,y),
\end{align*}
with $\|r'_{f_0,f}\|_{\infty} \leq C(\|f\|_{\mathcal C^3},\|f_0\|_{\mathcal C^3})$. 
By \eqref{eq: De Witt} and \eqref{eq: De Witt beta} of Lemma \ref{lem: heat kernel expansion}, we have
$$|\alpha_f(x,y)-\alpha_{f_0}(x,y)| \lesssim  \|f-f_0\|_{\infty}+\|f-f_0\|_{\mathcal C^1}|x-y|+|x-y|^2$$
and
$$|\beta_f(x,y)-\beta_{f_0}(x,y)| \lesssim \|f-f_0\|_{\mathcal C^2}+\|f-f_0\|_{\mathcal C^3}|x-y|+|x-y|^2,$$
 with the constant depending only on an upper bound for $\|f\|_{\C^4}$ and $\|f_0\|_{\C^4}$. Using the expansion for $\alpha_f(x,y)$ in Lemma \ref{lem: heat kernel expansion}, \eqref{eq: rep alpha} and that $f \geq f_{\min}$, we have that $\alpha_f(x,y), \alpha_{f_0}(x,y) \geq c > 0$, uniformly over $\C_N$. It follows that 
\begin{equation} \label{eq: remainder alpha beta}
\left|\frac{\beta_{f_0}(x,y)}{\alpha_{f_0}(x,y)}-\frac{\beta_{f}(x,y)}{\alpha_{f}(x,y)}\right|D \leq \big(\|f-f_0\|_{{\color{blue}\mathcal C^2}}D+\|f-f_0\|_{{\color{blue}\mathcal C^2}}|x-y|D+|x-y|^2D\big)r''_{f,f_0}(x,y),
\end{equation}
with $\sup_{f \in \mathcal C_N}\|r''_{f,f_0}\|_{\infty} \lesssim 1$.\\

 Next, writing $\widetilde {\alpha}_f(x,y) = f(y)^{d/2}{\alpha}_f(x,y)$ and $\widetilde {\alpha}_{f_0}(x,y) = f_0(y)^{d/2}{\alpha}_f(x,y)$,
\begin{align}
\log \frac{\alpha_{f_0}(x,y)}{\alpha_f(x,y)} & = \frac{d}{2}\log \frac{f(x)}{f_0(x)} + \log \widetilde \alpha_{f_0}(x,y) -\log \widetilde \alpha_f (x,y). \label{eq: decomp main} 
\end{align}
Using the expansion \eqref{eq: De Witt},
\begin{align*}
\widetilde \alpha_f(x,y) & = 1+\sum_{i =1}^d \widetilde \alpha^i\big(x,f(x), (\partial_{i'} f(x))_{i'}\big)(x^i-y^i) \nonumber\\
& +\sum_{i,j=1}^d \widetilde \alpha^{ij}\big(x,f(x), (\partial_{i'} f(x))_{i'}, (\partial_{i'j'}^2f(x))_{i'j'}\big)(x^i-y^i)(x^j-y^j) \nonumber \\
& +\sum_{i,j,k=1}^d \widetilde \alpha^{ijk}\big(x,f(x), (\partial_{i'} f(x))_{i'}, (\partial_{i'j'}^2f(x))_{i'j'},(\partial_{i'j'k'}^3f(x))_{i'j'k'}\big)(x^i-y^i)(x^j-y^j)(x^k-y^k) \nonumber\\
&+|x-y|^4r_f'(x,y)
\end{align*}
where similarly $\widetilde \alpha_f^\lambda(y) = f(y)^{d/2} \alpha^\lambda(y)$ for $\lambda = i, ij, ijk$ and $\alpha^\lambda(y)$ are the smooth functions given in \eqref{eq: De Witt} that do not depend on $f$ or $f_0$. For notational simplicity, using only a subscript to indicate the dependence on $f$ and setting $z=x-y$, the above expansion can be concisely written as
$$\widetilde \alpha_f = 1+\sum_i \widetilde \alpha_f^i z^i+\sum_{ij} \widetilde \alpha_f^{ij}z^iz^j+\sum_{ijk} \widetilde \alpha_f^{ijk}z^iz^jz^k+O(|z|^4).$$
Expanding the logarithm around $z=0$ to order 3,
\begin{align*}
\log \widetilde \alpha_f & = \sum_i \widetilde \alpha_f^i z^i+\sum_{ij}\widetilde \alpha_f^{ij}z^iz^j+\sum_{ijk}\widetilde \alpha_f^{ijk}z^iz^jz^k+O(|z|^4) \\
& \quad - \frac{1}{2}\Big(\sum_i \widetilde \alpha_f^i z^i+\sum_{ij}\widetilde \alpha_f^{ij}z^iz^j+\sum_{ijk}\widetilde \alpha_f^{ijk}z^iz^jz^k+O(|z|^4)\Big)^2\\
& \quad +\frac{1}{3}\Big(\sum_i \widetilde \alpha_f^i z^i+\sum_{ij}\widetilde \alpha_f^{ij}z^iz^j+\sum_{ijk}\widetilde \alpha_f^{ijk}z^iz^jz^k+O(|z|^4)\Big)^3 +O(|z|^4).
\end{align*} 
Keeping track of only the leading order terms, the quadratic term in the last display equals
\begin{align*}
\sum_{ij}(\alpha_f')^{ij}z^iz^j+\sum_{ijk}(\alpha_f'')^{ijk}z^iz^jz^k+O(|z|^4),
\end{align*}
with $(\alpha_f')^{ij} = \widetilde \alpha_f^{i} \widetilde \alpha_f^{j}$ and $(\alpha_f'')^{ijk} =  \widetilde \alpha_f^{i} \widetilde \alpha_f^{jk}$. Likewise, the cubic term equals $\sum_{ijk}(\alpha_f'')^{ijk}z^iz^jz^k+O(|z|^4)$ with $(\alpha_f'')^{ijk} = \widetilde \alpha_f^{i} \widetilde \alpha_f^{j} \widetilde \alpha_f^{k}$. Therefore,
$$\log \widetilde \alpha_f = \sum_i \bar \alpha_f^i z^i+\sum_{ij}\bar \alpha_f^{ij}z^iz^j+\sum_{ijk}\bar \alpha_f^{ijk}z^iz^jz^k+O(|z|^4),$$
with $\bar \alpha_f^i  = \widetilde \alpha_f^{i}$, $\bar \alpha_f^{ij} = \widetilde \alpha_f^{ij}-\tfrac{1}{2} \widetilde \alpha_f^{i} \widetilde \alpha_f^{j}$ and $\bar \alpha_f^{ijk} = \widetilde \alpha_f^{ijk} - \tfrac{1}{2} \widetilde \alpha_f^{i} \widetilde \alpha_f^{jk}+\tfrac{1}{3} \widetilde \alpha_f^{i} \widetilde \alpha_f^{j} \widetilde \alpha_f^{k}$.
Set now
\begin{align*}
\gamma_{f_0,f}^i(x) & = \bar \alpha_f^i\big(x,f(x), (\partial_{i'} f(x))_{i'}\big)-\bar  \alpha_{f_0}^i\big(x,f_0(x), (\partial_{i'} f_0(x))_{i'}\big), \\
\gamma_{f_0,f}^{ij}(x) & = \bar \alpha_f^{ij}\big(x,f(x), (\partial_{i'} f(x))_{i'}, (\partial_{i'j'}^2f(x))_{i'j'}\big)- \bar  \alpha_{f_0}^{ij}\big(x,f_0(x), (\partial_{i'} f_0(x))_{i'}, (\partial_{i'j'}^2f_0(x))_{i'j'}\big),\\
\gamma_{f_0,f}^{ijk}(x) & = \bar \alpha_f^{ijk}\big(x,f(x), (\partial_{i'} f(x))_{i'}, (\partial_{i'j'}^2f(x))_{i'j'}, (\partial_{i'j'k'}^3f(x))_{i'j'k'}\big) \\
& \quad -  \bar \alpha_{f_0}^{ijk}\big(x,f_0(x), (\partial_{i'} f_0(x))_{i'}, (\partial_{i'j'}^2f_0(x))_{i'j'}, (\partial_{i'j'k'}^3f_0(x))_{i'j'k'}\big)
\end{align*}
where we recall $\widetilde \alpha_f^\lambda(y) = f(y)^{d/2} \alpha^\lambda(y)$.
This gives the expansion 
\begin{align}
\log \frac{\widetilde \alpha_{f_0}(x,y)}{\widetilde \alpha_f (x,y)} &= \sum_{i =1}^d\gamma_{f_0,f}^i(x)(x^i-y^i)
 +\sum_{i,j=1}^d\gamma^{ij}_{f_0,f}(x)(x^i-y^i)(x^j-y^j) \nonumber \\
& \quad +\sum_{i,j,k=1}^d\gamma^{ijk}_{f_0,f}(x)(x^i-y^i)(x^j-y^j)(x^k-y^k)
 +|x-y|^4\widetilde r_{f_0,f}(x,y), \label{eq: new approx}
\end{align}
where $\|\widetilde r_{f_0,f}\|_{\infty} \lesssim 1$ is uniform over an upper bound for $\|f\|_{\mathcal C^4}$ and $\|f_0\|_{\mathcal C^4}$. By triangle inequality,
$$\|\gamma_{f_0,f}^i\|_{\infty} \lesssim  \|f^{d/2}-f_0^{d/2}\|_{\infty}  + \|f-f_0\|_{\C^1} \lesssim  \|f-f_0\|_{\C^1},$$
since $f,f_0\geq f_{\min}$ and where the constants in the inequality depend only on $f_{\min}$ and an upper bound for $\|f\|_{\C^1}$ and $\|f_0\|_{\C^1}$. Using similar expressions for $\|\gamma_{f_0,f}^{ij}\|_{\infty}$ and $\|\gamma_{f_0,f}^{ijk}\|_{\infty}$, we infer
\begin{align}
&\Big|\sum_{i,j=1}^d\gamma^{ij}_{f_0,f}(x)(x^i-y^i)(x^j-y^j) +\sum_{i,j,k=1}^d\gamma^{ijk}_{f_0,f}(x)(x^i-y^i)(x^j-y^j)(x^k-y^k)\Big| \nonumber \\
&\quad \lesssim  \|f-f_0\|_{\mathcal C^2} |x-y|^2+  \|f-f_0\|_{\mathcal C^3} |x-y|^3,
\label{eq: still approx}
\end{align}
again uniformly over $|f|_{\mathcal C^4}$ and $|f_0|_{\mathcal C^4}$. 
Hence, combining \eqref{eq: rough taylor}, \eqref{eq: remainder alpha beta}, \eqref{eq: decomp main}, \eqref{eq: new approx} and \eqref{eq: still approx}, we obtain
\begin{align*}
& \log \frac{\widetilde p_{f_0,D}(x,y)}{\widetilde p_{f,D}(x,y)}{\mathbf 1}_{\{x \in \mathcal O_0^{\delta}\}} \nonumber \\
 & = \Big(-\frac{1}{4D} (\ell_{f_0}(x,y)^2-\ell_{f}(x,y)^2)-\frac{d}{2}\log \frac{f_0(x)}{f(x)}\Big){\mathbf 1}_{\{x \in \mathcal O_0^{\delta}\}} \nonumber \\
&\quad +\gamma_{f_0,f}(x)\cdot (y-x){\mathbf 1}_{\{x \in \mathcal O_0^{\delta}\}} \nonumber \\
&\quad + \big(\|f-f_0\|_{\mathcal C^2}D+\|f-f_0\|_{\mathcal C^3}|y-x|D+|y-x|^2D + \|f-f_0\|_{\mathcal C^2}|y-x|^2 \nonumber \\
& \qquad \quad  +\|f-f_0\|_{\mathcal C^3}|y-x|^3+|y-x|^4\big)r_{f_0,f}(x,y),
\end{align*}
with $\sup_{f \in \mathcal C_N}\|r_{f_0,f}\|_{\infty} \lesssim 1$. With a slight abuse of notation, we may incorporate the term ${\mathbf 1}_{\{x \in \mathcal O_0^{\delta}\}}$ into the definition of $\gamma_{f_0,f}(x)$.  The final step is to expand the Riemannian metric $\ell_f$.

\begin{lem}  \label{lem: riemann dist approx}
We have
$$\ell_f(x,y)^2 = \frac{|x-y|^2}{f(x)}+\frac{1}{2}|x-y|^2\nabla f^{-1}(x)\cdot(x-y)+|x-y|^4r_f(x,y),$$
where $\sup_{f \in \mathcal C_N}\|r_f\|_{\infty} \lesssim 1$.
\end{lem}

\begin{proof}
This is textbook Riemannian geometry, see {\it e.g.} Equation $(A.21)$ in Appendix A.3 in \cite{B20} displayed in \eqref{eq: expansion metrique} above, with $\epsilon = x-y$ that gives (Einstein notation)
 \begin{align*}
 \ell^2_f(x,y) & = g_{ij}\epsilon^i\epsilon^j +\tfrac{1}{2}\partial_kg_{ij}\epsilon^k\epsilon^i \epsilon^j + O(|\epsilon|^4),
 \end{align*}
for the metric tensor $g_{ij} = (f(x))^{-1} \delta_{ij}$. This readily gives the result, the uniformity in $f \in \mathcal C_N$ being straighforward.
\end{proof}

Combining Lemma \ref{lem: riemann dist approx} with the expansion just derived completes the proof of Lemma \ref{thm: first big approx}.

\section{Remaining proofs for Theorem   \ref{thm:variance}}

\subsection{Preliminary estimates}
We first gather some technical bounds in the extended Model \eqref{eq: tanaka ext}. Although not difficult, the stochastic expansions we need require extra care due to the presence of a boundary.  We start with a standard variance estimate.
\begin{lem} \label{lem: variance esti}
For any real-valued function $\varphi$,  
\begin{equation*}\label{eq:variance D}
\mathrm{Var}_{f_0}\Big(\sum_{i = 1}^N\varphi(X_{(i-1)D})\Big) \lesssim ND^{-1} \var_{f_0} (\varphi (X_0)),
\end{equation*}
where the constant is uniform over $(\mathcal O, d, f_{min}, r,\delta,\mathfrak b)$.
\end{lem}
\begin{proof}
This is a consequence of the spectral gap property for the Markov chain $(X_0, X_D, \ldots, X_{ND})$ that reads $\|P_{f_0,D}\varphi\|_\infty \lesssim \exp(-\lambda_{f_0}D t)\|\varphi\|_{L^2}$. See {\it e.g.} \cite{cattiaux2017invariant}  for a probabilistic argument for diffusions with boundaries, but other analytical approaches are obviously possible.
\end{proof}
We next need a moment bound for the increments of the diffusion.
\begin{lem} \label{lem: moment bis}
For every $\tau \geq 0$ and $p\geq 1$, we have
\begin{equation} \label{eq: moment bound}
\E_{f}\Big[\sup_{s \leq u \leq t}|X_{u}-X_{s}|^p\,\big|\,\mathcal F_{s}\Big] \leq 2^{p-1}\big((2\mathfrak b)^p(1+\|f\|_{\mathcal C^1})^p(t-s)^p+c_\star^{p/2}d^{p/2}p^{p/2}\|f\|_\infty^p(t-s)^{p/2}\big),
\end{equation}
where $c_\star$ is a universal constant (arising in the Burkholder-Davis-Gundy inequality) and $\mathfrak b$ quantifies the size of the drift class $\mathcal B$ defined in \eqref{eq: def drift class}. 
\end{lem}
The proof is not difficult (see for instance Lions and Sznitman \cite{lions1984stochastic}) but requires some extra effort, due to the fact that we need to precisely track constants in $p$ for later use in Bernstein's inequalities. It is given in Appendix \ref{sec: technical bounds}. Let 
\begin{equation} \label{eq: def hitting}
\tau_{i,D} = \inf\{t \geq 0,\;X_{t+(i-1)D} \in \partial \mathcal O\}
\end{equation}
denote the hitting time of the boundary by the process $X$ started at $X_{(i-1)D}$. 

\begin{lem} \label{lem: hitting time}
Let $f_0 \in \F'  = \{ f\in \F: \|f\|_{\mathcal C^\alpha} \le r \}$ for $\alpha = \alpha_d$ as in \eqref{eq:alpha}. Then 
$$\PP_{f_0}\big(\tau_{i,D} \geq  D, \mathcal A_{i,D}\big) \lesssim \exp(-cD^{-1}),$$
where the constants are uniform over $(r, \delta,\mathfrak b)$. 
\end{lem}
The proof is given in Appendix \ref{sec: hitting time proof}. 
We will repeatedly use the decomposition
\begin{equation} \label{eq: def reflect}
X_{iD}-X_{(i-1)D} = b_{i,D}+\Sigma_{i,D}+L_{i,D},
\end{equation}
with
$$
b_{i,D} =  \int_{(i-1)D}^{iD} b(\nabla f_0(X_s), X_s)ds,\;\; 
\Sigma_{i,D} = \int_{(i-1)D}^{iD}\sqrt{2f_0(X_s)}dB_s,\;\;
L_{i,D}  = \int_{(i-1)D}^{iD}n(X_s) d|\ell|_s,
$$
together with the following bounds, for every $p \geq 1$,
\begin{equation} \label{eq: moment bounds reflect}
|b_{i,D}| \lesssim D,\;\;\E_{f_0}\big[|\Sigma_{i,D}|^p| \mathcal F_{(i-1)D}\big] \lesssim D^{p/2},\;\;\E_{f_0}\big[|L_{i,D}|^p{\bf 1}_{\mathcal A_{i,D}}\big] \lesssim D^{p/2}\exp(-cD^{-1}),
\end{equation}
which depend only on $\|f_0\|_{\C^1}$ and $\mathfrak b$.
The first bound is obvious, the second one stems from the Burkholder-Davis-Gundy inequality. For the third one, we rely on the following facts: first, writing $L_{i,D} = (X_{iD}-X_{(i-1)D}) - b_{i,D}-\Sigma_{i,D}$ and using the first two bounds of \eqref{eq: moment bounds reflect} together with Lemma \ref{lem: moment bis}, we have
$$\E_{f_0}\big[|L_{i,D}|^p\big] \lesssim D^{p/2}.$$
Second, on $\mathcal A_{i,D} \cap \{\tau_{i,D} \geq iD\}$, we have $L_{i,D}=0$. Thus, by Cauchy-Schwarz's inequality,
$$\E_{f_0}\big[|L_{i,D}|^p{\bf 1_{\mathcal A_{i,D}}}\big] \lesssim D^{p/2}\PP_{f_0}(\tau_{i,D} \geq D, \mathcal A_{i,D})^{1/2}$$
and the third estimate in \eqref{eq: moment bounds reflect} then follows from Lemma \ref{lem: hitting time}. 
\subsection{Proofs of Propositions \ref{prop:tildep_to_q} and \ref{prop:tildep_to_q_expectation}: approximating transition densities in small-time}\label{sec:tildep_to_p}

\begin{proof}[Proof of Proposition \ref{prop:tildep_to_q}]

We apply Lemma \ref{thm: first big approx}, establishing suitable bounds for the remainder terms.\\

 \noindent {\it Step 1:} Consider first the term $\sum_{i = 1}^N \gamma_{f_0,f}(X_{(i-1)D})\cdot (X_{iD}-X_{(i-1)D})$ in the remainder of Lemma \ref{thm: first big approx}. It splits into three parts thanks to the decomposition \eqref{eq: def reflect} and we bound each term separately. The drift term involving $b_{i,D}$ is of order $D$ by \eqref{eq: moment bounds reflect}, hence the property $\|\gamma_{f_0,f}\|_\infty \lesssim  \eps_{1,N}$
yields the crude variance bound $N^2  \eps_{1,N}^2 D^2$ for the first term. For the martingale term, we have
\begin{align*}
&\mathrm{Var}_{f_0}\Big(\sum_{i = 1}^N \gamma_{f_0,f}(X_{(i-1)D}) \cdot \Sigma_{i,D}\Big) \leq \sum_{i = 1}^N\|\gamma_{f_0,f}\|_\infty^2\E_{f_0}\big[|\Sigma_{i,D}|^2\big] \lesssim N \eps_{1,N}^2 D
\end{align*}
by the second estimate in \eqref{eq: moment bounds reflect}.
For the third term involving $L_{i,D}$, using that $\gamma_{f_0,f}(X_{(i-1)D})$ vanishes on $(\mathcal A_{i,D})^c$, we have
\begin{align*}
& \mathrm{Var}_{f_0}\Big( \sum_{i = 1}^N \gamma_{f_0,f}(X_{(i-1)D}) \cdot L_{i,D}\Big)  \lesssim N  \sum_{i = 1}^N \|\gamma_{f_0,f}\|_\infty^2\E_{f_0}\big[L_{i,D}^2{\bf 1}_{\mathcal A_{i,D}}\big] \lesssim N^2\varepsilon_{1,N}^2D\exp(-cD^{-1})
\end{align*}
by the third estimate in \eqref{eq: moment bounds reflect}. We have thus established
\begin{equation} \label{eq: basic variance order}
\mathrm{Var}_{f_0}\Big(\sum_{i = 1}^N \gamma_{f_0,f}(X_{(i-1)D})\cdot (X_{iD}-X_{(i-1)D})\Big) \lesssim N\varepsilon_N^2\big(\tfrac{\varepsilon_{1,N}}{\varepsilon_N}\big)^2
(D+ND^2+ND\exp(-cD^{-1}) \Big).
\end{equation}


\noindent {\it Step 2:} We next consider the term $ \frac{1}{8D}\sum_{i = 1}^N|X_{iD}-X_{(i-1)D}|^2 \zeta_{f_0,f}(X_{(i-1)D}) \cdot (X_{iD}-X_{(i-1)D}){\bf 1}_{\mathcal A_{i,D}}$, where $\zeta_{f_0,f}=  \nabla(f^{-1}-f_0^{-1})$ satisfies $\|\zeta_{f_0,f}\|_\infty \lesssim \eps_{1,N}$ for $f,f_0 \in \F'$. Using It\^o's formula, it splits into four parts according to the decomposition
\begin{equation} \label{eq: expand square}
|X_{iD}-X_{(i-1)D}|^2 = 2 d\,f_0(X_{(i-1)D})D+\widetilde b_{i,D}+\widetilde \Sigma_{i,D}+\widetilde L_{i,D},
\end{equation}
with
\begin{equation}\label{eq:tilde_decomp}
\begin{split}
\widetilde b_{i,D} & = 2 \int_{(i-1)D}^{iD} \big(d(f_0(X_s)-f_0(X_{(i-1)D}))+(X_s-X_{(i-1)D})\cdot b(\nabla f_0(X_s),X_s)\big)ds, \\
\widetilde \Sigma_{i,D} & = 2\int_{(i-1)D}^{iD}\sqrt{2f_0(X_s)}(X_s-X_{(i-1)D})\cdot dB_s, \\
\widetilde L_{i,D} & = 2\int_{(i-1)D}^{iD}(X_s-X_{(i-1)D})\cdot n(X_s) d|\ell|_s,
\end{split}
\end{equation}
appended with the moment estimates
\begin{equation} \label{eq: moment bounds reflect tilde}
\E_{f_0}\big[|\widetilde b_{i,D}|^p\big] \lesssim D^{3p/2},\;\;\E_{f_0}\big[|\widetilde\Sigma_{i,D}|^p| \mathcal F_{(i-1)D}\big] \lesssim D^{p},\;\;\E_{f_0}\big[|\widetilde L_{i,D}|^p{\bf 1}_{\mathcal A_{i,D}}\big] \lesssim D^{p}\exp(-cD^{-1}),\end{equation}
as for \eqref{eq: moment bounds reflect}.
More precisely, 
\begin{align*}
\E_{f_0}\big[|\widetilde b_{i,D}|^p\big]  & \leq D^{p-1}\int_{(i-1)D}^{iD} \E_{f_0}\big[ \big|d(f_0(X_s)-f_0(X_{(i-1)D}))+(X_s-X_{(i-1)D})\cdot b(\nabla f_0(X_s),X_s)\big|^p\big]ds \\
& \leq D^{p} C_{d,\|f_0\|_{\mathcal C^1}}^p\E_{f_0}\Big[\sup_{(i-1)D \leq s \leq iD}|X_s-X_{(i-1)D}|^p\Big]\\
& \leq C'_{p,d,\|f_0\|_{\mathcal C^1}} D^{3p/2}
\end{align*}
by Jensen's inequality and Lemma \ref{lem: moment bis}, with $C_{p,\|f_0\|_{\mathcal C^1}} = d\|f_0\|_{\mathcal C^1}+\mathfrak b(1+\|f_0\|_{\mathcal C^1}+\mathrm{diam}(\mathcal O))$ and 
$C'_{p,d,\|f_0\|_{\mathcal C^1}}$ equal to $C_{p,\|f_0\|_{\mathcal C^1}}$ times 
the constant in \eqref{eq: moment bound}. For the martingale part, the Burkholder-Davis-Gundy inequality with constant $C_p$ yields
\begin{align*}
\E_{f_0}\big[|\widetilde \Sigma_{i,D}|^p| \mathcal F_{(i-1)D}\big]  & \leq C_p 2^{3p/2} \E_{f_0}\Big[\Big(\int_{(i-1)D}^{iD} 2 f_0(X_s)^{3/2}|X_s-X_{(i-1)D}|^2ds\Big)^{p/2}\big| \mathcal F_{(i-1)D}\Big] \\
& \leq C_p 2^{2p}  \|f_0\|_{\infty}^{3p/4} \E_{f_0}\Big[\sup_{(i-1)D \leq s \leq iD}|X_s-X_{(i-1)D}|^{p}\Big] \\
&\leq C_p 2^{2p}  \|f_0\|_{\infty}^{3p/4}C_{p,d,f_0}D^{p},
\end{align*}
where we last used Lemma \ref{lem: moment bis}. The last bound in \eqref{eq: moment bounds reflect tilde} follows exactly the same lines as the last estimate in \eqref{eq: moment bounds reflect}, using now the bounds just established for $\widetilde b_{i,D}$ and $\widetilde \Sigma_{i,D}$ instead of those for $b_{i,D}$ and $\Sigma_{i,D}$, respectively.\\

We are now ready to handle  the term $\sum_{i = 1}^N \tfrac{d}{4}f_0(X_{(i-1)D}){\bf 1}_{\mathcal A_{i,D}}\zeta_{f_0,f}(X_{(i-1)D}) \cdot (X_{iD}-X_{(i-1)D})$, exactly as in Step 1, substituting $\zeta_{f_0,f}(X_{(i-1)D})$ by 
$ \tfrac{d}{4} f_0(X_{(i-1)D})\zeta_{f_0,f}(X_{(i-1)D}) {\bf 1}_{\mathcal A_{i,D}}.$
It has the same variance order as in \eqref{eq: basic variance order}.
For the term involving the drift part $\widetilde b_{i,D}$ we have 
\begin{align*}
&\mathrm{Var}_{f_0}\Big( \frac{1}{8D}\sum_{i = 1}^N\widetilde b_{i,D} \zeta_{f_0,f}(X_{(i-1)D}) \cdot (X_{iD}-X_{(i-1)D}){\bf 1}_{\mathcal A_{i,D}}\Big) \\
&\lesssim D^{-2}N\sum_{i = 1}^N \|\zeta_{f_0,f}\|_{\infty}^2\E_{f_0}\Big[\,\widetilde b_{i,D}^2 |X_{iD}-X_{(i-1)D}|^2\Big].
\end{align*}
and this yields the order $D^{-2}N^2\varepsilon_{1,N}^2D^4 = N\varepsilon_N^2 (\tfrac{\varepsilon_{1,N}^2}{\varepsilon_N^2}ND^2)$ 
by Cauchy-Schwarz's inequality combined with \eqref{eq: moment bounds reflect tilde}  and Lemma \ref{lem: moment bis} for controlling the term within the expectation. For the third term, we use the decomposition
$$\widetilde \Sigma_{i,D} \zeta_{f_0,f}(X_{(i-1)D}) \cdot (X_{iD}-X_{(i-1)D}){\bf 1}_{\mathcal A_{i,D}} = I+II+III,$$
with
\begin{align*}
I & = \widetilde \Sigma_{i,D} \zeta_{f_0,f}(X_{(i-1)D}) \cdot b_{i,D}{\bf 1}_{\mathcal A_{i,D}}, \\
II & = \widetilde \Sigma_{i,D} \zeta_{f_0,f}(X_{(i-1)D}) \cdot \Sigma_{i,D}{\bf 1}_{\mathcal A_{i,D}}, \\
III &= \widetilde \Sigma_{i,D} \zeta_{f_0,f}(X_{(i-1)D}) \cdot L_{i,D}{\bf 1}_{\mathcal A_{i,D}}.
\end{align*}
For term $I$,
\begin{align*}
\mathrm{Var}_{f_0}\Big( \frac{1}{8D}\sum_{i = 1}^N\widetilde \Sigma_{i,D} \zeta_{f_0,f}(X_{(i-1)D}) \cdot b_{i,D}{\bf 1}_{\mathcal A_{i,D}}\Big) 
\lesssim D^{-2}N\sum_{i = 1}^N \|\zeta_{f_0,f}\|_{\infty}^2\E_{f_0}\big[|\widetilde \Sigma_{i,D}|^2 |b_{i,D}|^2\big]
\end{align*}
and this yields the order $D^{-2}N^2\varepsilon_{1,N}^2D^4$ by Cauchy-Schwarz's inequality combined with \eqref{eq: moment bounds reflect tilde}  and Lemma \ref{lem: moment bis} again. For term $II$, by It\^o's formula,
\begin{align}
\widetilde \Sigma_{i,D}\zeta_{f_0,f}(X_{(i-1)D}) \cdot\Sigma_{i,D} & = \zeta_{f_0,f}(X_{(i-1)D}) \cdot M_{i,D} \nonumber \\
& +4\zeta_{f_0,f}(X_{(i-1)D}) \cdot \int_{(i-1)D}^{iD} (X_s-X_{(i-1)D})f_0(X_s)ds, \label{eq: ito dec}
\end{align}
where $ M_{i,D}  = \int_{(i-1)D}^{iD}\sqrt{2f_0(X_s)}\big(\E_{f_0}[\widetilde \Sigma_{i,D} |\mathcal F_s]+2\E_{f_0}[\Sigma_{i,D} |\mathcal F_s]\cdot (X_s-X_{(i-1)D})\big)dB_s$ is a martingale increment such that 
$$\E_{f_0}\big[ |\zeta_{f_0,f}(X_{(i-1)D}) \cdot M_{i,D} |^2\big] \lesssim \varepsilon_{1,N}^2D^3,$$
by applying repeatedly Cauchy-Schwarz's inequality together with  \eqref{eq: moment bounds reflect}, \eqref{eq: moment bounds reflect tilde}  and Lemma \ref{lem: moment bis}, 
hence 
\begin{align*}
\mathrm{Var}_{f_0}\Big( \frac{1}{8D}\sum_{i = 1}^N \zeta_{f_0,f}(X_{(i-1)D}) \cdot M_{i,D}{\bf 1}_{\mathcal A_{i,D}}\Big) 
= \frac{1}{64D^2}\sum_{i = 1}^N\E_{f_0}\big[ |\zeta_{f_0,f}(X_{(i-1)D}) \cdot M_{i,D} |^2{\bf 1}_{\mathcal A_{i,D}}\big]
\end{align*}
and this term is of order $D^{-2}N\varepsilon_{1,N}^2D^3 = N\varepsilon_N^2(\tfrac{\varepsilon_{1,N}}{\varepsilon_N})^2D$. Let $G(x,y) = \zeta_{f_0,f}(x) \cdot (y-x)f_0(y)$. For the second term in \eqref{eq: ito dec}, by It\^o's formula,
\begin{align*}
\int_{(i-1)D}^{iD} \zeta_{f_0,f}(X_{(i-1)D}) \cdot (X_s-X_{(i-1)D})f_0(X_s)ds 
 = \overline b_{i,D}(G) + \overline{\Sigma}_{i,D}(G)+\overline{L}_{i,D}(G),
\end{align*}
with
\begin{align}
\overline b_{i,D}(G) & =\int_{(i-1)D}^{iD} \int_{(i-1)D}^s\mathcal L_{f_0}G(X_{(i-1)D},X_u)duds, \nonumber \\
\overline \Sigma_{i,D}(G) & = \int_{(i-1)D}^{iD}\int_{(i-1)D}^s\nabla G(X_{(i-1)D},X_u) \sqrt{2f_0(X_u)} \cdot dB_u\,ds, \nonumber \\
\overline L_{i,D}(G) & = \int_{(i-1)D}^{iD}\int_{(i-1)D}^s \nabla G(X_{(i-1)D},X_u)\cdot n(X_u) d|\ell|_u ds. \label{eq: ito second level}
\end{align}
In notation, the differential operators act on $y \mapsto G(X_{(i-1)D}, y)$ and $\mathcal L_{f_0} = b(\nabla f_0(x),x) \mathrm{div} + f_0(x)\Delta$. 
The expansion is appended with the moment estimates 
\begin{align*}
&\E_{f_0}\big[|\overline b_{i,D}(G)|^p\big] \lesssim \varepsilon_{1,N}^{p}D^{2p},\;\;\E_{f_0}\big[|\overline \Sigma_{i,D}(G)|^p| \mathcal F_{(i-1)D}\big] \lesssim \varepsilon_{1,N}^{p}D^{3p/2},\\
&\;\;\E_{f_0}\big[|\overline L_{i,D}(G)|^p{\bf 1}_{\mathcal A_{i,D}}\big] \lesssim \varepsilon_{1,N}^{p}D^{3p/2}\exp(-cD^{-1}),
\end{align*}
in the same way as before. The variance of   $(4D)^{-1}\sum_{i = 1}^N \overline{b}_{i,D}(G){\bf 1}_{\mathcal A_{i,D}}$ is of order $D^{-2}N^2 \varepsilon_{1,N}^2D^4$. Also, the $\overline{\Sigma}_{i,D}(G)$ are centered $\mathcal F_{iD}$-increments hence
\begin{align*}
\mathrm{Var}_{f_0}\Big( \frac{1}{4D}\sum_{i = 1}^N\overline{\Sigma}_{i,D}(G)\,{\bf 1}_{\mathcal A_{i,D}}\Big) 
& = \frac{1}{16D^2}\sum_{i = 1}^N\E_{f_0}\big[ (\overline \Sigma_{i,D}(G))^2{\bf 1}_{\mathcal A_{i,D}}\big]  \lesssim D^{-2}N\varepsilon_{1,N}^2D^3
\end{align*}
and this term is of order $N\varepsilon_{N}^2(\tfrac{\varepsilon_{1,N}}{\varepsilon_N})^2D$. The variance of the term $(4D)^{-1}\sum_{i = 1}^N\overline{L}_{i,D}(G){\bf 1}_{\mathcal A_{i,D}}$ is of order $D^{-2}N^2 \varepsilon_{1,N}^2D^3\exp(-cD^{-1})$, and the term $II$ is controlled. For the term $III$, we simply have
\begin{align*}
\mathrm{Var}_{f_0}\Big( \frac{1}{8D}\sum_{i = 1}^N\widetilde \Sigma_{i,D} \zeta_{f_0,f}(X_{(i-1)D}) \cdot L_{i,D}{\bf 1}_{\mathcal A_{i,D}}\Big) 
\lesssim D^{-2}N\sum_{i = 1}^N \varepsilon_{1,N}^2\E_{f_0}\big[|\widetilde \Sigma_{i,D}|^2 |L_{i,D}|^2\,{\bf 1}_{\mathcal A_{i,D}}\big]
\end{align*}
that yields the order $N^2\varepsilon_{1,N}^2D^2\exp(-cD^{-1})$. Gathering all these estimates, we obtain that the variance of 
$\frac{1}{8D}\sum_{i = 1}^N|X_{iD}-X_{(i-1)D}|^2 \zeta_{f_0,f}(X_{(i-1)D}) \cdot (X_{iD}-X_{(i-1)D}){\bf 1}_{\mathcal A_{i,D}}$
is no bigger than the order established in \eqref{eq: basic variance order}.\\

\noindent {\it Step 3:} We finally control the remainder term in the expansion of Lemma \ref{thm: first big approx}. Define
\begin{align*}
&\mathcal R_{f_0,f}^N = \sum_{i = 1}^N\Big(\|f-f_0\|_{\mathcal C^2} D+\|f-f_0\|_{\mathcal C^3}|X_{iD}-X_{(i-1)D}|D+|X_{iD}-X_{(i-1)D}|^2D\\
& \qquad \qquad \qquad  + \|f-f_0\|_{\mathcal C^2}|X_{iD}-X_{(i-1)D}|^2+\|f-f_0\|_{\mathcal C^3}|X_{iD}-X_{(i-1)D}|^3\\
&\qquad \qquad \qquad + |X_{iD}-X_{(i-1)D}|^4\Big)r_{f_0,f}(X_{(i-1)D}, X_{iD}).
\end{align*}
By Lemma \ref{lem: moment bis}, the property $f \in \C_N$ and the fact that $\|r_{f_0,f}\|_\infty \lesssim 1$, we readily have
\begin{align*}
\mathrm{Var}_{f_0} (\mathcal R_{f_0,f}^N) & \lesssim N^2\big(\varepsilon_{2,N}^2D^2+ \eps_{3,N}^2 D^3+D^4 + \eps_{2,N}^2 D^2+ \eps_{3,N}^2 D^3+D^4\big) \nonumber\\
& \lesssim N\varepsilon_N^2\big( \big(\tfrac{\varepsilon_{2,N}}{\varepsilon_N})^2 ND^2 + (\tfrac{\varepsilon_{3,N}}{\varepsilon_N})^2 ND^3+N\varepsilon_N^{-2}D^4\big). \label{eq: final 3}
\end{align*}

\noindent {\it Step 4:} Putting together the estimates established in Steps 1-3, using Lemma \ref{thm: first big approx} and \eqref{eq: rough}, the quantity in Proposition \ref{prop:tildep_to_q} is bounded by a multiple of
\begin{align*}
& \var_{f_0} \big( \Lambda_{q,D}^N \big) + \mathrm{Var}_{f_0} \Big(\int_{\mathcal C_N} \sum_{i = 1}^N \Big[ \log \frac{\widetilde p_{f,D}}{\widetilde p_{f_0,D}}(X_i^D) - \frac{q_{f,D}}{q_{f_0,D}}(X_i^D) \Big] {\mathbf 1}_{\mathcal A_{i,D}}\nu(df)   \Big) \\
& \quad \lesssim \var_{f_0} \big( \Lambda_{q,D}^N \big) + \sup_{f \in \mathcal C_N}  \mathrm{Var}_{f_0}\left( \sum_{i = 1}^N \Big[ \log \frac{\widetilde p_{f,D}(X_i^D)}{\widetilde p_{f_0,D}(X_i^D)} - \log \frac{q_{f,D}(X_i^D)}{q_{f_0,D}(X_i^D)} \Big] \right) \\
& \quad \lesssim \var_{f_0} \big( \Lambda_{q,D}^N \big) +  N\varepsilon_N^2\Big(\big(\tfrac{\varepsilon_{1,N}}{\varepsilon_N}\big)^2 (D+ND^2) + \big(\tfrac{\varepsilon_{2,N}}{\varepsilon_N}\big)^2 ND^2 + \big(\tfrac{\varepsilon_{3,N}}{\varepsilon_N}\big)^2 ND^3 \\
 & \qquad \qquad \qquad \qquad \qquad \qquad +N\varepsilon_N^{-2}D^4 + \big(\tfrac{\varepsilon_{1,N}}{\varepsilon_N}\big)^2 ND\exp(-cD^{-1}) \Big),
\end{align*}
as required.
\end{proof}

\begin{proof}[Proof of Proposition \ref{prop:tildep_to_q_expectation}]
The proof follows similar, though easier, lines to that of Proposition \ref{prop:tildep_to_q}. We thus provide only a sketch of the main bounds, matching Steps 1-4 for convenience.\\

\noindent {\it Step 1:} Consider first the term $\gamma_{f_0,f}(X_{0})\cdot (X_{D}-X_{0})$ in the expansion provided by Lemma \ref{thm: first big approx}, and where we can insert the term ${\bf 1}_{{\mathcal A}_{1,D}}$ due to the support of $\gamma_{f_0,f}$. In the decomposition \eqref{eq: def reflect}, the term  $\gamma_{f_0,f}(X_{0})\Sigma_{0,D}$ has $\E_{f_0}$-expectation zero and so does not contribute. It follows that
\begin{equation*}
\E_{f_0}\big[\gamma_{f_0,f}(X_{0})\cdot (X_{D}-X_{0})\big] \leq \|\gamma_{f_0,f}\|_\infty (\E_{f_0}[|b_{0,D}|]+\E_{f_0}[L_{0,D} {\bf 1}_{{\mathcal A}_{1,D}}]) \lesssim \varepsilon_{1,N} \big(D+D^{1/2}\exp(-cD^{-1})\big),
\end{equation*}
where we used the moment bounds \eqref{eq: moment bounds reflect}.\\ 

\noindent {\it Step 2:} Write $\zeta_{f_0,f} =  \nabla(f^{-1}-f_0^{-1})$ as before. In a similar way to Step 2 of the proof of Proposition \ref{prop:tildep_to_q}, we have the decompositions
\begin{align*}
&\frac{1}{8D} |X_{D}-X_{0}|^2  {\bf 1}_{{\mathcal A}_{1,D}}= \tfrac{d}{4}f_0(X_0)+\tfrac{1}{8D}\widetilde b_{1,D}{\bf 1}_{{\mathcal A}_{1,D}}+\tfrac{1}{8D}\widetilde \Sigma_{1,D}{\bf 1}_{{\mathcal A}_{1,D}}+\tfrac{1}{8D}\widetilde L_{1,D}{\bf 1}_{{\mathcal A}_{1,D}}\\
&\zeta_{f_0,f}(X_{0}) \cdot (X_{D}-X_{0})  = \zeta_{f_0,f}(X_{0})\cdot b_{1,D}+ \zeta_{f_0,f}(X_{0}) \cdot \Sigma_{1,D}+ \zeta_{f_0,f}(X_{0}) \cdot L_{1,D},
\end{align*}
and note the term $\E_{f_0}\big[\frac{1}{8D} |X_{D}-X_{0}|^2\zeta_{f_0,f}(X_{0}) \cdot (X_{D}-X_{0})\big]$ is the sum of all expectations of cross terms in the two expansions above. Using the bounds \eqref{eq: moment bounds reflect} and \eqref{eq: moment bounds reflect tilde} and the martingale property of $\Sigma_{1,D}$ that ensures $\E_{f_0}\big[\tfrac{d}{4}f_0(X_0)\zeta_{f_0,f}(X_{0}) \cdot \Sigma_{1,D}\big]=0$, we see by Cauchy-Schwarz's inequality for instance that all the expectations of the cross terms are of order at most $(D+D^{1/2}\exp(cD^{-1}))\|\zeta_{f_0,f}\|_\infty = \varepsilon_{1,N}(D+D^{1/2}\exp(-cD^{-1}))$, except maybe for the term
$$\tfrac{1}{8D}\E_{f_0}\big[\widetilde \Sigma_{1,D}\zeta_{f_0,f}(X_{0}) \cdot \Sigma_{1,D}\big].$$
By \eqref{eq: ito dec}, this term exactly equals
$$\tfrac{1}{4D}\E_{f_0}\big[\zeta_{f_0,f}(X_0)\cdot \int_0^D(X_s-X_0)f_0(X_s)ds\big]=\tfrac{1}{4D}\E_{f_0}\big[\overline{b}_{1,D}(G)+\overline{L}_{1,D}(G)\big]$$
according to \eqref{eq: ito second level} with $G(x,y) = \zeta_{f_0,f}(x) \cdot (y-x)f_0(y)
$, and this term finally has the right order since $\E_{f_0}\big[|\overline b_{1,D}(G)|\big] \lesssim \varepsilon_{1,N}D^{2}$ and $\E_{f_0}\big[|\overline L_{1,D}(G)|{\bf 1}_{\mathcal A_{1,D}}\big] \lesssim \varepsilon_{1,N}D^{3/2}\exp(-cD^{-1})$.\\


\noindent {\it Step 3:} With the notation in the proof of Proposition \ref{prop:tildep_to_q}, Step 3, it suffices to bound $\E_{f_0}[\mathcal R_{f_0,f}^1]$. By Lemma \ref{lem: moment bis}, the property $f \in \C_N$ and the fact that $\|r_{f_0,f}\|_\infty \lesssim 1$ again, we readily obtain
\begin{align*}
\E_{f_0}\big[ \mathcal R_{f_0,f}^1\big] & \lesssim \eps_{2,N} D+ \eps_{3,N} D^{3/2}+D^2.
\end{align*}

\noindent {\it Step 4:}  Putting together the above bounds and keeping track of the leading order terms gives 
\begin{align*}
& \sup_{f \in \C_N} \E_{f_0} \Big[ \log \frac{\widetilde p_{f_0,D}(X_0^D)}{\widetilde p_{f,D}(X_0^D)}{\mathbf 1}_{\mathcal A_{1,D}}\Big] \\
& \lesssim \sup_{f\in \C_N} \E_{f_0} \Big[ \log \frac{q_{f_0,D}(X_0^D)}{q_{f,D}(X_0^D)} {\mathbf 1}_{\mathcal A_{1,D}} \Big]  \\
& \qquad +  \varepsilon_N^2 \Big(\tfrac{\eps_{1,N} }{\eps_N^2}D+ \tfrac{\eps_{2,N}}{\eps_N^2}D + \tfrac{\eps_{3,N}}{\eps_N^2}D^{3/2} + \tfrac{\eps_{1,N}}{\eps_N^2} D^{1/2} \exp(-cD^{-1}) \Big),
\end{align*}
which completes the  proof of Proposition \ref{prop:tildep_to_q_expectation}. 
\end{proof}

\subsection{Proofs of Propositions \ref{thm: variance proxy} and \ref{thm: expectation proxy}: expectation and variance of the log-likelihood with proxy density
}\label{sec:proxy_variance}

Recall that the proxy $q_{f,D}$ of the transition density $\widetilde p_{f,D}$ is given by
\begin{equation*}
q_{f,D}(x,y) = \frac{1}{(4\pi D f(x))^{d/2}} \exp\Big(-\frac{|y-x|^2}{4Df(x)}\Big),
\end{equation*}
which is the density function of a $N_d(x,2Df(x) I_d)$ distribution, formally obtained by taking an Euler scheme without drift. Note that while $q_{f,D}$ is defined on $\R^d\times \R^d$, it will be evaluated at $(X_{(i-1)D}, X_{iD})$, which lies in $\mathcal O \times \mathcal O$ almost surely.

\begin{proof}[Proof of Proposition \ref{thm: variance proxy}]

\noindent {\it Step 1:} We look for a simple expansion of the approximate log-likelihood with the proxy. Write
\begin{align*}
\log \frac{q_{f,D}(x,y)}{q_{f_0,D}(x,y)} & = \frac{d}{2}\log \frac{f_0(x)}{f(x)}-\frac{1}{4D}\Big(\frac{1}{f(x)}-\frac{1}{f_0(x)}\Big)|y-x|^2 \\
& =  \frac{d}{2}\Big(\log \frac{f_0(x)}{f(x)}-\Big(\frac{1}{f_0(x)}-\frac{1}{f(x)}\Big)\frac{|y-x|^2}{2dD}\Big).
\end{align*}
Recall the expansion \eqref{eq: expand square} and \eqref{eq:tilde_decomp} in the proof of Proposition \ref{prop:tildep_to_q} that takes the form
$$|X_{iD}-X_{(i-1)D}|^2 = 2 d\,f_0(X_{(i-1)D})D+\widetilde b_{i,D}+\widetilde \Sigma_{i,D}+\widetilde L_{i,D}$$
for appropriate drift, diffusion and boundary remainder terms $\widetilde{b}_{i,D}, \widetilde{\Sigma}_{i,D}$ and $\widetilde{L}_{i,D}$, respectively. Therefore, setting $\xi_{f_0,f} = \frac{1}{f}-\frac{1}{f_0}$, we obtain
\begin{align}
 \frac{2}{d} \log \frac{q_{f,D}(X_i^D)}{q_{f_0,D}(X_i^D)}& = \log \frac{f_0(X_{(i-1)D})}{f(X_{(i-1)D})}-\Big(\frac{f_0(X_{(i-1)D})}{f(X_{(i-1)D})}-1\Big) \nonumber \\
 & \qquad  +\frac{1}{2dD}\xi_{f_0,f}(X_{(i-1)D})\big(\,\widetilde b_{i,D}+\widetilde \Sigma_{i,D}+\widetilde L_{i,D}\big). \label{eq: expansion proxy var}
\end{align}
Since $|\log \kappa-(\kappa-1)| \leq C(\kappa-1)^2$ in a neighbourhood of $\kappa=1$, the property $\|f-f_0\|_{\infty} \leq \eps_N$ and $f \geq f_{\min}$ ensures
$$\big\| \log \frac{f_0}{f}-\Big(\frac{f_0}{f}-1\Big)\big\|_\infty \lesssim \varepsilon_N^2.$$
By Lemma \ref{lem: variance esti}, this entails
$$\mathrm{Var}_{f_0}\Big(\sum_{i = 1}^N\Big(\log \frac{f_0(X_{(i-1)D})}{f(X_{(i-1)D})}-\big(\frac{f_0(X_{(i-1)D})}{f(X_{(i-1)D})}-1\big)\Big){\bf 1}_{\mathcal A_{i,D}}\Big) \lesssim ND^{-1} \varepsilon_N^4.$$

\noindent {\it Step 2:} We next  consider the term
$\sum_{i = 1}^N\frac{1}{2dD}\xi_{f_0,f}(X_{(i-1)D})\widetilde b_{i,D}{\bf 1}_{\mathcal A_{i,D}}$ and proceed similarly to Step 2 in the proof of Proposition \ref{prop:tildep_to_q} above.
Define $\widetilde G(x,y) = 2d(f_0(x)-f_0(y))+2b(\nabla f_0(y),y)\cdot (x-y)$ so that 
\begin{equation*} \label{eq: decomp b tilde}
\widetilde  b_{i,D} = \int_{(i-1)D}^{iD} \widetilde G(X_{(i-1)D},X_s)ds =  \overline{b}_{i,D}(\widetilde G)+\overline{\Sigma}_{i,D}(\widetilde G)+\overline{L}_{i,D}(\widetilde G),
\end{equation*}
applying It\^o's formula again, 
where 
\begin{align*}
&\E_{f_0}\big[|\overline b_{i,D}(\widetilde G)|^p\big] \lesssim D^{2p},\quad \E_{f_0}\big[|\overline \Sigma_{i,D}(\widetilde G)|^p| \mathcal F_{(i-1)D}\big] \lesssim D^{3p/2},\\
&\;\;\E_{f_0}\big[|\overline L_{i,D}(\widetilde G)|^p{\bf 1}_{\mathcal A_{i,D}}\big] \lesssim D^{3p/2}\exp(-cD^{-1})
\end{align*}
in the same way as before, uniformly in $b \in \mathfrak B$.  

The variance of $\sum_{i = 1}^N\frac{1}{2dD}\xi_{f_0,f}(X_{(i-1)D})(\overline{b}_{i,D}(\widetilde G)+r_{i,D})\,{\bf 1}_{\mathcal A_{i,D}}$ is of order $D^{-2}N^2 \varepsilon_{N}^2D^4$. Also, the $\overline{\Sigma}_{i,D}(\widetilde G)$ are centered $\mathcal F_{iD}$-increments hence
\begin{align*}
\mathrm{Var}_{f_0}\Big(\frac{1}{2dD}\sum_{i = 1}^N\xi_{f_0,f}(X_{(i-1)D})\overline{\Sigma}_{i,D}(\widetilde G)\,{\bf 1}_{\mathcal A_{i,D}}\Big) 
& = \frac{1}{4d^2D^2}\sum_{i = 1}^N\E_{f_0}\big[ |\xi_{f_0,f}(X_{(i-1)D}) \overline \Sigma_{i,D} |^2{\bf 1}_{\mathcal A_{i,D}}\big] \\
& \lesssim D^{-2}N\varepsilon_{N}^2D^3
\end{align*}
and this term is of order $N\varepsilon_{N}^2D$. The variance of $(2dD)^{-1}\sum_{i = 1}^N\xi_{f_0,f}(X_{(i-1)D}) \cdot \overline{L}_{i,D}(\widetilde G)\,{\bf 1}_{\mathcal A_{i,D}}$ is of order $D^{-2}N^2 \varepsilon_{N}^2D^3\exp(-cD^{-1})$.  Gathering all these estimates, we obtain
$$\mathrm{Var}_{f_0}\big(\sum_{i = 1}^N\frac{1}{2dD}\xi_{f_0,f}(X_{(i-1)D})\widetilde b_{i,D}{\bf 1}_{\mathcal A_{i,D}}\Big) \lesssim N\varepsilon_N^2\big(ND^2+D+ND\exp(-cD^{-1})\big).$$

\noindent {\it Step 3:} We now consider the martingale term $\sum_{i = 1}^N\frac{1}{2dD}\xi_{f_0,f}(X_{(i-1)D})\widetilde \Sigma_{i,D}{\bf 1}_{\mathcal A_{i,D}}$. We have
\begin{align*}
\mathrm{Var}_{f_0}\Big(\frac{1}{2dD}\sum_{i = 1}^N\xi_{f_0,f}(X_{(i-1)D})\widetilde \Sigma_{i,D}{\bf 1}_{\mathcal A_{i,D}}\Big) & = \frac{1}{4d^2D^2}\sum_{i = 1}^N\E_{f_0}\big[\big(\xi_{f_0,f}(X_{(i-1)D})\widetilde \Sigma_{i,D}\big)^2{\bf 1}_{\mathcal A_{i,D}}\big] \\
& \lesssim D^{-2}N\varepsilon_N^2D^2
\end{align*}
and this term has the right order $N \varepsilon_N^2$.\\

\noindent {\it Step 4:} We finally consider the remainder term $\sum_{i = 1}^N\frac{1}{2dD}\xi_{f_0,f}(X_{(i-1)D})\widetilde L_{i,D}{\bf 1}_{\mathcal A_{i,D}}$ and conclude. We have
\begin{align*}
\mathrm{Var}_{f_0}\Big(\frac{1}{2dD}\sum_{i = 1}^N\xi_{f_0,f}(X_{(i-1)D})\widetilde L_{i,D}{\bf 1}_{\mathcal A_{i,D}}\Big) & \lesssim D^{-2}N\sum_{i = 1}^N\|\xi_{f_0,f}\|_\infty^2\E_{f_0}\big[\widetilde L_{i,D}^2{\bf 1}_{\mathcal A_{i,D}}\big] \\
& \lesssim N^2\varepsilon_N^2\exp(-cD^{-1}).
\end{align*}
Gathering the estimates from Steps 1-4 completes the proof of Proposition \ref{thm: variance proxy}.
\end{proof}

\begin{proof}[Proof of Proposition \ref{thm: expectation proxy}]
The proof follows similar, though easier, lines to that of the proof of Proposition \ref{thm: variance proxy}. We thus provide only a sketch of the main bounds, matching Steps 1-4  for convenience.\\

\noindent {\it Step 1}: We again use the expansion \eqref{eq: expansion proxy var} that yields
\begin{align*}
 \frac{2}{d} \log \frac{q_{f,D}(X_i^D)}{q_{f_0,D}(X_i^D)}& = \log \frac{f_0(X_{(i-1)D})}{f(X_{(i-1)D})}-\Big(\frac{f_0(X_{(i-1)D})}{f(X_{(i-1)D})}-1\Big) \nonumber \\
 & +\frac{1}{2dD}\xi_{f_0,f}(X_{(i-1)D})\big(\,\widetilde b_{i,D}+\widetilde \Sigma_{i,D}+\widetilde L_{i,D}\big). 
\end{align*}
The main term is bounded above as before:
$$\big\| \log \frac{f_0}{f}-\Big(\frac{f_0}{f}-1\Big)\big\|_\infty \lesssim \varepsilon_N^2,$$
as follows from $|\log \kappa-(\kappa-1)| \leq C(\kappa-1)^2$ in a neighbourhood of $\kappa=1$, together with the properties $\|f-f_0\|_{\infty} \leq \eps_N$ and $f \geq f_{\min}$.\\

\noindent {\it Step 2:} By using the refinement of Step 2 in the proof of Proposition \ref{thm: variance proxy}, we have
$$\widetilde b_{i,D} = \overline{b}_{i,D}(\widetilde G)+\overline{\Sigma}_{i,D}(\widetilde G)+\overline{L}_{i,D}(\widetilde G),$$
where 
\begin{align*}
&\E_{f_0}[|\overline{b}_{i,D}(\widetilde G)|^p] \lesssim D^{2p}, \qquad \E_{f_0}\big[|\overline{L}_{i,D}(\widetilde G)|^p{\bf 1}_{\mathcal A_{i,D}}\big] \lesssim D^{3p/2}\exp(-cD^{-1}),
\end{align*} 
and $\overline{\Sigma}_{i,D}(\widetilde G)$ is a $\mathcal F_{iD}$-martingale increment. Therefore,
$$\E_{f_0}\Big[ \frac{1}{2dD}\xi_{f_0,f}(X_{(i-1)D})\widetilde b_{i,D}{\bf 1}_{\mathcal A_{i,D}} \Big] \lesssim \varepsilon_N \big(D + D^{1/2}\exp(-cD^{-1})\big).$$

\noindent {\it Step 3:} Since $\widetilde \Sigma_{i,D}$ is a $\mathcal F_{iD}$-increment, we simply have
$$\E_{f_0}\big[\frac{1}{2dD}\xi_{f_0,f}(X_{(i-1)D})\widetilde \Sigma_{i,D}{\bf 1}_{\mathcal A_{i,D}}\big] = 0.$$

\noindent {\it Step 4:} We finally have 
$$\E_{f_0}\big[\big|\frac{1}{2dD}\xi_{f_0,f}(X_{(i-1)D})\widetilde L_{i,D}{\bf 1}_{\mathcal A_{i,D}}\big|\big] \lesssim \eps_N D^{1/2}\exp(cD^{-1}).$$
Gathering the estimates from Steps 1-4 completes the proof of Proposition \ref{thm: expectation proxy}. 
\end{proof}

\section{Proof of Theorem \ref{thm:exp_inequality}: an exponential inequality for a least squares type estimator}\label{sec:proof_exp_ineq}

\subsection{Risk decomposition}

Recall from \eqref{eq:PJ} that $\overline{P}_J f = 1+P_J[f-1]$ for $f\in \F$, where $P_J$ is the wavelet projection onto the space $V_J$ in \eqref{eq:VJ} used to reconstruct functions with support in $\O_0$, such as $f-1$ when $f\in \F$. We will decompose $\|\widehat{f}_N - \overline{P}_J f\|_2$ into martingale, bias and remainder terms that we control on the event 
\begin{equation} \label{eq: def nice event}
\mathcal B_N = \left\{ (1-\kappa)\|g\|_{2}^2 \leq |g|_N^2 \leq (1+\kappa)\|g\|_{2}^2 ~\text{for all } g \in V_J \right\},
\end{equation}
for some fixed $0 < \kappa < 1$, and where
$$|g|_N^2 = \frac{1}{N} \sum_{i = 1}^N g(X_{(i-1)D})^2{\bf 1}_{\mathcal A_{i,D}},\qquad \mathcal A_{i,D} = \{X_{(i-1)D} \in \mathcal O_0^\delta\},$$
denotes a random empirical semi-norm for every continuous function on the $\delta/2$-enlargement $\mathcal O_0^\delta$ of $\O_0$. By classical concentration techniques, we prove in Appendix \ref{sec: proof of good event} the following result.
\begin{lem} \label{lem: compdisccont}
Suppose that $2^J \to \infty$ and $2^{Jd} = o(\sqrt{ND})$ as $N\to\infty$. Then for every $0 < \kappa < 1$, as $N \rightarrow \infty$,
$$\sup_{f \in \mathcal F} \PP_f(\mathcal B_N^c) \rightarrow 0.$$
\end{lem}
By Lemma \ref{lem: compdisccont} we have $\sup_{f \in \mathcal F} \PP_f(\mathcal B_N^c) \rightarrow 0$ as $N \to \infty$ for every $0<\kappa <1$ under the theorem hypotheses, so that we may further restrict to working on $\mathcal B_N$.
We have under $\PP_{f}$
\begin{equation} \label{eq: decomp signal plus bruit}
Y_{i,D} = (2d)^{-1} D^{-1}|X_{iD}-X_{(i-1)D}|^2 = f(X_{(i-1)D})+\mathcal R_{i,D},
\end{equation}
with
\begin{equation} \label{eq: remainder regression}
\mathcal R_{i,D} = (2d)^{-1} D^{-1}\big(\widetilde b_{i,D}+\widetilde \Sigma_{i,D}+\widetilde L_{i,D}\big),
\end{equation}
%
where we recall the above decomposition from \eqref{eq: expand square}-\eqref{eq:tilde_decomp} but with $f_0$ replaced by $f$.
Let 
$$\Gamma_N(g) = \frac{1}{N} \sum_{i = 1}^N\big(Y_{i,D}-1-g(X_{(i-1)D})\big)^2{\bf 1}_{\mathcal A_{i,D}}$$
denote the normalized objective function defining our least squares estimator \eqref{eq:estimator}, using in particular that $g(X_{(i-1)D}) = 0$ on $(\mathcal A_{i,D})^c$ for all $g\in V_J$. By \eqref{eq: decomp signal plus bruit}, for any function $g$,
\begin{align} \label{eq: the rep}
\Gamma_N(g)-\Gamma_N(f-1) = |1+g-f|_N^2+\frac{2}{N} \sum_{i = 1}^N(f-1-g)(X_{(i-1)D})\mathcal R_{i,D}{\bf 1}_{\mathcal A_{i,D}}.
\end{align}
Since $\Gamma_N(\widehat g_N)  = \inf_{g \in V_J}\Gamma_N(g)$ by construction, we have $\Gamma_N(\widehat g_N)-\Gamma_N(f-1) \leq  \Gamma_N(P_J[f-1]) -  \Gamma_N(f-1)$. Substituting \eqref{eq: the rep} into both sides of this inequality with $g = \widehat{g}_N$ and $g=P_J[f-1]$, respectively, then yields (with $\widehat f_N = \widehat g_N +1$)
\begin{align*}
&|\widehat f_N-f|_N^2+2N^{-1}\sum_{i = 1}^N(f-\widehat f_N)(X_{(i-1)D})\mathcal R_{i,D}{\bf 1}_{\mathcal A_{i,D}} \\
&\leq  |1+P_J[f-1]-f|_N^2+2N^{-1}\sum_{i = 1}^N(f-1-P_J[f-1])(X_{(i-1)D})\mathcal R_{i,D}{\bf 1}_{\mathcal A_{i,D}},
\end{align*}
or equivalently 
\begin{align}
|\widehat f_N-f|_N^2 & \leq  |P_J[f-1]-(f-1)|_N^2+2N^{-1}\sum_{i = 1}^N(\widehat f_N - 1-P_J[f-1])(X_{(i-1)D})\mathcal R_{i,D}{\bf 1}_{\mathcal A_{i,D}}. \label{eq: quasi}
\end{align}
Using that $\widehat f_N -1-P_J[f-1] \in V_J$ and $2ab \leq \rho a^2+\rho^{-1}b^2$ for any $a,b,\rho>0$, on the event $\mathcal B_N$, the second term in the last display equals
\begin{align*}
 &   \frac{2}{N}\sum_{i = 1}^N(\widehat g_N-P_J[f-1])(X_{(i-1)D})\mathcal R_{i,D}{\bf 1}_{\mathcal A_{i,D}}  \\
 & \leq 2\|\widehat g_N-P_J[f-1]\|_{2} \sup_{g \in V_J: \|g\|_{2}=1}\frac{1}{N} \sum_{i = 1}^Ng(X_{(i-1)D})\mathcal R_{i,D}{\bf 1}_{\mathcal A_{i,D}} \\
 & \leq \rho \|\widehat g_N-P_J[f-1]\|_{2}^2+\rho^{-1} \Big(\sup_{g \in V_J: \|g\|_{2} \leq 1}\frac{1}{N} \sum_{i = 1}^Ng(X_{(i-1)D})\mathcal R_{i,D}{\bf 1}_{\mathcal A_{i,D}}\Big)^2 \\
  & \leq  (1-\kappa)^{-1}\rho |\widehat g_N-P_J[f-1]|_N^2+\rho^{-1} \Big(\sup_{g \in V_J: \|g\|_{2} \leq 1}\frac{1}{N} \sum_{i = 1}^Ng(X_{(i-1)D})\mathcal R_{i,D}{\bf 1}_{\mathcal A_{i,D}}\Big)^2.
\end{align*}
Substituting this last display into \eqref{eq: quasi} then gives
\begin{align*}
|\widehat f_N-f|_N^2 & \leq |P_J[f-1]-(f-1)|_N^2+(1-\kappa)^{-1}\rho |\widehat g_N-P_J[f-1]|_N^2\\
&\qquad +\rho^{-1} \Big(\sup_{g \in V_J: \|g\|_{2} \leq 1}\frac{1}{N} \sum_{i = 1}^Ng(X_{(i-1)D})\mathcal R_{i,D}{\bf 1}_{\mathcal A_{i,D}}\Big)^2.
\end{align*}
Using that $\sqrt{a+b} \leq \sqrt{a} + \sqrt{b}$ for $a,b>0$, this implies
\begin{align*}
|\widehat{f}_N-1-P_J[f-1]|_N & \leq |\widehat f_N - f|_N + |f-1-P_J[f-1]|_N \\
& \leq 2|f-1-P_J[f-1]|_N + (1-\kappa)^{-1/2}\rho^{1/2} |\widehat f_N-1-P_J[f-1]|_N \\
&\quad +\rho^{-1/2}\Big|\sup_{g \in V_J: \|g\|_{2} \leq 1}\frac{1}{N} \sum_{i = 1}^Ng(X_{(i-1)D})\mathcal R_{i,D}{\bf 1}_{\mathcal A_{i,D}}\Big|.
\end{align*}
Collecting the $|\widehat{f}_N-1-P_J[f-1]|_N$ terms on the left-hand side thus yields the risk decomposition
\begin{align*}
& \left( 1 - \sqrt{\tfrac{\rho}{1-\kappa}}\right) |\widehat{f}_N-1-P_J[f-1]|_N  \\
& \qquad \qquad \leq 2|f-1-P_J[f-1]|_N  +\rho^{-1/2}\Big|\sup_{g \in V_J: \|g\|_{2} \leq 1}\frac{1}{N} \sum_{i = 1}^Ng(X_{(i-1)D})\mathcal R_{i,D}{\bf 1}_{\mathcal A_{i,D}}\Big|.
\end{align*}
Taking $\rho>0$ small enough that $\rho/(1-\kappa) < 1$ and using that $|\widehat{f}_N-1 - P_J[f-1]|_N \geq \sqrt{1-\kappa}\|\widehat f_ N -1 - P_J[f-1]\|_{2}$ on $\mathcal B_N$, we then obtain that on $\mathcal B_N$,
\begin{align*}
C^{-1} \|\widehat f_N-1 - P_J[f-1]\|_{2} \leq 2|f-1-P_J[f-1]|_N  +\rho^{-1/2} \Big|\sup_{g \in V_J: \|g\|_{2} \leq 1}\frac{1}{N} \sum_{i = 1}^Ng(X_{(i-1)D})\mathcal R_{i,D}{\bf 1}_{\mathcal A_{i,D}}\Big|,
\end{align*}
where $C^{-1} = \sqrt{1-\kappa} (1- \sqrt{\rho/(1-\kappa)}) = \sqrt{1-\kappa} - \sqrt{\rho}>0$. Recalling that $\overline{P}_J f = 1+P_J[f-1]$, we finally obtain the following risk decomposition:
\begin{align}
 \PP_f\big(\|\widehat f_N - \overline{P}_J f \|_{2} \geq K\xi_N,\mathcal B_N\big) 
& \leq \PP_f\big(|f -1- P_J[f-1]|_N \geq K_1 \xi_N,\mathcal B_N\big) \nonumber \\
& \quad + \PP_f\Big( \sup_{g \in V_J: \|g\|_{2} \leq 1}\sum_{i = 1}^Ng(X_{(i-1)D})\mathcal R_{i,D}{\bf 1}_{\mathcal A_{i,D}} \geq K_2 N\xi_N, \mathcal B_N\Big) \label{eq: controle} 
\end{align}
where the constants $K_i = C_i K$ for constants $C_i$ depending only on $\kappa$ and $\rho$, which are henceforth considered fixed. It therefore suffices to prove that each term on the right-hand side of \eqref{eq: controle} is bounded by a multiple of $e^{-LN\varepsilon_N^2}$ for arbitrary fixed $L>0$ and large enough $K = K(L)$.

\subsection{Bias term}

We consider here the first term in \eqref{eq: controle}, which is a bias term with the randomness only entering from the design in the random empirical semi-norm $|\cdot|_N$. Write $g_J = f-1-P_J[f-1]$ and
$$Z_N(g_J) = |g_J|_N^2 - \|g_J\|_{2}^2 = \frac{1}{N}\sum_{i = 1}^N\big(g_J(X_{(i-1)D})^2-\E_f[g_J(X_{(i-1)D})^2]\big),$$
since $\E_f[g_J(X_{(i-1)D})^2]=\|g_J\|_{2}^2$. The variance satisfies 
$$\var_f(g_J(X_{(i-1)D})^2) \leq \E_f[g_J(X_{(i-1)D})^4] = \|g_J\|_{L^4}^4,$$
while $|g_J(x)^2 - \E_f[g_J(X_{(i-1)D})^2]| \leq 2\|g_J\|_\infty^2$. Since $(X_0,X_D,\dots,X_{ND})$ is a stationary reversible Markov chain whose spectral gap is lower bounded by $rD$ by \eqref{eq:spectral_gap}, Theorem of 3.3 of \cite{Paulin2015} (cf. (3.21)) yields the following Bernstein-type inequality:
\begin{align*}
\PP_f\big( N|Z_N(g_J)| \geq t\big) & \leq 2\exp\left(-\frac{ rD t^2 }{4N\|g_J\|_{L^4}^4+20\|g_J\|_\infty^2 t}\right)
\end{align*}
for all $t>0$. Setting $t = MN \xi_N^2$ then gives
\begin{align*}
\PP_f\big( |Z_N(g_J)| \geq M\xi_N^2 \big) & \leq 2\exp\left(-\frac{ r D M^2 N \xi_N^4 }{4\|g_J\|_{L^4}^4+20M \|g_J\|_\infty^2\xi_N^2 }\right).
\end{align*}
By $L^p$ interpolation, $\|g_J\|_{L^4} \leq \|g_J\|_{L^2}^{1/2} \|g_J\|_{\infty}^{1/2}$. Using this and rearranging, one gets that the right-hand side is smaller than $2e^{-LN\eps_N^2}$ if
$$4L\|g_J\|_{L^2}^2 \|g_J\|_\infty^2 \eps_N^2 + 20LM \|g_J\|_\infty^2 \xi_N^2 \eps_N^2 \leq r D M^2 \xi_N^4.$$
But for any $f\in\F_N'$ contained in the set defined in the theorem statement, the left side above is bounded by a multiple of $R_J^2 \eps_N^2 \xi_N^2 = O(D \xi_N^4)$ under the assumed hypothesis $R_{J}^2  \eps_N^2 \lesssim D \xi_N^2$. Since also $r=r(f_{\min},\O)$ by \eqref{eq:spectral_gap}, this verifies the condition in the last display. In summary, we have shown that for any $L>0$, there exists $M = M(L,r)>0$ large enough such that
\begin{align*}
\sup_{f\in \F_N'}  \PP_f(|f-1-P_J[f-1]|_N^2 \geq \|f-1-P_J[f-1]\|_{2}^2 + M \xi_N^2 ) \leq 2e^{-LN\eps_N^2}
\end{align*}
under the theorem hypotheses.

\subsection{Deviation of the remainder term}
We now consider the second term in \eqref{eq: controle} involving $\mathcal R_{i,D}$. By a union bound, up to a modification of the constants, it suffices to prove the deviation for each term in the decomposition \eqref{eq: remainder regression}.\\

\noindent {\it Step 1:} We postpone the term involving the drift $\widetilde b_{i,D}$ to Step 3 and rather first control the term involving $\widetilde \Sigma_{i,D}$ in \eqref{eq: remainder regression}. We will prove  
\begin{equation} \label{eq: dev martingale}
\PP_{f}\Big(\sup_{g \in V_J:\|g\|_{2}\leq 1}|\widetilde Z_N(g)| \geq M N\xi_N, \mathcal B_N\Big) \leq e^{-LN\varepsilon_N^2},
\end{equation}
for sufficiently large $M$, where 
$$\widetilde Z_N(g) = D^{-1}\sum_{i =1}^{N}g(X_{(i-1)D})\widetilde \Sigma_{i,D}{\bf 1}_{\mathcal A_{i,D}}.$$
Define 
$$\PP_{\mathcal B_N,f} = \PP_f(\cdot|\,\mathcal B_N) = \PP_f(\mathcal B_N)^{-1}\PP_f(\cdot \cap \mathcal B_N).$$
Since $\PP_f(\mathcal B_N) \rightarrow 1$ uniformly in $f \in \mathcal F$ by Lemma \ref{lem: compdisccont}, we have $\tfrac{1}{2} \PP_{\mathcal B_N,f} \leq \PP_f(\cdot \cap \mathcal B_N) \leq \PP_{\mathcal B_N,f}$ for $N$ large enough.
It thus suffices to establish \eqref{eq: dev martingale} with $\PP_{\mathcal B_N,f}$ instead of $\PP_{f}$.
We prove a Bernstein inequality for the increments of the process $(\widetilde Z_N(g) : g \in V_J)$ and  conclude using a chaining argument under $\PP_{\mathcal B_N,f}$. To that end, we need the following Bernstein inequality, see {\it e.g.} \cite{pinelis1994optimum, victor1999general}.

\begin{lem} \label{lem: bernstein mg}
Let $(\mathcal M_n)_{n \geq 0}$ be a $(\mathcal G_n)_{n \geq 0}$-martingale with $\mathcal M_0=0$ and let $\langle \mathcal M\rangle_n = \sum_{i=1}^n  \E[|\mathcal M_i-\mathcal M_{i-1}|^2\,|\,\mathcal G_{i-1}]$ denotes its predictable quadratic variation. If for some $c >0$,
$$\sum_{i = 1}^n \E[|\mathcal M_i-\mathcal M_{i-1}|^p\,|\,\mathcal G_{i-1}] \leq \frac{c^{p-2}p!}{2} \langle \mathcal M\rangle_n  \qquad \text{for every } p \geq 2,$$
then for every $t,y> 0$,
$$\PP\big(\mathcal M_n \geq t,\langle \mathcal M\rangle_n  \leq y\big) \leq  \exp\Big(-\frac{t^2}{2(y+ct)}\Big).$$
\end{lem}  

We will apply Lemma \ref{lem: bernstein mg} to the $\mathcal G_{n}$-martingale 
$\mathcal M_{n} = \widetilde Z_N(g) = D^{-1}\sum_{i =1}^{n}g(X_{(i-1)D}) \widetilde \Sigma_{i,D}{\bf 1}_{\mathcal A_{i,D}}$, 
with $\mathcal G_n = \sigma(X_{iD}: 0 \leq i \leq n)$. First, applying the Burkholder-Davis-Gundy inequality with best  constant $c_\star^{p/2}p^{p/2}$ (see {\it e.g.} \cite{barlow1982semi}), we have for all $p\geq 1$:
\begin{align*}
&\E_f\big[|D^{-1}\widetilde \Sigma_{i,D}{\bf 1}_{\mathcal A_{i,D}}|^p\,|\,\mathcal F_{(i-1)D}]\\
 & = D^{-p} \E_f\Big[\big|\int_{(i-1)D}^{iD}\sqrt{2f(X_s)}(X_s-X_{(i-1)D})\cdot dB_s\big|^p\,\big|\,\mathcal F_{(i-1)D}\Big]{\bf 1}_{\mathcal A_{i,D}}\\
& \leq D^{-p}  (2\|f\|_\infty c_\star p)^{p/2} \E_f\Big[\big|\int_{(i-1)D}^{iD}|X_s-X_{(i-1)D}|^2ds\big|^{p/2}\,\big|\,\mathcal F_{(i-1)D}\Big] \\
& \leq (D^{-1} 2\|f\|_\infty c_\star p)^{p/2} \max_{1 \leq i \leq d}\E_f\Big[\sup_{(i-1)D \leq s \leq iD}|X_s-X_{(i-1)D}|^{p}\,|\,\mathcal F_{(i-1)D}\Big] \\
& \leq  (2\|f\|_\infty c_\star p)^{p/2}2^{p-1}(\|f\|_{\mathcal C^1}^pD^{p/2}+(c_\star p)^{p/2}\|f\|_\infty^p) \\
& \leq (C \|f\|_{\C^1}^{3/2})^p p!,
\end{align*}
for some $C >0$ that only depends on $d$, since $D\to 0$ and
where we used  Lemma \ref{lem: moment bis} to obtain the second last inequality together with $p^p \leq (C')^p p!$ for some universal constant $C' \geq 1$.
Next, we claim that the following lower bound holds
\begin{equation} \label{eq: stoch lb}
\E_f\big[(D^{-1}\widetilde \Sigma_{i,D})^2 {\bf 1}_{\mathcal A_{i,D}}\,|\,\mathcal F_{(i-1)D}\big] \geq  df_{\min}^2-O(D^{1/2}).
\end{equation}
Assuming for now that \eqref{eq: stoch lb} is true, for $D \to 0$ small enough,
\begin{align*}
\sum_{i =1}^{n}|g(X_{(i-1)D})|^p  & \leq \|g\|_\infty^{p-2} \sum_{i =1}^{n}|g(X_{(i-1)D})|^{2} \tfrac{2}{f_{\min}^2d}\E_f\big[( \tfrac{1}{D} \widetilde \Sigma_{i,D})^2 {\bf 1}_{\mathcal A_{i,D}}\,|\,\mathcal F_{(i-1)D}\big] \\
& = \tfrac{2}{f_{\min}^2d}\|g\|_{\infty}^{p-2} \langle \mathcal M \rangle_n.
\end{align*}
It follows that
\begin{align*}
\sum_{i =1}^{n}\E\big[||\mathcal M_i-\mathcal M_{i-1}|^p\,|\,\mathcal F_{(i-1)D}] & = 
\sum_{i =1}^{n}|g(X_{(i-1)D})|^p \E_f\big[|D^{-1}\widetilde \Sigma_{i,D}{\bf 1}_{\mathcal A_{i,D}}|^p\,|\,\mathcal F_{(i-1)D}] \\
& \leq (C\|f\|_{\C^1}^{3/2})^p p! \sum_{i =1}^{n}|g(X_{(i-1)D})|^p\\
& \leq \tfrac{2}{f_{\min}^2d}  (C\|f\|_{\C^1}^{3/2})^p p! \|g\|_{\infty}^{p-2} \langle \mathcal M \rangle_n
\end{align*}
for small enough $D$. The condition of Lemma \ref{lem: bernstein mg} therefore holds with 
$$c = C_{f_{\min}, \|f\|_{\mathcal C^1}} = \max \left( 1, 4C^2 \|f\|_{\C^1}^2 d^{-1} f_{\min}^{-2} \right)C \|f\|_{\mathcal C^1}^{3/2}\|g\|_\infty.$$
Moreover, recalling that $g \in V_J$ in \eqref{eq: controle} and using the bound \eqref{eq: moment bounds reflect tilde}, on the event $\mathcal B_N$, 
$$\langle \mathcal M\rangle_N = \frac{1}{D^2} \sum_{i=1}^N g(X_{(i-1)D})^2 \E_f [ \widetilde \Sigma_{i,D}^2 | \F_{(i-1)D}] {\bf 1}_{\mathcal A_{i,D}} \leq C_{\|f\|_{\C^1}} N|g|_N^2 \leq C_{\|f\|_{\C^1}} (1+\kappa)N \|g\|_2^2 .
$$
Applying Lemma \ref{lem: bernstein mg} at $n=N$ with $y= C_{\|f\|_{\mathcal C^1}} (1+\kappa)N \|g\|_{2}^2 $, we finally obtain for all $t>0$,
\begin{align*} 
\PP_{\mathcal B_N,f}\big(\big|\widetilde Z_N(g) \big|\geq t\big) & \leq 2 \PP_{f}\big(\big|\mathcal M_N\big|\geq t, \mathcal B_N\big)\\
& \leq 2\PP_{f}\big(\big|\mathcal M_N\big|\geq t, \langle \mathcal M\rangle_N \leq C_{\|f\|_{\mathcal C^1}}N\|g\|_{2}^2\big)\\
&\leq 4\exp\Big(-\frac{t^2}{2\big(C_{\|f\|_{\mathcal C^1}}N\|g\|_{2}^{2}+C_{\|f\|_{\C^1}, f_{\min}}\|g\|_\infty t \big)}\Big).
\end{align*}
Using that $\|g\|_\infty \leq C 2^{Jd/2} \|g\|_{2}$ for $g \in V_J$, we have that for all $g,g' \in V_J$ and $t >0$,
$$\PP_{\mathcal B_N,f}\big(|\widetilde Z_N(g)-\widetilde Z_N(g')| \geq t\big) \leq 4 \exp\Big(-\frac{Ct^2}{N\|g-g'\|_{2}^2+2^{Jd/2} \|g-g'\|_{2} t}\Big).$$
We may therefore apply the generic chaining bound in Lemma \ref{lem:Dirksen} in Section \ref{sec: generic chaining} below with $m = \dim(V_J) = O(2^{Jd})$, $\|\mG_J\|_2 = \sup_{g,h \in \mG_J}\|g-h\|_2 \leq 2$, $\alpha^2 = N$, $\beta=2^{Jd/2}$ to obtain that for all $u \geq 1$,
\begin{equation*}
\PP_{\mathcal B_N,f} \Big( \sup_{g \in V_J: \|g\|_{2} \leq 1} |\widetilde Z_N(g)| \geq C_{\|f\|_{C^1},f_{\min}} \big( 2^{3Jd/2} + 2^{Jd/2} N^{1/2} + 2^{Jd/2} u + N^{1/2} \sqrt{u} \big) \Big) \leq e^{-u}.
\end{equation*}
Setting $u = LN\eps_N^2$ with $L>0$ then gives
\begin{equation*} 
\PP_{\mathcal B_N,f}\Big(\sup_{g \in V_J:\|g\|_{2}\leq 1}|\widetilde Z_N(g)| \geq M N\xi_N\Big) \leq e^{-LN\varepsilon_N^2}
\end{equation*}
for any $\eps_N,\xi_N,2^J$ satisfying
\begin{equation}\label{eq:rate_conditions}
2^{3Jd/2}N^{-1} + 2^{Jd/2} N^{-1/2} + 2^{Jd/2} \eps_N^2 + \eps_N \lesssim \xi_N
\end{equation}
upon taking $M = M(L,\|f\|_{\C^1},f_{\min})$ large enough. This establishes \eqref{eq: dev martingale}, so that we have the required exponential inequality  in \eqref{eq: controle} for the term $(dD)^{-1}\widetilde \Sigma_{i,D}$ in $\mathcal R_{i,D}$.\\

It remains to prove \eqref{eq: stoch lb}. By It\^o's isometry, 
\begin{align*}
\E_f\big[(D^{-1}\widetilde \Sigma_{i,D})^2 {\bf 1}_{\mathcal A_{i,D}}\,|\,\mathcal F_{(i-1)D}\big] & = \frac{8}{D^2}\E_f\Big[\int_{(i-1)D}^{iD}|X_s-X_{(i-1)D}|^2 f(X_s)ds\,|\,\mathcal F_{(i-1)D}\Big]{\bf 1}_{\mathcal A_{i,D}} \\
& \geq\frac{8f_{\min}}{D^{2}}\E_f\Big[\int_{(i-1)D}^{(i-1)D+\tau_{i,D}\wedge D}|X_s-X_{(i-1)D}|^2ds\,|\,\mathcal F_{(i-1)D}\Big]{\bf 1}_{\mathcal A_{i,D}} ,
\end{align*}
where $\tau_{i,D}$ defined in \eqref{eq: def hitting}
is the hitting time of the boundary $\partial \mathcal O$ by the process $X$ started at time $(i-1)D$. Furthermore,
\begin{align}
& \E_f\Big[\int_{(i-1)D}^{(i-1)D+\tau_{i,D}\wedge D}|X_s-X_{(i-1)D}|^2ds\,|\,\mathcal F_{(i-1)D}\Big]  \nonumber \\
&= \E_f\Big[\int_{(i-1)D}^{(i-1)D+\tau_{i,D}\wedge D}\int_{(i-1)D}^s d f(X_u)duds\,|\,\mathcal F_{(i-1)D}\Big] \nonumber\\
&\quad +2\E_f\Big[\int_{(i-1)D}^{(i-1)D+\tau_{i,D}\wedge D}\int_{(i-1)D}^s(X_s-X_{(i-1)D})\cdot \nabla f(X_u)duds\,|\,\mathcal F_{(i-1)D}\Big]. \label{eq: last ito}
\end{align}
Lemma \ref{lem: moment bis} applied to the second term of the right-hand side of \eqref{eq: last ito} yields
\begin{align*}
\big|2\E_f\big[\int_{(i-1)D}^{(i-1)D+\tau_{i,D}\wedge D}\int_{(i-1)D}^s(X_s-X_{(i-1)D})\cdot \nabla f(X_u)duds\,|\,\mathcal F_{(i-1)D}\big]\big| \lesssim D^{2+1/2}
\end{align*}
up to a (deterministic) constant that depends only on $\|f\|_{\mathcal C^1}$, while the first term of the right-hand side of \eqref{eq: last ito} times ${\bf 1}_{\mathcal A_{i,D}}$ can bounded below by $ d f_{\min} $ times
\begin{align*}
 \E_f\big[\big(\tau_{i,D}\wedge D \big)^2{\bf 1}_{\{\tau_{i,D} \geq D\}}\,|\,\mathcal F_{(i-1)D}\big]{\bf 1}_{\mathcal A_{i,D}} 
& =  D^2\PP_f\big(\tau_{i,D} \geq D\,|\,\mathcal F_{(i-1)D}\big){\bf 1}_{\mathcal A_{i,D}} \\
& \geq D^2(1-C\exp(-cD^{-1})){\bf 1}_{\mathcal A_{i,D}} 
\end{align*}
thanks to Lemma \ref{lem: hitting time}. We thus have 
$$\E_f\Big[\int_{(i-1)D}^{(i-1)D+\tau_{i,D}\wedge D}|X_s-X_{(j-1)D}|^2ds\,|\,\mathcal F_{(i-1)D}\Big] \geq d D^2 f_{\min}-O(D^{5/2})$$
and \eqref{eq: stoch lb} readily follows.\\

\noindent {\it Step 2:} We next turn to the boundary term in \eqref{eq: remainder regression}. Set $\widetilde L_N(g) = D^{-1}\sum_{i = 1}^Ng(X_{(i-1)D})\widetilde L_{i,D}{\bf 1}_{\mathcal A_{i,D}}$. Since $\sup_{g \in V_J, \|g\|_{2} \leq 1}|g|_N^2 \leq (1+\kappa)$ on $\mathcal B_N$,
by Cauchy-Schwarz's inequality,
\begin{align*}
\PP_f\Big( \sup_{g \in V_J, \|g\|_{2} \leq 1}|\widetilde L_N(g)| \geq CN\xi_N, \mathcal B_N\Big) 
 & \leq  \PP_f\big((1+\kappa)D^{-2}N^{-1}\sum_{i = 1}^{N}\widetilde L_{i,D}^2{\bf 1}_{\mathcal A_{i,D}}\geq C\xi_N^2\big) \\
& \leq \sum_{i = 1}^{N}\PP_f\big(\tau_{i,D}  \geq D, \mathcal A_{i,D}\big) \lesssim N\exp(-cD^{-1})
\end{align*}
using that $\widetilde L_{i,D}$ vanishes on $\{\tau_{i,D} < D\} \cap \mathcal A_{i,D}$ and  Lemma \ref{lem: hitting time}.\\

\noindent {\it Step 3:} We finally consider the drift term in \eqref{eq: remainder regression}. Setting $\widetilde b_N(g) = D^{-1}\sum_{i = 1}^Ng(X_{(i-1)D})\widetilde b_{i,D}{\bf 1}_{\mathcal A_{i,D}}$, we will prove
\begin{equation} \label{eq: concentration remainder}
\PP_f\Big( \sup_{g \in V_J: \|g\|_{2} \leq 1}|\widetilde b_{N}(g)| \geq MN\xi_N, \mathcal B_N\Big) \leq C\exp(-LN\varepsilon_N^2)
\end{equation}
for arbitrary fixed $L>0$ and $M=M(L)$ large enough, under the assumption $ND^2\rightarrow 0$. By It\^o's formula, with $H(x,y)  = d(f(y)-f(x))+(y-x)\cdot \nabla f(y)$,
similarly to \eqref{eq: ito second level} in Step 2 of the proof of Proposition \ref{prop:tildep_to_q}, we can decompose
\begin{align*}
\widetilde b_{i,D}  & =\int_{(i-1)D}^{iD} H(X_{(i-1)D}, X_s)ds  = \overline b_{i, D}(H)+\overline \Sigma_{i,D}(H)+\overline L_{i, D}(H),
\end{align*}
with
\begin{align*}
\overline b_{i, D}(H) & =  \int_{(i-1)D}^{iD} \int_{(i-1)D}^s \mathcal L_fH(X_{(i-1)D}, X_u)du\,ds\\
\overline \Sigma_{i,D}(H) & = \int_{(i-1)D}^{iD}\int_{(i-1)D}^s\nabla H(X_{(i-1)D}, X_u)\sqrt{2f(X_u)} \cdot dB_uds\\
\overline L_{i, D}(H)& =  \int_{(i-1)D}^{iD} \int_{(i-1)D}^s \nabla H(X_{(i-1)D}, X_u) \cdot n(X_u)d|\ell|_u.
\end{align*} 
We bound the deviation probability of each contribution in the expansion of $\widetilde b_N(g)$ via a union bound. Using Cauchy-Schwarz's inequality and the bound
$$\big|\int_{(i-1)D}^{iD} \int_{(i-1)D}^s \mathcal L_fH(X_{(i-1)D}, X_u)duds\big| \leq \tfrac{1}{2}D^2\|\mathcal L_fH\|_\infty \leq C_{\mathcal O, \|f\|_{\mathcal C^3}}D^2,$$
where the supremum is taken over $(x,y) \in \mathcal O \times \mathcal O$, we derive
\begin{align*}
&\PP_f\Big( \sup_{g \in V_J, \|g\|_{2} \leq 1}D^{-1}\big|\sum_{i = 1}^Ng(X_{(i-1)D})\overline b_{i,D}(H) {\bf 1}_{\mathcal A_{i,D}}\big| \geq CN \xi_N, \mathcal B_N\Big) \\
& \leq \PP_f\big(N^{-1}\sum_{i = 1}^N (\overline b_{i,D}(H))^2 \geq C^2D^2\xi_N^2 \big)  \leq N {\bf 1}_{\big\{C_{\mathcal O, \|f\|_{\mathcal C^3}}^2D^2 \geq  C\xi_N\big\}}
\end{align*}
where we used that $\sup_{g \in V_J, \|g\|_{2} \leq 1}|g|_N^2 \leq (1+\kappa)$ on $\mathcal B_N$.
The assumptions $ND^2\rightarrow 0$ and $N\xi_N^2 \rightarrow \infty$ together imply that for sufficiently large $N$, we necessarily have $C_{\mathcal O, \|f\|_{\mathcal C^3}}^2D^2 < C\xi_N^2$ and the above probability is thus zero. For the martingale term associated to $\overline \Sigma_{i,D}(H)$, we proceed exactly as in Step 1: define
$$\overline{\mathcal M}_n = D^{-1}\sum_{i = 1}^n g(X_{(i-1)D})\overline{\Sigma}_{i,D}(H) {\bf 1}_{\mathcal A_{i,D}}.$$
By Fubini's theorem and the Burkholder-Davis-Gundy inequality:
\begin{align*}
&\E_f\big[|D^{-1}\overline \Sigma_{i,D}(H){\bf 1}_{\mathcal A_{i,D}}|^p\,|\,\mathcal F_{(i-1)D}]\\
 & = D^{-p}\E_f\Big[\big|\int_{(i-1)D}^{iD}(u-(i-1)D)\nabla H(X_{(i-1)D}, X_u)\sqrt{2f(X_u)} \cdot dB_u\big|^p\,\big|\,\mathcal F_{(i-1)D}\Big]{\bf 1}_{\mathcal A_{i,D}}ds\\
&\leq D^{-p} (c_\star p)^{p/2} \E_f\Big[\big|\int_{(i-1)D}^{iD}(u-(i-1)D)^2|\nabla H(X_{(i-1)D}, X_u)|^22f(X_u)ds\big|^{p/2}\,\big|\,\mathcal F_{(i-1)D}\Big]{\bf 1}_{\mathcal A_{i,D}}ds \\
& \leq D^{p/2-1}(c_\star p)^{p/2} \||\nabla H|^2 2f\|_\infty^{p/2-1}\E_f\big[(D^{-1}\overline \Sigma_{i,D}(H){\bf 1}_{\mathcal A_{i,D}})^2\,|\,\mathcal F_{(i-1)D}]\\\
& \leq (D^{1/2}C_{\|f\|_{\mathcal C^2}})^{p-2}p!\,\E_f\big[(D^{-1}\overline \Sigma_{i,D}(H){\bf 1}_{\mathcal A_{i,D}})^2\,|\,\mathcal F_{(i-1)D}]\\
\end{align*}
with
$C_{\|f\|_{\mathcal C^3}} \geq C'\max(1, (2c_\star\|f\|_\infty)^{1/2}\|\nabla H\|_\infty)$ and $C'$ a universal constant such that $p^p \leq C'p!$.
It follows that
\begin{align*}
\sum_{i =1}^{n}\E\big[||\mathcal M_i-\mathcal M_{i-1}|^p\,|\,\mathcal F_{(i-1)D}] & =\sum_{i = 1}^n|g(X_{(i-1)})|^p \E_f\big[|D^{-1}\overline \Sigma_{i,D}(H){\bf 1}_{\mathcal A_{i,D}}|^p\,|\,\mathcal F_{(i-1)D}] \\
& \leq  (D^{1/2}C_{\|f\|_{\mathcal C^2}}\|g\|_\infty)^{p-2} p! \langle \overline{\mathcal M} \rangle_n,
\end{align*}
so that the condition of Lemma \ref{lem: bernstein mg} holds with $c =  D^{1/2}C_{\|f\|_{\mathcal C^2}}'\|g\|_\infty$ for instance. Moreover, 
$$\langle \overline{\mathcal M}\rangle_N = \sum_{i = 1}^N g(X_{(i-1)D})^2\E_f\big[(D^{-1}\overline \Sigma_{i,D}(H){\bf 1}_{\mathcal A_{i,D}})^2\,|\,\mathcal F_{(i-1)D}] \leq CN|g|_N^2D \leq C(1+\kappa)\|g\|_{2}^2ND$$
on the event $\mathcal B_N$,
so that applying Lemma  \ref{lem: bernstein mg} with $n=N$ and $y=C(1+\kappa)\|g\|_{2}^2ND$  yields
\begin{align*}
\PP_{\mathcal B_N,f}(D^{-1}\sum_{i = 1}^n g(X_{(i-1)D})\overline{\Sigma}_{i,D}(H) {\bf 1}_{\mathcal A_{i,D}} \geq t) & \leq C\PP_f\big(\overline{\mathcal M}_N \geq  t,\langle \overline{\mathcal M}\rangle_N \leq C'ND\|g\|_{2}^2\big) \\
& \leq 2C\exp\Big(-C\frac{t^2}{ND\|g\|_{2}^2+D^{1/2}\|g\|_\infty}\Big)
\end{align*}
for all $t >0$. Setting $\overline{Z}_N(g) =D^{-1}\sum_{i = 1}^n g(X_{(i-1)D})\overline{\Sigma}_{i,D}(H) {\bf 1}_{\mathcal A_{i,D}}$, the process $\overline{Z}_N = (\overline{Z}_N(g) :  g \in \{V_J: \|g\|_{2}\leq 1\})$ therefore satisfies the conditions of Lemma \ref{lem:Dirksen} with $m = \dim(V_J) = O(2^{Jd})$, $\|\mG\|_2 \leq 2$, $\alpha^2 = ND$, $\beta=D^{1/2} 2^{Jd/2}$, where we have also used the linearity of $g \mapsto \overline{Z}_N(g)$ and that $\|g\|_\infty \leq C 2^{Jd/2} \|g\|_{2}$ for $g \in V_J$. Lemma \ref{lem:Dirksen} thus implies that for all $u \geq 1$,
$$\PP_{\mathcal B_N,f}\Big(\sup_{g \in V_J,\|g\|_{2}\leq 1}|\overline{Z}_N(g)| \geq CD^{1/2} (2^{3Jd/2} +  2^{Jd/2}N^{1/2}+\sqrt{u}N^{1/2}+u 2^{Jd/2})\Big) \leq e^{-u}.$$
Returning to \eqref{eq: concentration remainder}, we want a bound of the form
\begin{equation*} \label{eq: objective mg}
\PP_{\mathcal B_N,f}\Big(\sup_{g \in V_J,\|g\|_{2}\leq 1}|\overline{Z}_N(g)| \geq CN\xi_N\Big) \leq e^{-LN\varepsilon_N^2}.
\end{equation*}
Setting $u=LN\varepsilon_N^2$ in the second last display, this follows as soon as
$$
D^{1/2} (2^{3Jd/2}N^{-1} +  2^{Jd/2}N^{-1/2}+\eps_N+ 2^{Jd/2}\eps_N^2) \lesssim \xi_N.
$$
But this is exactly condition \eqref{eq:rate_conditions} above with the left-side multiplied by the superoptimal factor $D^{1/2}$, and hence this condition is implied by \eqref{eq:rate_conditions}. For the boundary term associated to $\overline L_{i,D}(H)$, we proceed exactly as in Step 2, replacing $\widetilde L_{i,D}$ by $\overline{L}_{i,D}$. The conclusion is the same.

\section{Proofs of posterior contraction results}

\subsection{Proof of Theorem \ref{thm:post_contract}: a general contraction theorem}\label{sec:general_contract_proof}

The proof follows the general testing approach of \cite{GV17} and the idea of using plug-in tests based on frequentist estimators satisfying concentration inequalities \cite{GN11}, adapted to the present high-frequency multidimensional diffusion setting. The next result follows by combining the proof of Theorem 8.9 in \cite{GV17} for the i.i.d. case with our evidence lower bound in Theorem \ref{thm:variance}.

\begin{lem}\label{lem:test_contract}
Suppose $f_0 \in \F_0$ and let $\Pi=\Pi_N$ be a sequence of prior distributions supported on $\F$ in \eqref{eq:F}. Further let $0 < \varepsilon_N \leq \varepsilon_{1,N} \leq \eps_{2,N} \leq \eps_{3,N} \to 0$, $\xi_N \to 0$ be positive sequences such that $N\eps_N^2 \to\infty$, let $E_N$ and $V_N$ be the corresponding quantities defined in \eqref{eq:EnVn}, and let $r>0$ be fixed. Set 
$$\C_N  = \Big\{f \in \mathcal F: \|f\|_{\C^{\alpha}} \leq r, \|f-f_0\|_{\infty} \leq \varepsilon_N,  \|f-f_0\|_{\mathcal C^k} \leq \varepsilon_{k,N} ~ \text{for } k=1,2,3 \Big\}$$
to be the set defined in \eqref{eq:Cn} and $C_0>0$ to be the fixed constant in Theorem \ref{thm:variance}, which depends only on $f_0,r,d,\O,\delta,f_{\min}$. Suppose that for some $K_0>0$,
$$E_N \leq K_0 \eps_N^2, \qquad \qquad V_N/(N^2\eps_N^4) \to 0,$$
as $N\to\infty$. Further suppose that for some $C,L,M>0$, the prior $\Pi$ satisfies $\Pi(\C_N) \geq e^{-CN\eps_N^2}$,
and there exist sets $\F_N \subseteq \F$ satisfying $\Pi(\F_N^c) \leq Le^{-(C+C_0K_0+2)N\eps_N^2}$, events $B_N$ satisfying $\PP_{f_0}(B_N) \to 1$ and a sequence of test functions $\Psi_N = \Psi_N(X_0,X_D,\dots,X_{ND})$ such that
\begin{align*}
\E_{f_0} [\Psi_N1_{B_N}] \to 0, \qquad \qquad \sup_{f \in \F_N: \|f-f_0\|_2 \geq M \xi_N} \E_f[(1-\Psi_N)1_{B_N}] \leq L e^{-(C+C_0K_0+2)N\eps_N^2}.
\end{align*}
Then as $N \to\infty$,
$$\E_{f_0} \Pi(f: \|f-f_0\|_2 \geq M \xi_N|X_0,X_D,\dots,X_{ND}) \to 0.$$
\end{lem}

\begin{proof}
Writing $A_N = \{f: \|f-f_0\|_2 \geq M \xi_n\}$ and $X^N = (X_0,X_D,\dots,X_{ND})$, we have
$$\E_{f_0} \Pi(A_N|X^N) \leq \E_{f_0} [\Pi(A_N|X^N)(1-\Psi_N)1_{B_N}] +  \E_{f_0} [\Psi_N1_{B_N}]  + \PP_{f_0}(B_N^c).$$
Since the last two terms tend to zero by assumption, it remains to control the first term. We thus need only prove convergence in $\PP_{f_0}$-probability to zero of
\begin{align*}
\Pi(A_N|X^N)(1-\Psi_N)1_{B_N} &= \frac{\int_{A_N} \prod_{i = 1}^N \frac{p_{f,D}(X_{(i-1)D}, X_{iD})}{p_{f_0,D}(X_{(i-1)D}, X_{iD})} d\Pi(f) }{\int_{\mathcal F} \prod_{i = 1}^N \frac{p_{f,D}(X_{(i-1)D}, X_{iD})}{p_{f_0,D}(X_{(i-1)D}, X_{iD})} d\Pi(f)} (1-\Psi_N)1_{B_N}.
\end{align*}
By Theorem \ref{thm:variance}, we have for any $c>0$ and any probability measure $\nu$ supported on $\C_N$,
\begin{align*}
\PP_{f_0} \left( \int_{\C_N} \prod_{i = 1}^N \frac{p_{f,D}(X_{(i-1)D}, X_{iD})}{p_{f_0,D}(X_{(i-1)D}, X_{iD})}\nu(df) \leq e^{-cN\eps_N^2-C_0 N E_N} \right) \leq \frac{V_N}{c^2 N^2 \eps_N^4}.
\end{align*}
Let $\nu = \Pi( \cdot \cap \C_N) /\Pi(\C_N)$ be the prior distribution conditioned to the set $\C_N$, and define the events
$$\Omega_N = \left\{ \int_{\C_N} \prod_{i = 1}^N \frac{p_{f,D}(X_{(i-1)D}, X_{iD})}{p_{f_0,D}(X_{(i-1)D}, X_{iD})} d\Pi(f) 
\geq \Pi(\C_N) e^{-(C_0K_0+1)N\eps_N^2} \geq e^{-(C+C_0K_0+1)N\eps_N^2} \right\},$$
where we used $\Pi(\C_N) \geq e^{-CN\eps_N^2}$ in the last inequality. Setting $c=1$ and since $E_N \leq K_0\eps_N^2$ by assumption, the second last display implies $\PP_{f_0}(\Omega_N^c) \leq V_N/(N^2 \eps_N^4) \to 0$ as $N\to\infty$. Thus for any $\eta>0$,
\begin{align*}
& \PP_{f_0} \left( \Pi(A_N|X^N)(1-\Psi_N)1_{B_N} >\eta \right) \\
&\leq  \PP_{f_0} \left(  e^{(C+C_0K_0+1)N\eps_N^2} \int_{A_N} \prod_{i = 1}^N \frac{p_{f,D}(X_{(i-1)D}, X_{iD})}{p_{f_0,D}(X_{(i-1)D}, X_{iD})} d\Pi(f)  (1-\Psi_N)1_{B_N} > \eta \right) + \PP_{f_0}(\Omega_N^c).
\end{align*}
By Fubini's theorem and since $0\leq 1-\Psi_N \leq 1$,
\begin{align*}
& \E_{f_0} \int_{A_N} \prod_{i = 1}^N \frac{p_{f,D}(X_{(i-1)D}, X_{iD})}{p_{f_0,D}(X_{(i-1)D}, X_{iD})} d\Pi(f)  (1-\Psi_N)1_{B_N} \\
&\quad =\int_{A_n} \E_f [(1-\Psi_N)1_{B_N}] d\Pi(f)  \\
& \quad \leq \Pi(\F_N^c) + \sup_{f\in \F_N: \|f-f_0\|_2 \geq M \xi_n} \E_f [(1-\Psi_N)1_{B_N}] \leq 2Le^{-(C+C_0K_0+2) N\eps_N^2}
\end{align*}
using the assumptions on the tests $\Psi_N$ and on $\F_N$. Combining the last two displays and using Markov's inequality gives for any $\eta>0$,
\begin{align*}
 \PP_{f_0} \left( \Pi(A_N|X^N)(1-\Psi_N)1_{B_N} >\eta \right) \leq  2 \eta^{-1} L e^{-N\eps_N^2} + \PP_{f_0}(\Omega_N^c) \to 0
\end{align*}
since $N\eps_N^2 \to\infty$ as $N\to\infty$.
\end{proof}

It therefore remains to show that the required tests in Lemma \ref{lem:test_contract} exist under the conditions of Theorem \ref{thm:post_contract}. 
 Setting $\widehat{f}_N$ to be the projection estimator in \eqref{eq:estimator_f}, consider tests of the form
\begin{equation}\label{eq:test}
\Psi_N = \Psi_N(X_0,X_D,\dots,X_{ND}) =  \{ \|\widehat{f}_N - f_0\|_2 \geq \widetilde{M} \xi_N\},
\end{equation}
where $\widetilde{M}>0$ is a large enough constant. The required properties then follow from Theorem \ref{thm:exp_inequality}.

\begin{lem}\label{lem:tests_exist}
Let $\eps_N,\xi_N \to 0$, $2^J = 2^{J_N} \to\infty$ and $R_J$ satisfy the conditions of Theorem \ref{thm:exp_inequality} and define the sets
\begin{equation*}
\F_N' = \left\{ f\in \F: \|f\|_{\C^1} \leq r, ~ \|f-\overline{P}_Jf\|_2 \leq C\xi_N, ~ \|f-\overline{P}_Jf\|_\infty \leq R_{J} \right\},
\end{equation*}
where $C,r>0$ and $\overline{P}_J$ denotes the projection \eqref{eq:PJ}. Assume the true $f_0\in \F_0$ satisfies the same conditions as $\F_N'$, possibly up to different constants (e.g. $\|f_0-\overline{P}_Jf_0\|_2 \leq M_0 \xi_N$ for some $M_0 >0$). Then for any $R>0$, one can take $M,\widetilde{M}>0$ large enough (depending also on $R$) such that the tests $\Psi_N$ in \eqref{eq:test} satisfy
\begin{align*}
\E_{f_0} [\Psi_N1_{\mathcal{B}_N}] \to 0, \qquad \qquad \sup_{f \in \F_N': \|f-f_0\|_2 \geq M \xi_N} \E_f[(1-\Psi_N)1_{\mathcal{B}_N}] \leq C' e^{-RN\eps_N^2},
\end{align*}
where the event $\mathcal B_N$ satisfies $\sup_{f \in \mathcal F}\PP_{f}(\mathcal B_N^c) \rightarrow 0$ as $N \rightarrow \infty$.
\end{lem}

\begin{proof}
Consider the event $\mathcal{B}_N$ in \eqref{eq: def nice event} and used in Theorem \ref{thm:exp_inequality}, which by Lemma \ref{lem: compdisccont} satisfies $\PP_{f_0}(\mathcal{B_N}) \to 1$ as $N\to\infty$ since $f_0 \in \F_0 \subset \F$. Using the definition \eqref{eq:test} of the test $\Psi_N$, that $\|f_0 - \overline{P}_Jf_0\|_2 \leq M_0 \xi_N$ and the triangle inequality,
\begin{align*}
\E_{f_0} [\Psi_N 1_{\mathcal{B}_N}] \leq  \PP_{f_0} (\|\widehat{f}_N - \overline{P}_J f_0\|_2 \geq (\widetilde{M}-M_0) \xi_N,\mathcal{B}_N) .
\end{align*}
Since the conditions of Theorem \ref{thm:exp_inequality} are satisfied, applying that theorem with $\widetilde{M}>0$ large enough bounds the right-hand side by $C'e^{-N\eps_N^2} \to 0$, giving the first part of the lemma.

For the type-II errors, let $f \in \F_N'$ satisfy $\|f-f_0\|_2 \geq M\xi_N$ for some $M>0$ to be specified below. Then, since $\|f-\overline{P}_Jf\|_2 \leq C\xi_N$,
\begin{align*}
\E_f[(1-\Psi_N)1_{\mathcal B_N}] &= \PP_f(\|\widehat{f}_N - f_0\|_2 \leq \widetilde{M} \xi_n, \mathcal{B}_N )\\
& \leq  \PP_f( \|f_0-f\|_2 -  \|\widehat{f}_N - \overline{P}_Jf\|_2 - \|f-\overline{P}_Jf\|_2 \leq \widetilde{M} \xi_n, \mathcal{B}_N )\\
& \leq  \PP_f( (M-\widetilde{M}-C)\xi_N \leq  \|\widehat{f}_N - \overline{P}_Jf\|_2 , \mathcal{B}_N ) \leq C' e^{-RN\eps_N^2},
\end{align*}
uniformly over such $f$, where the last inequality follows from Theorem \ref{thm:exp_inequality} upon taking $M=M(R,\widetilde{M},C)>0$ large enough.
\end{proof}

Theorem \ref{thm:post_contract} then follows from using the tests in Lemma \ref{lem:tests_exist} together with Lemma \ref{lem:test_contract}.

\subsection{Proof of Theorem \ref{thm:GP}: contraction rates for Gaussian priors}

Let $\|\cdot\|_{\H_W}$ denote the RKHS of $W$. Then $\chi W$ in \eqref{eq:prior} is a mean-zero Gaussian process with RKHS $\H_{\chi W} = \{\chi w: w\in \H_W\}$ and whose RKHS norm satisfies that for every $w\in \H_W$, there exists some $w^* \in \H_W$ such that $\chi w = \chi w^*$ and 
$$\|\chi w\|_{\H_{\chi W}} =\|w^*\|_{\H_W}$$ 
(Exercise 2.6.5 of \cite{GN16}). Furthermore, we have that $\H_W = \H_V$ with RKHS norm $\|\cdot\|_{\H_W} = \sqrt{N}\eps_N \|\cdot\|_{\H_V}$. 

We require the following lemma about the concentration of Gaussian measures on suitable sieve sets. The proof is similar to Theorem 2.2.2 of \cite{N22} (see also Lemma 5.2(i) in \cite{GR20}) and is hence omitted.

\begin{lem}\label{lem:prior_concentrate}
For $s,M>0$ and the sequence $\eps_N = N^{-\frac{s}{2s+d}}$, define the sets
\begin{equation}\label{eq:WN}
\mathcal{W}_N = \{ W =W_1+W_2:\|W_1\|_\infty \leq \eps_N, ~ \|W_2\|_{H^s} \leq M, ~ \|W\|_{\C^4} \leq M\}.
\end{equation}
Let $W = V/(\sqrt{N}\eps_N)$ for $V \sim \Pi_V$ a mean-zero Gaussian process satisfying Condition \ref{cond:RKHS}. Then for every $K>0$, there exists $M>0$ large enough such that $\Pi_W(\mathcal{W}_N^c) \leq e^{-KN\eps_N^2}$. 
\end{lem}

\begin{proof}[Proof of Theorem \ref{thm:GP}]
We verify the assumptions of Theorem \ref{thm:post_contract} with $\eps_N = N^{-\frac{s}{2s+d}}$, $\eps_{k,N} = N^{-\frac{s-k}{2s+d}}$, $k=1,2,3$, and $2^J \simeq N^\frac{1}{2s+d}$ as in Remark \ref{rk:ex}, so that $E_N \leq K_0 \eps_N^2$ and $V_N /(N^2 \eps_N^4) \to 0.$

\textit{Small-ball probability}: consider the ``small-ball condition'' (ii) in Theorem \ref{thm:post_contract}. Recall that under the prior $\Pi$ in \eqref{eq:prior}, $f = \Phi (\chi W)$ with $W = V/(\sqrt{N\eps_N)}$ for $V$ a mean-zero Gaussian process and $\Phi(x) = f_{\min} + (1-f_{\min})e^x$.

We first state some useful inequalities regarding the smoothness of functions when composed with $\Phi$ and $\Phi^{-1}$. When composed with a uniformly bounded function $g:\O \to \R$ satisfying $\|g\|_\infty \leq M$, we may restrict the domain of $\Phi$ to $[-M,M]$. Since $\Phi^{(k)} = (1-f_{\min})e^x$ satisfies $\|\Phi^{(k)}\|_{L^\infty[-M,M]} \leq e^M$ for $k=1,2,\dots$, we deduce that $\|\Phi\|_{\C^k[-M,M]} \leq C_{k,d} e^M <\infty$ for $k=1,2,\dots$. By Theorem 4.3(ii)(3) in \cite{DO99}, we thus have that for any $r\geq 1$,
\begin{equation}\label{eq:Phi_bound}
\|\Phi(g)\|_{\C^r} \leq C_{r,d} \|\Phi\|_{\C^r[-M,M]} (1+\|g\|_{\C^r}^r) \leq C_{r,d,M} (1+\|g\|_{\C^r}^r)  \qquad \forall g:\|g\|_\infty \leq M.
\end{equation}
Using the multivariate version of Fa\`a di Bruno's formula \cite{CS96}, one can show that for any $g_1,g_2\in \C^k$ with $\|g_i\|_\infty \leq M$ and integer $k\geq 1$,
\begin{equation}\label{eq:Phi_stability}
\begin{split}
\|\Phi(g_1) - \Phi(g_2)\|_{\C^k} & \leq C_{k,d} \max_{1 \leq j \leq k+1} \|\Phi^{(j)}\|_{L^\infty[-M,M]} (1+\|g_1\|_{\C^k}^k + \|g_2\|_{\C^k}^k)\|g_1-g_2\|_{\C^k} \\
& \leq C_{k,d,M} (1+\|g_1\|_{\C^k}^k + \|g_2\|_{\C^k}^k)\|g_1-g_2\|_{\C^k} \qquad \forall g_1,g_2: \|g_1\|_\infty ,\|g_2\|_\infty \leq M
\end{split}
\end{equation}
(see e.g. Lemma 2 in \cite{RS17} for the proof of a similar argument). Turning to $\Phi^{-1}(y) = \log \frac{y-f_{\min}}{1-f_{\min}}$, we have $[\Phi^{-1}]^{(k)}(y) = (-1)^{k-1}(k-1)!(y-f_{\min})^{-k}$ for $k=1,2,\dots$, so that $\|[\Phi^{-1}]^{(k)}\|_{L^\infty[2f_{\min},\infty)} \leq (k-1)! f_{\min}^{-k} < \infty$. Since $f_0 \in \F_0 \cap \C^s$ satisfies $f_0 \geq 2f_{\min}$, arguing as \eqref{eq:Phi_bound} implies that $w_0 = \Phi^{-1}(f_0) \in \C^s$. Similarly, since $f_0 \in \F_0 \subseteq \C^\alpha(\O)$ for $\alpha = 4  \vee (2 \left\lfloor d/4+1/2 \right\rfloor)$ in \eqref{eq:alpha}, we may find $r>0$ large enough that $\|w_0\|_{\C^{\alpha}} \leq r/2$.

Let $g:\O \to \R$ satisfy $\|g\|_{\C^\alpha} \leq r$. By \eqref{eq:Phi_bound}, $\|\Phi(g)\|_{\C^\alpha} \leq C_{\alpha,d,r}$, while by \eqref{eq:Phi_stability}, we have $\|\Phi(g)-\Phi(w_0)\|_{\C^k} \leq C_{k,d,r}\|g-w_0\|_{\C^k}$ for $k=1,2,3$ since $\alpha \geq 4$. Similarly, $\|\Phi(g)-\Phi(w_0)\|_\infty \leq C_r \|g-w_0\|_\infty$ using that the exponential function is Lipschitz on a bounded interval. Setting $g = \chi w$, we thus have
\begin{align*}
& \left\{ w \in \C^4(\O): \|\chi w\|_{\C^\alpha} \leq r, \|\chi w-w_0\|_\infty \lesssim \eps_N, \|\chi w -w_0\|_{\C^k} \lesssim \eps_{k,N} ~\text{for } k=1,2,3\right\} \\
& \quad \subset \left\{ w \in \C^4(\O): \|\Phi(\chi w)\|_{\C^\alpha} \leq C, \|\Phi(\chi w)-f_0\|_\infty \leq \eps_N, \|\Phi(\chi w) -f_0\|_{\C^k} \leq \eps_{k,N} ~\text{for } k=1,2,3\right\},
\end{align*}
where the $\lesssim$ above depend only on $d,\alpha,r$. Writing $\Pi_W$ and $\Pi$ for the prior laws of $W$ and $f=\Phi(\chi W)$, respectively, the prior probability $\Pi(\C_N)$ equals the $\Pi_W$-probability of the last set in above display. We thus conclude that
\begin{align*}
\Pi(\C_N)  \geq \Pi_W (W \in \C^4(\O): \|\chi W\|_{\C^\alpha}\leq r, ~ \|\chi W -w_0\|_{\infty} \lesssim \varepsilon_N, \|\chi W -w_0\|_{\mathcal C^k} \lesssim \varepsilon_{k,N} ~ \text{for } k=1,2,3).
\end{align*}
where the $\lesssim$ above depend only on $\Phi,d,\alpha$. Under the theorem hypotheses, the last probability is lower bounded by
\begin{align*}
\Pi_W( W\in \C^4(\O): &  \|\chi W - \chi v_{0,N} \|_{\C^\alpha} \leq 2r, \|\chi W -\chi v_{0,N}\|_{\infty} \lesssim \varepsilon_N, \\
& \qquad  \|\chi W -\chi v_{0,N}\|_{\mathcal C^k} \lesssim \varepsilon_{k,N} ~ \text{for } k=1,2,3),
\end{align*}
possibly after replacing $r>0$ by a larger constant and then $\eps_{k,N}$ by a multiple of itself. Since $\chi W$ is a mean-zero Gaussian process under the prior, and $\chi v_{0,N} \in \H_{\chi W}$ is in its RKHS, Lemma I.27 of \cite{GV17} lower bounds the last display by
$$e^{-\frac{1}{2}\|\chi v_{0,N}\|_{\H_{\chi W}}^2} \Pi_W (W \in \C^4(\O) :\|\chi W\|_{\C^\alpha} \leq 2 r, \|\chi W\|_{\infty} \lesssim \varepsilon_N,  \|\chi W\|_{\mathcal C^k} \lesssim  \varepsilon_{k,N} ~ \text{for } k=1,2,3).$$
Note that $\|\chi v_{0,N}\|_{\H_{\chi W}}^2 \leq N\eps_N^2 \|v_{0,N}\|_{\H_V}^2 = O(N\eps_N^2)$ using Exercise 2.6.5 of \cite{GN16} and the assumption on $v_{0,N}$. Using the last display, that $\chi \in \C^\infty$, $W = V/\sqrt{N}\eps_N$ and the Gaussian correlation inequality \cite{R14} (see Lemma A.2 in \cite{GR20} for the exact formulation we use), we have
\begin{equation}\label{eq:GP_SB}
\begin{split}
\Pi(\C_N) \geq e^{-CN\eps_N^2} & \Pi_V (V \in \C^4(\O): \|V\|_{\C^\alpha} \leq 2r \sqrt{N}\eps_N) \\
& \times \prod_{k=0}^3\Pi_V (V\in\C^4(\O): \|V\|_{\C^k} \lesssim  \sqrt{N}\eps_N\eps_{k,N} ),
\end{split}
\end{equation}
for some $C>0$ and where above we have written $\eps_{0,N} = \eps_N$. It therefore suffices to lower bound each of the prior probabilities in \eqref{eq:GP_SB}.

Let $\overline{N}(\mathcal{A},\|\cdot\|,\tau)$ denote the covering number of a set $\mathcal{A}$ by $\|\cdot\|$-balls of radius $\tau$. 
For $\H_{V,1}$ and $H_1^s$ the unit balls of the RKHS $\H_V$ and $H^s$, respectively, we have under Condition \ref{cond:RKHS} that for integer $k\geq 0$,
$$\log \overline{N}(\H_{V,1},\|\cdot\|_{\C^k},\tau) \leq \log \overline{N}(cH_1^s,\|\cdot\|_{\C^k},\tau) \lesssim \tau^{-\frac{d}{s-k}},$$
where the last inequality follows by arguing as in the proof of Theorem 4.3.36 in \cite{GN16} as soon as $s-k>d/2$. Applying the small ball estimate in Theorem 1.2 of \cite{LL99} (in particular, (1.3) in \cite{LL99} with exponents $\alpha = \frac{2d}{2(s-k)-d}$ and $\beta = 0$), we have for $s-k>d/2$,
\begin{equation*}
\Pi_V ( \|V\|_{\C^k} \leq \eta) \geq \exp \left( -c \eta^{-\frac{2d}{2(s-k)-d}} \right) \qquad \text{as } \eta \to 0.
\end{equation*}
Using the last display with $\eta_N = 1/(\log N)$ and $k=\alpha$ shows that the first prior probability in \eqref{eq:GP_SB} is greater than or equal to $e^{-c(\log)^\kappa} \geq e^{-CN\eps_N^2}$ for $s>\alpha+d/2$, any fixed $r>0$ and some $\kappa>0$. In particular, it can be checked that the minimal smoothness $s_{d,a}^*$ in \eqref{eq:s_star} assumed in the present theorem satisfies $s_{d,a}^* \geq \alpha+d/2$ for any dimension $d\in \mathbb{N}$, and hence this last condition is satisfied. For the choice $\eps_{k,N} = N^{-\frac{s-k}{2s+d}}$, we have $\sqrt{N}\eps_{0,N}\eps_{k,N} = N^{-\frac{s-k-d/2}{2s+d}} \to 0$ for $s>k+d/2$. Using the last display with $\eta = \eta_{k,N} = c \sqrt{N}\eps_{0,N} \eps_{k,N}$ then yields
$$\prod_{k=0}^3\Pi_V (\|V\|_{\C^k} \lesssim  \sqrt{N}\eps_N\eps_{k,N} ) \geq \prod_{k=0}^3 \exp \left( -c \eta_{k,N}^{-\frac{2d}{2(s-k)-d}} \right) \geq \exp \left( -c N\eps_N^2\right)$$
using that $(\sqrt{N}\eps_{0,N}\eps_{k,N} )^{-\frac{2d}{2(s-k)-d}} = N^\frac{d}{2s+d} = N\eps_N^2$. Together with \eqref{eq:GP_SB}, this gives the required lower bound $\Pi(\C_N) \geq e^{-CN\eps_N^2}$ for any fixed $r>1$ and $s>k+d/2$, which verifies the small-ball condition (ii) in Theorem \ref{thm:post_contract}.

\textit{Sieve sets:} consider the condition (i) in Theorem \ref{thm:post_contract}. For $s$ as in this theorem and $M>0$, let $\mathcal{W}_N$ be the set defined in \eqref{eq:WN}. For any $K>0$, we can find $M=M(K)>0$ sufficiently large that $\Pi_W(\mathcal{W}_N^c) \leq e^{-KN\eps_N^2}$ by Lemma \ref{lem:prior_concentrate}. In particular, let $K>C+C_0K_0+2$ for $C$ the constant in Theorem \ref{thm:post_contract} coming from the small-ball condition (ii) just proved. Define
$$\F_N = \{f = \Phi (\chi w): w\in \mathcal{W}_N\},$$
so that $\Pi(\F_N^c) \leq e^{-KN\eps_N^2}$ as required. Since $w \in \mathcal{W}_N$ satisfies $\|w\|_{\C^4} \leq M$, we have $\|\Phi(\chi w)\|_{\C^1} \leq C_{d,M} (1+\|\chi\|_{\C^1} \|w\|_{\C^1}) \leq C_{d,M,\chi}(1+M)$ using \eqref{eq:Phi_bound}. To verify the bias conditions, we invoke Lemma \ref{lem:bias} below with $p=2,\infty$ to get that for all $f\in \F_n$,
$$\|f - \overline{P}_Jf \|_p \lesssim  \eps_N + 2^{-J(s-d/2+d/p)}  \lesssim 2^{-J(s-d/2+d/p)}.$$
Therefore, since $2^{-Js} \simeq \eps_N$, $\F_N$ satisfies
$$\F_N \subseteq \left\{ f\in \F: \|f\|_{\C^1} \lesssim 1, ~ \|f-\overline{P}_Jf\|_2 \lesssim \eps_N, ~ \|f-\overline{P}_J f\|_\infty \lesssim R_{J} \right\}$$
for $R_{J} \simeq 2^{-J(s-d/2)} \simeq N^{-\frac{s-d/2}{2s+d}}$, i.e. condition (i) in Theorem \ref{thm:post_contract}.

We lastly verify the numeric constraints on the sequence choices with also $D=N^{-a}$, $a\in(1/2,1)$. Since $f_0 \in \F_0 \cap \C^s(\O)$, we have by similar arguments to the above that $\|f_0-\overline{P}_J f_0\|_2 \lesssim \eps_N$ and $\|f_0 - \overline{P}_J f_0\|_\infty \lesssim R_{J}$. The condition $2^{Jd} \simeq N^\frac{d}{2s+d} =o( \sqrt{ND}) = o( N^{(1-a)/2})$ is equivalent to $\frac{d(1+a)}{2(1-a)}<s$. Turning to the quantitative conditions \eqref{eq:exp_ineq_conditions}, for our sequence choices these reduce to
\begin{align*}
R_{J,\infty}^2 \eps_N^2 \lesssim D\xi_N^2 \quad & \iff \quad N^{\frac{a}{2} - \frac{2s-d/2}{2s+d}} \lesssim \xi_N \\
2^{3Jd/2}N^{-1} + 2^{Jd/2} N^{-1/2} + 2^{Jd/2} \eps_N^2 + \eps_N \lesssim \xi_N \quad & \iff \quad N^{-\frac{2s-d/2}{2s+d}} + N^{-\frac{s}{2s+d}} \lesssim \xi_N
\end{align*}
Since $s>d/2$ by assumption, we finally get rate
$\xi_N \simeq N^{-\frac{s}{2s+d}} +  N^{\frac{a}{2} - \frac{2s-d/2}{2s+d}},$
i.e. the largest of the two conditions above. One can check that $N^{-\frac{s}{2s+d}}$ is the largest term for $s>\frac{d(1+a)}{2(1-a)}$, $a\in (1/2,1)$.
\end{proof}

\begin{lem}\label{lem:bias}
For $s>d/2\vee 1$ and $0<\eps_N\leq M$, let $\mathcal{W}_N$ be the set \eqref{eq:WN}. Let $\Phi(x) = f_{\min} + (1-f_{\min})e^x$ be the link function, $\chi\in \C^\infty (\O)$ a smooth cutoff function such that $\chi \equiv 1$ on $K$ and $\chi \equiv 0$ outside $\O_0$, $w\in \mathcal{W}_N$ and $\overline{P}_J$ be the projection operator \eqref{eq:PJ}. Then for $p\in [2,\infty]$ and $w\in \mathcal{W}_N$,
$$\|\Phi \circ (\chi w) - \overline{P}_J[ \Phi \circ (\chi w)]\|_p \leq C (\eps_N + 2^{-J(s-d/2+d/p)}),$$
where $C$ depends only on $\Phi,\chi,p,M,s,d,\text{vol}(\O)$ and the wavelet basis.
\end{lem}

\begin{proof}
Let $w\in \mathcal{W}_N$ and write $w = w_1+w_2$ as in \eqref{eq:WN}. Then for any $x\in \O$,
\begin{equation}\label{eq:proj_decomp}
\begin{split}
 |\Phi \circ (\chi w)(x) -1 - P_J[ \Phi \circ (\chi w)-1](x)| 
& \leq |\Phi \circ (\chi w_1 + \chi w_2)(x) - \Phi \circ (\chi w_2)(x)| \\
& \quad + | \Phi \circ (\chi w_2)(x) -1- P_J[\Phi \circ (\chi w_2)-1](x) | \\
& \quad + |P_J[\Phi \circ (\chi w_2)-1](x) - P_J[\Phi \circ (\chi w_1 + \chi w_2)-1](x)|.
\end{split}
\end{equation}
By the mean-value theorem, the first term in \eqref{eq:proj_decomp} is bounded by $e^{\|\chi\|_\infty (\|w_1\|_\infty + \|w_2\|_\infty)} \|\chi\|_\infty \|w_1\|_\infty \lesssim \|w_1\|_\infty$ since $\|w_2\|_\infty \lesssim \|w_2\|_{H^s} \leq M$ by the Sobolev embedding theorem.

Let $K_J(x,y)= 2^{Jd} \sum_{k \in \mathbb{Z}} \phi(2^Jx-k) \phi(2^Jy-k)$ denote the wavelet projection kernel on all of $\R^d$, where $\phi$ is the Daubechies father wavelet. By the localization property of wavelets, $\int_{\R^d}|K_J(x,y)| dy \lesssim 1$ for all $x\in \R^d$.
Since the full $L^2$-projection operator $\widetilde{P}_J:L^2(\R^d) \to L^2(\R^d)$ onto $\{ \psi_{l,r}: l \leq J, r\in \mathbb{Z}^d\}$ satisfies $\widetilde{P}_J[g](x) = \int_{\R^d} K_J(x,y)g(y) dy$, and $P_J$ and $\widetilde{P}_J$ coincide for functions whose support is contained in $\O_0$, the third term in \eqref{eq:proj_decomp} can be expanded as
\begin{align*}
& |P_J[\Phi \circ (\chi w_2)-1](x) - P_J[\Phi \circ (\chi w_1 + \chi w_2)-1](x)| \\
& \quad = \left| \int_{\R^d} K_J(x,y)  \big( \Phi(\chi(y) w_2(y)) - \Phi (\chi(y)w_1(y)+ \chi(y)w_2(y)) \big) dy \right| \\
& \quad \leq C_{M,\chi} \|\chi\|_\infty \|w_1\|_\infty \int_{\R^d} |K_J(x,y)| dy \lesssim \|w_1\|_\infty,
\end{align*}
where we applied the mean-value theorem as above for the first inequality.
Combining the bounds for \eqref{eq:proj_decomp}, taking $p^{th}$-powers of everything and integrating over the bounded domain $\O$ then yields
$$\|\Phi \circ (\chi w) -1 - P_J[ \Phi \circ (\chi w)-1]\|_p^p \lesssim \|w_1\|_\infty^p + \| \Phi \circ (\chi w_2) -1- P_J[\Phi \circ (\chi w_2)-1] \|_p^p.$$
Since $\supp(\Phi \circ (\chi w_2)-1) \subseteq \O_0$, we may replace the $L^p(\O)$ norm in the last term by $L^p(\R^d)$ and the projection operator $P_J$ by $\widetilde{P}_J$. Applying Corollary 3.3.1 of \cite{Cohen03}, we then get
$$\| \Phi \circ (\chi w_2) -1- \widetilde{P}_J[\Phi \circ (\chi w_2)-1] \|_{L^p(\R^d)} \lesssim 2^{-Jt} \|\Phi \circ (\chi w_2)-1\|_{B_{pq}^t}$$
for any $1 \leq q \leq \infty$. Using the embedding $H^s(\R^d) = B_{22}^s(\R^d) \hookrightarrow B_{p2}^{s-d/2+d/p}$ for any $p\geq 2$ (\cite{GN16}, Proposition 4.3.10), we may set $t = s-d/2+d/p$ and $q=2$ in the last display and replace the $B_{pq}^t$-norm by an $H^s$-norm. Since $\|\chi w_2\|_\infty \leq C_{\chi,s,d} M$ by the Sobolev embedding theorem, we may also replace $\Phi$ in the last composition by some smooth function $\Phi^M$ that coincides with $\Phi$ on $[-CM,CM]$, but is bounded and all of whose derivatives are uniformly continuous (depending on $M$). Since $s>d/2\vee 1$, Theorem 4(i) of \cite{S92} yields $\|\Phi^M \circ (\chi w_2)-1 \|_{H^s} \lesssim \|\chi w_2\|_{H^s} + \|\chi w_2\|_{H^s}^s \lesssim \|w_2\|_{H^s} + \|w_2\|_{H^s}^s$. In summary, for $p\geq 2$,
$$\|\Phi \circ (\chi w) -1 - P_J[ \Phi \circ (\chi w)-1] \|_p \lesssim \|w_1\|_\infty + 2^{-J(s-d/2+d/p)} (\|w_2\|_{H^s} + \|w_2\|_{H^s}^s),$$
where the constant depends on $\chi,p,M,s,d,\text{vol}(\O)$. Substituting in the bounds for $w_1$ and $w_2$ coming from the definition \eqref{eq:WN} of $\mathcal{W}_N$ then gives the result.
\end{proof}

\subsection{Proof of Corollaries \ref{cor:Matern} and \ref{cor:gauss_wav}: examples of Gaussian process priors}

\begin{proof}[Proof of Corollary \ref{cor:Matern}]
The Mat\'ern process on all of $\R^d$ has RKHS norm equal to $\|\cdot\|_{\H_V} = \|\cdot\|_{H^s(\R^d)}$ (\cite{GV17}, Chapter 11), and hence its restriction to $\O$ has RKHS norm equal to $H^s(\O)$ by Exercise 2.6.5 of \cite{GN16}. Moreover, by Proposition I.4 of \cite{GV17}, $V$ has a version who sample paths are in $\C^{r}(\O)$ $\Pi$-almost surely for any $r<s-d/2$. In particular, $\C^r(\O)$ is a separable linear subspace of $\C^4(\O)$ for any $r>4$, a suitable choice of which exists as soon as $s>d/2+4$. The Mat\'ern process thus satisfies Condition \ref{cond:RKHS} for $s>d/2+4$ since its RKHS norm equals the $H^s(\O)$-norm.

We may therefore apply Theorem \ref{thm:GP} to the Mat\'ern process when $f_0 \in \C^s(\O)$. Indeed, since $w_0 = \Phi^{-1}(f_0) \in \C^s(\O)$ by \eqref{eq:Phi_bound} (which applies also to $\Phi^{-1}$) and $\supp(w_0) \subseteq K$, we may take the constant sequence $v_{0,N} = w_0 \in H^s(\O) = \H_V$, which trivially satisfies the conditions (i)-(iii) in Theorem \ref{thm:GP}.
\end{proof}

\begin{proof}[Proof of Corollary \ref{cor:gauss_wav}]
The Gaussian wavelet series prior \eqref{eq:gauss_wavelet} has RKHS equal to
$$\H_V = \left\{ h = \sum_{l=J_0}^J \sum_{r \in R_l} h_{lr} \psi_{lr}: \|h\|_{\H_V}^2 = \sum_{l=J_0}^J \sum_{r \in R_l} 2^{2ls} h_{lr}^2 < \infty \right\},$$
with $\|h\|_{\H_V} = \infty$ if $h$ is not a truncated sum up to level $J$. Hence $\H_V \hookrightarrow H^s(\O)$ by the wavelet characterization of $L^2$-Sobolev norms. Draws from the prior \eqref{eq:gauss_wavelet} are finite sums of wavelets $(\psi_{lr})$, hence $V$ will almost surely have the same H\"older smoothness as $(\psi_{lr})$. Thus taking a smooth enough wavelet basis, we have that $V$ is supported on a separable linear subspace of $\C^4(\O)$, thereby satisfying Condition \ref{cond:RKHS}.

We may therefore apply Theorem \ref{thm:GP}. For $f_0 \in \C^s(\O)$, $w_0 = \Phi^{-1}(f_0) \in \C^s(\O)$ by \eqref{eq:Phi_bound} (which applies also to $\Phi^{-1}$) and $\supp(w_0) \subseteq K \subseteq \O_0$. Set
$$v_{0,N}(x) = \sum_{l=J_0}^J \sum_{r \in R_l} \langle w_0,\psi_{lr} \rangle_2 \psi_{lr}(x),$$
the wavelet projection at resolution level $J$. In particular, $\|v_{0,N}\|_{\H_V} \simeq \|v_{0,N}\|_{H^s} \leq \|w_0\|_{H^s}<\infty$. Moreover, using that $w_0 = \chi w_0$ and the standard Besov space embeddings $B_{\infty\infty}^r \hookrightarrow \C^r \hookrightarrow B_{\infty 1}^r$ for all $r\geq 0$, we have
\begin{align*}
\|w_0 - \chi v_{0,N}\|_{\C^k} \lesssim \|w_0 -v_{0,N}\|_{\C^k} & \lesssim \sum_{l=J+1}^\infty 2^{l(k+d/2)} \max_{r \in R_l} |\langle w_0,\psi_{lr} \rangle_2| \\
& \lesssim \|w_0\|_{B_{\infty\infty}^s} \sum_{j=J+1}^\infty 2^{-l(s-k)} \lesssim \|w_0\|_{\C^s} 2^{-J(s-k)} \lesssim N^{-\frac{s-k}{2s+d}}
\end{align*}
since $2^J \simeq N^\frac{1}{2s+d}$. Together these verify conditions (i)-(iii) of Theorem \ref{thm:GP}, thereby giving the desired result.
\end{proof}

\section{Appendix}

\subsection{Proof of Theorem \ref{thm:lower_bd}: minimax lower bound} \label{sec: proof minimax lower bound}
We go along a classical scheme via the usual Assouad cube technique, see for instance Tsybakov \cite{tsybakov2004introduction}.\\ 

\noindent {\it Step 1:} Pick a cube $[c_1,c_2]^d \subset \mathcal{K}$ and a smooth wavelet $\psi$ with compact support on $\R$. Set $\psi_{J,k} = 2^{J/2}\psi(2^J\cdot-k)$ and denote by $K_J \subset \mathbb Z$ a maximal set of indices $k$ such that $\mathrm{supp}(\psi_{Jk}) \subset [c_1,c_2]$ for all $k \in K_J$, and $\mathrm{supp}(\psi_{Jk}) \cap \,\mathrm{supp}(\psi_{Jk'}) = \emptyset$ for every $k,k' \in K_J$ with $k \neq k'$. For $\ell = (\ell_1,\ldots ,\ell_d) \in (K_J)^d$, set
$$\psi_{J,(\ell_1,\ldots, \ell_d)}(x) = \prod_{i = 1}^d\psi_{J\ell_i}(x^i),\qquad x=(x^1,\ldots, x^d) \in \O,$$
so that $\mathrm{supp}(\psi_{J,(\ell_1,\ldots, \ell_d)}) \subset [c_1,c_2]^d$ and $\mathrm{supp}(\psi_{J,(\ell_1,\ldots, \ell_d)}) \cap \,\mathrm{supp}(\psi_{J,(\ell_1',\ldots, \ell_d')}) = \emptyset$ for
$(\ell_1,\ldots, \ell_d) \neq (\ell'_1, \ldots, \ell_d')$ in $(K_J)^d$. For $\varepsilon  = (\eps_{(\ell_1,\dots,\ell_d)}: \ell_1,\dots,\ell_d \in K_J) \in \{-1, 1\}^{|K_J|^d}$, we set
$$f_{\varepsilon}(x) = 1+\gamma \sum_{(\ell_1,\ldots, \ell_d)\in (K_J)^d}\varepsilon_{(\ell_1, \ldots, \ell_d)}\psi_{J(\ell_1,\ldots, \ell_d)}(x),\qquad \varepsilon_{(\ell_1, \ldots, \ell_d)} \in \{-1, 1\}.$$
Taking $\gamma \simeq N^{-1/2}$ and $2^{J} \simeq N^{1/(2s+d)}$, we have $f_\varepsilon \in \mathcal F_0 \cap \{f: \|f\|_{\mathcal C^s} \leq M\}$ by choosing prefactors sufficiently small to accomodate constants.\\

\noindent {\it Step 2:} For an arbitrary estimator $\widehat f_N$, we repeat the classical argument to bound the maximal $L^2$ risk. We have
\begin{align*}
&\sup_{f \in  \mathcal F_0 \cap \{f: \|f\|_{\mathcal C^s} \leq M\}} \E_f[\|\widehat f_N-f\|_2^2]  \geq \max_{\varepsilon \in \{-1, 1\}^{|K_J|^d}} \E_{f_\varepsilon}[\|\widehat f_N-f_\varepsilon\|_2^2] \\
& \geq \frac{1}{2^{|K_J|^d}}\sum_{(\ell_1,\ldots, \ell_d) \in (K_J)^d} \sum_{\varepsilon \in \{-1, 1\}^{|K_J|^d}}\int_{\mathrm{supp}(\psi_{J(\ell_1,\ldots, \ell_d)})}\E_{f_\varepsilon}[|\widehat f_N(x)-f_\varepsilon(x)|^2] dx. 
 \end{align*}
 For a given configuration $(\ell_1, \ldots, \ell_d)$, we write $\varepsilon = (\check{\varepsilon}_{(\ell_1, \ldots, \ell_d)}, \varepsilon_{(\ell_1, \ldots, \ell_d)})\in \{-1,1\}^{|K_J|^d}$ with $\check{\varepsilon}_{(\ell_1,\dots,\ell_d)} \in \{-1,1\}^{|K_J|^d-1}$ and $\varepsilon_{(\ell_1, \ldots, \ell_d)} \in \{-1,1\}$, possibly after reordering. It follows that 
 \begin{align*}
&  \sum_{\varepsilon \in \{-1, 1\}^{|K_J|^d}}\int_{\mathrm{supp}(\psi_{J(\ell_1,\ldots, \ell_d)})}\E_{f_\varepsilon}[ |\widehat f_N(x)-f_\varepsilon(x)|^2] dx \\
= & \sum_{\check{\varepsilon}_{(\ell_1, \ldots, \ell_d)} \in \{-1, 1\}^{|K_J|^d-1}}\int_{\mathrm{supp}(\psi_{J(\ell_1,\ldots, \ell_d)})}\Big(\E_{f_{(\check{\varepsilon}_{(\ell_1, \ldots, \ell_d)}, +1)}}[ |\widehat f_N(x)-f_{(\check{\varepsilon}_{(\ell_1, \ldots, \ell_d)}, +1)}(x)|^2] \\
&\hspace{3cm}+\E_{f_{(\check{\varepsilon}_{(\ell_1, \ldots, \ell_d)}, -1)}}[|\widehat f_N(x)-f_{(\check{\varepsilon}_{(\ell_1, \ldots, \ell_d)}, -1)}(x)|^2] \Big)dx \\
\geq &  \tfrac{1}{2} \sum_{\check{\varepsilon}_{(\ell_1, \ldots, \ell_d)} \in \{-1, 1\}^{|K_J|^d-1}}\int_{\mathrm{supp}(\psi_{J(\ell_1,\ldots, \ell_d)})}\big|f_{(\check{\varepsilon}_{(\ell_1, \ldots, \ell_d)}, +1)}(x)-f_{(\check{\varepsilon}_{(\ell_1, \ldots, \ell_d)}, -1)}(x)\big|^2dx \\
&\hspace{3cm}\times e^{-\rho} \left( 1-(1-e^{-\rho})^{-1} \|\mathbb P_{f_{(\check{\varepsilon}_{(\ell_1, \ldots, \ell_d)}, 1)}}-\mathbb P_{f_{(\check{\varepsilon}_{(\ell_1, \ldots, \ell_d)}, -1)}}\|_{TV} \right)
  \end{align*}
for any $\rho >0$, as follows by triangle inequality and classical information bounds, see {\it e.g.} \cite{GHR04} Section 5. From
$$\int_{\mathrm{supp}(\psi_{J(\ell_1,\ldots, \ell_d)})} \big|f_{(\check{\varepsilon}_{(\ell_1, \ldots, \ell_d)}, +1)}(x)-f_{(\check{\varepsilon}_{(\ell_1, \ldots, \ell_d)}, -1)}(x)\big|^2 dx= 4\gamma^2 \|\psi_{J(\ell_1, \ldots, \ell_d)}\|_2^2 \simeq \gamma^{2},$$
we infer
$$\sup_{f \in  \mathcal F_0 \cap \{f: \|f\|_{\mathcal C^s} \leq M\}}\E_f [\|\widehat f_N-f\|_2^2]   \gtrsim 2^{Jd}\gamma^2 \simeq N^{-2s/(2s+d)}$$
by taking $\rho$ sufficiently large and using $|K_J| \gtrsim 2^J$, provided the total variation is bounded away from $1$ uniformly in $\varepsilon$.\\

\noindent {\it Step 3:} Write $g_\varepsilon = f_{(\check{\varepsilon}_{(\ell_1, \ldots, \ell_d)}, 1)}$ and $h_\varepsilon = f_{(\check{\varepsilon}_{(\ell_1, \ldots, \ell_d)}, -1)}$. It remains to bound 
$$\|\mathbb P_{g_\varepsilon}-\mathbb P_{h_\varepsilon}\|_{TV}^2 \leq 2 \E_{g_{\varepsilon}}\Big[ \log \frac{d\mathbb P_{g_\varepsilon}}{d\mathbb P_{h_{\varepsilon}}}\Big] = 2N  \E_{g_\varepsilon} \left[ \log \frac{p_{g_\varepsilon,D}(X_{0}, X_{D})}{p_{h_\varepsilon,D}(X_{0}, X_{D})} \right] $$
by Pinsker's inequality and the fact that the diffusion is stationary for the last equality. We plan to apply a slight modification of Theorem \ref{thm:variance} having $\varepsilon_N = N^{-1/2}$ for subsets $\mathcal C_N$ of the form $\|f-f_0\|_2 \leq \varepsilon_N$ and $\|f-f_0\|_\infty \rightarrow 0$, the rest being unchanged. This is simply done by revisiting Step 1 in the proof of Proposition \ref{thm: expectation proxy}. Indeed, it suffices to notice that 
$$\E_{f_0}\Big[\log \frac{f_0(X_{(i-1)D})}{f(X_{(i-1)D})}-\Big(\frac{f_0(X_{(i-1)D})}{f(X_{(i-1)D})}-1\Big)\Big] \lesssim \|f_0-f\|_2^2$$
for $\|f-f_0\|_\infty$ sufficiently small, 
as follows from $|\log \kappa-(\kappa-1)| \leq C(\kappa-1)^2$ in a neighbourhood of $\kappa=1$, together with the property $f \geq f_{\min}$. Here, we used the fact that since $X_0$ is distributed according to the invariant measure of the process, which is uniform over $\mathcal 0$, the process is stationary and $X_{(i-1)D}$ is uniformly distributed over $\mathcal O$ as well. We apply the result to $f_0=g_{\varepsilon}$ and $f=h_{\varepsilon}$ and check that under the sampling assumption  $D = N^{-a}$ with $1/2< a < 1$, we have $E_N \lesssim N^{-1}$. The result follows.

\subsection{Proof of Theorem \ref{thm: minimaxity final}} \label{sec: proof minimaxity} 
If $f \in \mathcal F_0$ and $\|f\|_{\mathcal C^s} \leq M$, we have $0 \leq f(x) \leq M$ for every $x \in \O$. By construction, we also have $0 \leq \widehat f_N^\star(x) \leq M$ for every $x \in \O$. It follows that
$$\E_f [\|\widehat{f}_N^\star-f\|_2^2] \leq  \E_f [\|\widehat{f}_N^\star-f\|_2^2{\bf 1}_{\mathcal{B}_N} ]+2M^2\mathbb P_f(\mathcal B_N^c),$$
where $\mathcal{B}_N$ is the event in \eqref{eq: def nice event}.
A glance at the proof of Lemma \ref{lem: compdisccont} shows that we in fact have the rate 
$$\sup_{f \in \mathcal F}\PP_{f}(\mathcal B_N^c) \lesssim e^{-2^{Jd}} \lesssim N^{-1},$$
hence the second term has a negligible order. Next, $|\widehat f_N^\star(x)-f(x)| = |\min(\widehat f_N(x), M)_+-f(x)| \leq |\widehat f_N(x)-f(x)|$ since $0 \leq f(x) \leq M$. Therefore,
$$\E_f [\|\widehat{f}_N^\star-f\|_2^2{\bf 1}_{\mathcal{B}_N} ] \leq  \E_f [\|\widehat{f}_N-f\|_2^2{\bf 1}_{\mathcal{B}_N} ] \lesssim \|f-\overline{P}_Jf\|^2_{L^2}+ \E_f [\|\widehat f_N-\overline{P}_Jf\|_{L^2}^2 {\bf 1}_{\mathcal{B}_N}] ,$$
where $\overline{P}_Jf$ denotes the projection \eqref{eq:PJ}. The first term is of order $2^{-2Js} \simeq N^{-2s/(2s+d)}$ by wavelet approximation since the Daubechies wavelet has at least $\lfloor s \rfloor-1$ vanishing moments, and thus has the right order. The second term also has the right order as a direct consequence of Theorem \ref{thm:exp_inequality}.

\subsection{Proof of Lemma \ref{lem: moment bis}} \label{sec: technical bounds}
Set $$Y_t = X_0+\int_0^t b(\nabla f(X_s), X_s)ds+\int_0^t \sqrt{2f(X_s)}dB_s.$$
Then $(X,\ell)$ is solution of the Skorokhod problem for $(\mathcal O, n, Y)$ since
$$X_t = Y_t + \int_0^t n(X_s)d\ell_s,$$
and the Skorokhod mapping: $\Gamma: Y \mapsto X = \Gamma Y$ is uniquely defined see {\it e.g.} Lions and Sznitman \cite{lions1984stochastic}. Moreover, we have 
$\Omega_\delta(\Gamma Y) \leq \Omega_\delta(Y),$
where 
$$\Omega_\delta(\psi) = \sup_{0 \leq s,t \leq T, |t-s| \leq \delta}|\psi(t)-\psi(s)|$$
denotes the modulus of continuity of $\psi:[0,T] \rightarrow \R^d$, 
so that it suffices to prove \eqref{eq: moment bound}  for $Y$ instead of $X$. 
We bound the first term by
\begin{align*}
\sup_{s \leq u \leq t}\Big|\int_{s}^{u}\nabla f(X_v)dv\Big|^{p}  \leq C_{\mathfrak b}^p(1+\|f\|_{\mathcal C^1})^p (t-s)^p
\end{align*}
uniformly over $b \in \mathcal B$. 
For the second term, we apply the Burkholder-Davis-Gundy inequality to obtain
\begin{align*}
\E_{f}\left[\sup_{s \leq u \leq t}\left|\int_s^{u}\sqrt{2f(X_v)}dB_v\right|^p\,\bigg|\,\mathcal F_s\right] & \leq c_\star^{p/2}d^{p/2}p^{p/2}\E_{f}\left[\left|\int_s^{t}f(X_u)du\right|^{p/2}\,\bigg|\,\mathcal F_s\right] \\
&\leq c_\star^{p/2}d^{p/2}p^{p/2}\|f\|_{\infty}^{p/2}(t-s)^{p/2}.
\end{align*}

\subsection{Proof of Lemma \ref{lem: hitting time}}\label{sec: hitting time proof}
Consider the process $(\widetilde X_t)_{t \geq 0}$ defined in \eqref{eq: diff whole space}. On $\mathcal A_{i,D}$, the two paths $(X_t-X_{(i-1)D})_{t \geq (i-1)D}$ and $(\widetilde X_t-X_{(i-1)D})_{t \geq (i-1)D}$ coincide for $t- (i-1)D  \leq \tau_{i,D}$. Using that $\mathsf{dist}(\mathcal O_0^{\delta}, \partial \mathcal O) \geq \delta/2$, it follows that
\begin{align*}
\PP_{f_0}(\tau_{i,D} \geq D, \mathcal A_{i,D}) & \leq \PP_{f_0}\Big(\sup_{0 \leq t \leq D} |\widetilde X_{(i-1)D+t}-X_{(i-1)D}|^2  > \tfrac{\delta}{2}, X_{(i-1)D} \in \mathcal O_0^\delta\Big) \\
& \leq \sup_{x \in \mathcal O_0^\delta} \PP_{f_0}\Big(\sup_{0 \leq t \leq D} |\widetilde X_t-x|^2  > \tfrac{\delta}{2}\,|\,X_0=x\Big),
\end{align*}
by the Markov property. By It\^o's formula, we further have 
\begin{align*}
\sup_{0 \leq t \leq D}|\widetilde X_t-x|^2 & \leq 2 \mathfrak b^2(1+\mathsf{diam}(\mathcal O)+ \|f_0\|_{\mathcal C^1})^2D^2+2\sup_{0 \leq t \leq D}\Big|\int_0^t \big(2f_0(\widetilde X_s)\big)^{1/2}dB_s\Big|^2,
\end{align*}
and hence
\begin{align*}
\PP_{f_0}\Big(\sup_{0 \leq t \leq D} |\widetilde X_t-x|^2 \geq  \tfrac{1}{4}\delta^2\,|\,X_0=x\Big) & \leq \PP_{f_0}\Big(\sup_{0 \leq t \leq D} |\mathcal M_{t}(g)|^2 \geq  \tfrac{1}{4}\delta^2- C_{f_0} D^2\,|\,X_0=x\Big) \\
&\leq \PP_{f_0}\Big(\sup_{0 \leq t \leq D} |\mathcal M_{t}(f_0)|^2 \geq \tfrac{1}{5}\delta^2\,|\,X_0=x\Big)
\end{align*}
for small enough $D$, where $\mathcal M_{t}(f_0) = \int_0^t (2f_0(\widetilde X_s))^{1/2}dB_s$ is a $d$-dimensional martingale with predictable bracket $\langle \mathcal M^i{f_0}, \mathcal M^j(f_0)\rangle_t = 2\int_0^tf_0(\widetilde X_s)ds \delta_{ij}\leq 2\|f_0\|_\infty t\,\delta_{ij}$. It follows that for any $x \in \O_0^{\delta}$:
\begin{align*}
\PP_{f_0}\Big(\sup_{0 \leq t \leq D} |\widetilde X_t-x|^2 \geq  \tfrac{1}{4}\delta^2\,|\,X_0=x\Big) & \leq \sum_{i = 1}^d \PP_{f_0}\big(\sup_{0 \leq t \leq D}|\mathcal M_{t}^i(f_0)| \geq \tfrac{1}{\sqrt{5 d}}\delta\,|\,X_0=x\big) \nonumber \\
& \leq 2d\exp\Big(-\frac{\delta^2}{20d\|f_0\|_\infty D}\Big), \label{eq:hitting_exponential}
\end{align*}
where we have used Bernstein's inequality for continuous local martingale $(\mathcal M_t)_{t \geq 0}$: 
$$\PP\Big(\sup_{0 \leq s \leq t} \mathcal M_s \geq x, \langle \mathcal M\rangle_t \leq y \Big) \leq e^{-x^2/(2y)}$$
(see e.g. \cite{RY99} p.154).

\subsection{A generic chaining inequality and the event $\mathcal B_N$} \label{sec: generic chaining}

In several places, we require the following concentration inequality, which is based on a chaining argument for stochastic processes with \textit{mixed tails}, see Theorem 2.2.28 in \cite{T14} or Theorem 3.5 in \cite{D15}.

\begin{lem}\label{lem:Dirksen}
Let $S\subset L^2(\mathcal O)$ be a finite-dimensional linear space with dimension $\dim(S) =m<\infty$. Let $Z=(Z(g):g\in \mG)$ be a stochastic process with index set $\mG \subset S$ satisfying $\|\mG\|_2 = \sup_{g,h\in \mG} \|g-h\|_{2}<\infty$. Suppose that $Z$ satisfies for all $g,h \in \mG$ and $t\geq 0$,
\begin{equation*}
\PP \left( |Z(g) - Z(h)| \geq t \right) \leq C \exp \left( - \frac{ct^2}{\alpha^2 \|g-h\|_{2}^2 + \beta \|g-h\|_{2} t}\right),
\end{equation*}
where $C,c,\alpha,\beta>0$. Then for any $g_0 \in \mG$ and all $u \geq 1$,
\begin{equation*}
\PP \left( \sup_{g \in \mG} |Z(g)-Z(g_0)| \geq K \big( m \beta + \sqrt{m} \alpha + \beta u + \alpha \sqrt{u} \big) \|\mG\|_2 \right) \leq e^{-u},
\end{equation*}
where $K$ depends only on $C$ and $c$.
\end{lem}

\begin{proof}
Let $(e_k:1\leq k \leq m)$ be an $L^2$-orthonormal basis for $S$. After rearranging, the Bernstein inequality implies that for $C_1,C_2>0$ large enough, depending only on $C$ and $c$,
$$\PP \left( |Z(g) - Z(h)| \geq C_1\beta u \|g-h\|_{2} + C_2 \alpha \sqrt{u} \|g-h\|_{2} \right) \leq 2e^{-u}$$
holds for all $u\geq 0$. The process $Z$ therefore has a \textit{mixed tail} as in (3.8) of Dirksen \cite{D15} with respect to the metrics
\begin{align*}
d_1(g,h) &= C_1\beta \|g-h\|_{2} = C_1 \beta |\langle g-h, e_k \rangle_{2}|_{\R^m},\\
d_2(g,h) &= C_2 \alpha \|g-h\|_{2} = C_2 \alpha  |\langle g-h, e_k \rangle_{2} |_{\R^m},
\end{align*}
where the last equalities follow from Parseval's theorem and $|\cdot|_{\R^m}$ is the Euclidean norm on $\R^m$. The $d_1$- and $d_2$-diameters of $\mG$ are therefore bounded by $\Delta_{d_1}(\mG) \leq C_1 \beta \|\mG\|_2$ and $\Delta_{d_2}(\mG) \leq C_2 \alpha\|\mG\|_2$, respectively. Theorem 3.5 of \cite{D15} thus yields that for absolute constants $C',c'>0$ and any $u\geq 1$,
\begin{equation*}
\PP \left( \sup_{g \in \mG} |Z(g)-Z(g_0)| \geq C'\gamma_1(\mG,d_1) + C'\gamma_2(\mG,d_2) +c'C_1\beta\|\mG\|_2u + c' C_2 \alpha \|\mG\|_2 \sqrt{u}  \right) \leq e^{-u},
\end{equation*}
where $\gamma_1,\gamma_2$ are the `generic chaining functionals' defined in \cite{D15}. In particular, (2.3) of \cite{D15} gives the estimate $\gamma_\alpha(\mG,d_i) \leq C(\alpha) \int_0^\infty (\log \overline{N}(\mG,d_i,\eps))^{1/\alpha} d\eps$, where $\overline{N}(\mG,d_i,\eps)$ denotes the covering number of the set $\mG$, i.e. the smallest number of $d_i$-balls of radius $\eps$ needed to cover $\mG$. Since $\mG\subset S$ and $S$ is a finite-dimensional linear space, the second-last display implies that
$$\overline{N}(\mG,d_1,\eps) \leq \overline{N}(\{ x \in \R^m : |x|_{\R^m} \leq \|\mG\|_2 \}, C_1\beta  |\cdot|_{\R^m}, \eps ),$$
and likewise for $d_2$. Using for instance Proposition 4.3.34 in the textbook \cite{GN16} to bound the log-covering number of a ball in finite-dimensional Euclidean space, 
\begin{align*}
\gamma_1(\mG, d_1) & \lesssim \int_0^\infty \log \overline{N}(\{ x \in \R^m : |x|_{\R^m} \leq \|\mG\|_2 \}, C_1\beta  |\cdot|_{\R^m}, \eps )  d\eps \\
& \leq \int_0^{\Delta_{d_1}(\mG)} m \log \left( \frac{3\Delta_{d_1}(\mG)}{\eps} \right)  d\eps = m \Delta_{d_1}(\mG)  \int_0^{1} \log (3/u)  du  \lesssim m C_1 \beta \|\mG\|_2.
\end{align*}
Similarly,
\begin{align*}
\gamma_2(\mG, d_2) & \lesssim \int_0^{\Delta_{d_2}(\mG)} \sqrt{ m\log \left( \frac{3\Delta_{d_2}(\mG)}{\eps} \right)  } d\eps  = \sqrt{m} \Delta_{d_2}(\mG)  \int_0^{1} \sqrt{\log (3/u)}  du \lesssim \sqrt{m} C_2 \alpha \|\mG\|_2.
\end{align*}
Substituting these bounds into the exponential inequality derived above then gives the result.
\end{proof}


\subsection{Proof of Lemma \ref{lem: compdisccont}} \label{sec: proof of good event}

We have
\begin{equation}\label{eq:bilinear_prob}
\begin{split}
\PP_f(\mathcal B_N^c) & \leq  \PP_f\big(\exists g \in V_J: (1-\kappa) \|g\|_{2}^2 > |g|_N^2 \big) +\PP_f\big(\exists g \in V_J: |g|_N^2 > (1+\kappa) \|g\|_{2}^2  \big) \\
& = \PP_f\left(\exists g \in V_J: \|g\|_2 = 1,\;  1-\kappa > |g|_N^2 \right) +\PP_f\left(\exists g \in V_J: \|g\|_2 = 1,\;  |g|_N^2>1+\kappa \right) \\
& \leq 2 \PP_f\left(\sup_{g \in V_J:\|g\|_2 = 1} \left| |g|_N^2-1\right| >  \kappa \right)\\
&= 2 \PP_f \left(\sup_{g \in V_J:\|g\|_2 =1} \left| \sum_{i = 1}^N\big(g(X_{iD})^2-\E_f[g(X_{iD})^2]\big) \right| >  \kappa N\right)
\end{split}
\end{equation}
since $\E_f[g(X_{iD})^2]=\int_{\O}g(x)^2 dx =1$ by stationarity. Write $\mG_J =  \{ g\in V_J: \|g\|_2 = 1\}$ and define the process $Z_N = (Z_N(g): g \in \mG_J)$ by
$$Z_N(g)= \sum_{i = 1}^N\big(g(X_{iD})^2-\E_f[g(X_{iD})^2]\big).$$
For $g,h \in \mG_J$, we have $\E_f [g(X_{iD})^2 -  h(X_{iD})^2] = 0$ and $\var_f(g(X_{iD})^2 - h(X_{iD})^2) = \|g^2-h^2\|_{2}^2$. Since $(X_0,X_D,\dots,X_{ND})$ is a stationary reversible Markov chain whose spectral gap is lower bounded by $rD$ by \eqref{eq:spectral_gap}, Theorem 3.3 of \cite{Paulin2015} (cf. (3.21)) yields 
yields
\begin{align*}
\PP_f\big( |Z_N(g)-Z_N(h)| \geq t\big) & =  \PP_f\left( \left| \sum_{i = 1}^N  g(X_{iD})^2-h(X_{iD})^2 \right| \geq t\right) \\
& \leq 2\exp\left(-\frac{ \lambda_{f} D t^2 }{4N\|g^2-h^2\|_{2}^2+10\|g^2-h^2\|_\infty t}\right) \\
& \leq 2\exp\left(-\frac{ \lambda_{f} D t^2 }{4N \|g+h\|_\infty^2 \|g-h\|_{2}^2+10\|g+h\|_\infty \|g-h\|_\infty t}\right) \\
& \leq 2\exp\left(-\frac{ C  t^2 }{N 2^{Jd} D^{-1} \|g-h\|_{2}^2+ 2^{Jd} D^{-1} \|g-h\|_{2} t}\right)
\end{align*}
using that $\|g\|_\infty \leq C 2^{Jd/2} \|g\|_{2} = C2^{Jd/2}$ for $g \in \mG_J$, and where $C$ depends on $r$, and hence $f_{\min}$ and $\O$ by \eqref{eq:spectral_gap}. Applying Lemma \ref{lem:Dirksen} with $m = \dim(V_J) = O(2^{Jd})$, $\|\mG_J\|_2  = \sup_{g,h\in \mG} \|g-h\|_{2}\leq 2$, $\alpha^2 = N2^{Jd}D^{-1}$, $\beta=2^{Jd}D^{-1}$ and $u=2^{Jd}$ thus gives
\begin{equation*}
\PP_f \left( \sup_{g \in \mG_J} |Z_N(g)| \geq C \left(  2^{2Jd} D^{-1} + 2^{Jd} N^{1/2}D^{-1/2} \right) \right) \leq e^{-2^{Jd}} \to 0
\end{equation*}
since $2^J \to \infty$ as $N\to\infty$. Since $2^{Jd} =$ by assumption, the right-hand side within the last probability is $o(N)$ as $N \to \infty$. Together with \eqref{eq:bilinear_prob} this gives the result.

\bigskip

\noindent \textbf{Acknowledgements:} We thank two reviewers for their many helpful comments which significantly improved the manuscript. We are also grateful to Richard Nickl for insightful discussions that motivated the genesis of this project.

\noindent \textbf{Conflict of interest:} The authors declare that they have no conflict of interest.

\noindent \textbf{Data availability:} Not applicable as no datasets were generated or analysed in this work.

\bibliography{HF_diffusion_refs}{}
\bibliographystyle{acm}

\end{document}